\documentclass[10pt]{amsart}
\pdfoutput=1
\usepackage{amsmath, amsthm, amssymb,slashed,stmaryrd}
\usepackage{tikz}
\usetikzlibrary{matrix,arrows,decorations.pathmorphing,backgrounds,patterns,shapes.geometric,calc}

\usepackage[pdftex,plainpages=false,hypertexnames=false,pdfpagelabels]{hyperref}
 \theoremstyle{plain}
 \newtheorem{thm}{Theorem}[section]

 \newtheorem{cor}[thm]{Corollary}
 \newtheorem{lem}[thm]{Lemma}
  
 \newtheorem{prop}[thm]{Proposition}
 \theoremstyle{definition}
 \newtheorem{defn}[thm]{Definition}

 \newtheorem{ex}[thm]{Example}
 \theoremstyle{remark}
 \newtheorem{rmk}[thm]{Remark}

\def\beq{\begin{eqnarray}}
\def\eeq{\end{eqnarray}}

\DeclareSymbolFont{bbold}{U}{bbold}{m}{n}
\DeclareSymbolFontAlphabet{\mathbbold}{bbold}

 \newcommand{\bp}{\begin{proof}[Proof]}
 \newcommand{\ep}{\end{proof}}
\DeclareMathOperator{\SM}{\underline{\sf SMfld}}

\def\RG{{\rm RG}}
\def\nd{{\rm nd}}

\def\pt{\rm pt}
\def\MF{{\rm MF}}

\def\H{{{\mathbb{H}}}}
\def\MString{{\rm MString}}
\def\KMF{{\rm K}_{\rm MF}}

\def\SL{{\rm SL}}

\def\CS{\rm CS}

\def\ev{{\rm ev}}
\def\odd{{\rm odd}}
\def\cl{{\rm cl}}

\newcommand{\sq}{\mathord{/\!\!/}}

\DeclareMathOperator{\Cl}{{{\C{\rm l}}}}
\DeclareMathOperator{\Fer}{{{\rm Fer}}}
\DeclareMathOperator{\Rep}{{{\rm Rep}}}

\def\R{{\mathbb{R}}}
\def\Rge{{\R{}^{1|1}_{\ge 0}}}
\def\A{{\mathbb{A}}}
\def\F{{\mathcal{F}}}
\def\N{{\mathbb{N}}}
\def\id{{{\rm id}}}

\def\im{{\rm{im}}}
\def\vol{{\rm{vol}}}
\def\Tate{{\rm{Tate}}}

\def\K{{\rm {K}}}
\def\C{{\mathbb{C}}}
\def\CC{{\sf{C}}}
\def\DD{{\sf{D}}}

\def\Z{{\mathbb{Z}}}
\def\E{{\mathbb{E}}} 

\def\Fun{{\sf Fun}}

\def\Vect{{\sf Vect}} 
\def\pt{{\rm pt}}
\def\hs{\H}
\def\End{{\sf End}}

\def\Hom{\mathop{\sf Hom}}

\def\TMF{{\rm TMF}}

\def\SM{ {\underline{\sf SMfld}}}
\def\EFT{ \hbox{-{\sf EFT}}}

\def\Path{ {{\sf sP}_0}}
\def\path{ {{\mathfrak{p}}_0}}
\def\Ann{ {{\sf sAnn}_0}}
\def\Rot{ {{\sf Rot}}}
\def\sRot{ {{\sf sRot}}}

\def\Spin{ \hbox{\rm Spin}}
\def\GL{ \hbox{\rm GL}}
\def\EBord{ \hbox{-{\sf EBord}}}

\newcommand{\toto}{\rightrightarrows}

\def\twocommute{\ensuremath{\rotatebox[origin=c]{30}{$\Rightarrow$}}}

\newcommand{\op}{{\sf{op}}}   

\def\twocommute{\ensuremath{\rotatebox[origin=c]{30}{$\Rightarrow$}}}

\begin{document}

\title[{Topological $q$-expansion and the supersymmetric sigma model}]{Topological $q$-expansion\\ and the supersymmetric sigma model}

\def\obfuscate#1#2{\rlap{\hphantom{#2}@#1}#2\hphantom{@#1}}
\author{Daniel Berwick-Evans}
\address{Department of Mathematics, University of Illinois at Urbana-Champaign}
\email{\obfuscate{illinois.edu}{danbe}}
\begin{abstract} 
The Hamiltonian and Lagrangian formalisms offer two perspectives on quantum field theory. This paper sets up a framework to compare these approaches for the supersymmetric sigma model. The goal is to use techniques from physics to construct topological invariants. In brief, the Hamiltonian formalism studies positive energy representations of super annuli. This leads to a model for elliptic cohomology at the Tate curve over $\Z$. The Lagrangian approach studies sections of line bundles over a moduli stack of super tori. This leads to a model for ordinary cohomology valued in weak modular forms over $\C$. Compatibility between the two formalisms is a field theory version of the topological $q$-expansion principle. Combining these ingredients constructs a cohomology theory admitting an orientation for string manifolds that is closely related to Witten's Dirac operator on loop space. 
\end{abstract}

\maketitle
\setcounter{tocdepth}{1}
\tableofcontents

\section{Introduction}

Since Witten and Segal's groundbreaking papers~\cite{Witten_Elliptic,Witten_Dirac,Segal_Elliptic} there has been a tantalizing yet elusive connection between elliptic cohomology and 2-dimensional field theories. Just as the Dirac operator connects K-theory to supersymmetric quantum mechanics, the dream has been that a suitable geometric or analytic object would connect elliptic cohomology to 2-dimensional supersymmetric sigma models. This paper makes progress by providing a field-theoretic counterpart to Laures' topological $q$-expansion principle~\cite{LauresPhD}. The main construction is a differential cocycle model for elliptic cohomology at the Tate curve based on positive energy representations of a category of super Euclidean annuli. This is designed to be compatible with the differential cocycle model for $\TMF\otimes \C$ in~\cite{DBE_WG}. Indeed, constructions from physics (e.g., a cutoff version of the supersymmetric sigma model) produce differential cocycles in both theories that are suitably compatible. This corresponds to the compatibility between the Lagrangian and Hamiltonian points of view on 2-dimensional supersymmetric field theories. 

Throughout, $M$ is a smooth, compact manifold without boundary. The topological $q$-expansion principle is summarized by the commuting diagram
\begin{equation}
\begin{array}{c}
\begin{tikzpicture}[node distance=3.5cm,auto]
  \node (A) {${\rm TMF}(M)$};
  \node (B) [node distance= 4.5cm, right of=A] {$ {\rm K}_{\rm Tate}(M)$};
  \node (C) [node distance = 1.5cm, below of=A] {${\rm H}_\MF(M)$};
  \node (D) [node distance = 1.5cm, below of=B] {${\rm H}_\C(M)\llbracket q \rrbracket [q^{-1}].$};
  \draw[->] (A) to node {${\rm ev}_{\rm Tate}$} (B);
  \draw[->] (A) to node [swap] {$\otimes \C$} (C);
  \draw[->] (C) to node {$q{\rm -expand}$} (D);
  \draw[->] (B) to node {${\rm Ch}$} (D);
\end{tikzpicture}\end{array}\label{square1}
\end{equation}
The top horizontal map is Miller's elliptic character~\cite{Miller} that evaluates the universal elliptic cohomology theory of topological modular forms (TMF) in a formal punctured neighborhood of the Tate curve. The left vertical map is the Chern--Dold character of TMF from tensoring over~$\Z$ with~$\C$, followed by the identification $\TMF\otimes \C\cong {\rm H}_\MF$ with the ordinary cohomology theory with coefficients in weak modular forms over~$\C$. We can also identify elliptic cohomology at the Tate curve with ordinary K-theory with coefficients in powers of~$q$, ${\rm K}_{\rm Tate}(M)\cong {\rm K}(M)\llbracket q \rrbracket [q^{-1}]$. Then the right vertical arrow is induced by the Chern character in ${\rm K}$-theory, and the lower horizontal arrow is determined by the $q$-expansion of weak modular forms. 

This paper gives a field-theoretic counterpart to~\eqref{square1}, with differential cocycle models for all the cohomology theories with the exception of~$\TMF$. The players are (i) a category of positive energy representations of constant super annuli in $M$, denoted $\Rep(\Ann(M))$, (ii) a stack $\widetilde{\mathcal{L}}^{2|1}_0(M)$ of constant super annuli with the same source and target super circle in $M$; and (iii) a super double loop stack stack $\mathcal{L}^{2|1}_0(M)$ of constant super tori in~$M$. Then the analog of~\eqref{square2} is
\begin{equation}
\begin{array}{c}
\begin{tikzpicture}[node distance=3.5cm,auto]
  \node (A) {$\left\{\begin{array}{c} 2|1{\rm -dimensional \ field} \\ {\rm theories\ over\ } M {\rm ??}\end{array}\right\} $};
  \node (B) [node distance= 7cm, right of=A] {$\Rep(\Ann(M))$};
  \node (C) [node distance = 2.25cm, below of=A] {$\mathcal{O}(\mathcal{L}^{2|1}_0(M))$};
  \node (D) [node distance = 2.25cm, below of=B] {$\widehat{\mathcal{O}}(\widetilde{\mathcal{L}}^{2|1}_0(M)).$};
  \draw[->,dashed] (A) to node {${\rm time\ evolution}$} (B);
  \draw[->,dashed] (A) to node [swap] {$\begin{array}{l} {\rm partition} \\ {\rm function}\end{array}$} (C);
  \draw[->] (C) to node {${\rm forget}$} (D);
  \draw[->] (B) to node {${\rm character}$} (D);
\end{tikzpicture}\end{array}\label{square2}
\end{equation}
The appropriate definition of a $2|1$-dimensional field theory is still under active investigation, and we take the preliminary definition of Stolz and Teichner~\cite{ST11} as a guide; see~\S\ref{sec:ST} below. In brief, the top dashed arrow is the value of a field theory on a category of super annuli (the time-evolution operator), and the left dashed arrow is the value of a field theory on super tori (the partition function). The right vertical arrow is a character map for positive energy representations of super annuli. The lower horizontal arrow is induced by the functor $\widetilde{\mathcal{L}}^{2|1}_0(M)\to \mathcal{L}^{2|1}_0(M)$ that views a super annulus with the same source and target super circle as a super torus in $M$. This connects with the diagram~\eqref{square1} as follows: (i) the Grothendieck group of $\Rep(\Ann(M))$ is ${\rm K}_{\rm Tate}(M)$, (ii) elements of $\mathcal{O}(\mathcal{L}^{2|1}_0(M))$ define cocycles for ${\rm H}_\MF(M)$ and (iii) elements of $\widehat{\mathcal{O}}(\widetilde{\mathcal{L}}^{2|1}_0(M))$ define cocycles for~${\rm H}(M;\C)\llbracket q \rrbracket [q^{-1}]$. 

Connections between 2-dimensional field theories, representations of annuli, and elliptic cohomology at the Tate curve have been understood in varying degrees for awhile. Early on, Segal~\cite{SegalCFT,Segal_Elliptic} described a relationship between representations of annuli and Witten's construction of the Witten genus. Using related ideas, Stolz and Teichner sketched a map from their proposed elliptic objects to elliptic cohomology at the Tate curve~\cite[Theorem~1.0.2]{ST04}, though a complete definition of these elliptic objects hasn't yet been worked out. In his thesis, Pokman Cheung~\cite{PokmanPhD} constructed a space of annular supersymmetric field theories whose homotopy type is a representing space for~$\K_\Tate$. Our construction draws on the insights of these previous authors, particularly from Stolz and Teichner. We describe the connection between our framework with the Stolz--Teichner program in~\S\ref{sec:ST}. 

The main new feature below is a careful treatment of the character theory (or \emph{partition functions}) for smooth families of representations of super annuli. This builds on ideas of Fei Han~\cite{Han}, who related the partition function of $1|1$-Euclidean field theories with the Chern character of a vector bundle with connection. 
In the $2|1$-dimensional setting, this type of geometry gives an evident relationship between the Chern character of $\K_\Tate(M)$ and cocycles in~${\rm H}_\MF(M)$: both define functions on closely related moduli stacks of super tori over~$M$. Furthermore, suitably compatible cocycles in $\K_\Tate(M)$ and ${\rm H}_\MF(M)$ in~\eqref{square2} allow one to deduce integrality and modularity properties of the relevant topological invariants. This flavor of argument runs in complete parallel to a standard one in physics. To give a sketch, the integrality of the Witten genus can be seen by viewing it as a class in~$\K_\Tate(\pt)$ (which corresponds to the Hamiltonian perspective on the supersymmetric sigma model), and modularity follows from viewing it as an element of~${\rm H}_\MF(\pt)$ (which comes from the Lagrangian perspective). Playing these two points of view on quantum theory off each other is an age-old tool in physics, with considerable mathematical depth. For example, the physics proof of the Atiyah--Singer index theorem~\cite{Alvarez} is the assertion that the partition function in $1|1$-dimensional quantum mechanics is the same, whether one computes in either the Hamiltonian or the Lagrangian framework. We explore this further in~\S\ref{sec:physmot}

When studied in families, there is a rub in the $2|1$-dimensional case: partition functions from the Hamiltonian and Lagrangian perspectives need not be equal on the nose. In terms of the topology, for a family of string manifolds $\pi\colon X\to M$, we get classes~$[\sigma(X)]\in \K_\Tate(M)$ and $[\sigma_{\rm H}(X)]\in {\rm H}_\MF(M)$ from the Ando--Hopkins--Rezk string orientation of TMF~\cite{AHS,AHR} postcomposed with the maps ${\rm TMF}\to \K_\Tate$ and ${\rm TMF}\to {\rm H}_\MF$, respectively. In our geometric context, these classes can be refined to cocycles which have a field theoretic interpretation as in~\eqref{square2}. In this model, typically the Chern character of~$\sigma_\K(X)$ does not equal the $q$-expansion of~$\sigma_{\rm H}(X)$, e.g., as a differential form with values in $\C\llbracket q \rrbracket [q^{-1}]$. However, a choice of string structure on $\pi\colon X\to M$ specifies a smooth homotopy (or concordance) between these two cocycles. This is an example of an anomaly in physics, and the choice of string structure (which trivializes the anomaly) has homotopical meaning that can be understood in terms of the diagram~\eqref{square1}: it is the homotopy that witnesses the string orientation as a map in the homotopy pullback. We explain the physical picture in greater depth in~\S\ref{sec:physmot}.

These geometric ideas lead to a differential cocycle model for a cohomology theory denoted $\KMF$ defined by the homotopy pullback in spectra,
\begin{equation}
\begin{array}{c}
\begin{tikzpicture}[node distance=3.5cm,auto]
  \node (A) {$\KMF$};
  \node (B) [node distance= 4.5cm, right of=A] {$ {\rm K}_{\rm Tate}$};
  \node (C) [node distance = 1.5cm, below of=A] {${\rm H}_\MF$};
  \node (D) [node distance = 1.5cm, below of=B] {${\rm H}_\C\llbracket q \rrbracket [q^{-1}].$};
  \draw[->] (A) to (B);
  \draw[->] (A) to (C);
  \draw[->] (C) to (D);
  \draw[->] (B) to (D);
\end{tikzpicture}\end{array}\nonumber
\end{equation}
By the universal property, there is a map $\TMF\to \KMF$, but it is easy to verify that this is not an equivalence: the coefficients are different, with
$$
\KMF^\bullet(\pt)=\left\{ \begin{array}{ll} \MF_\Z^{2n} & \bullet=2n \\ \C\llbracket q \rrbracket [q^{-1}]/\Z\llbracket q \rrbracket [q^{-1}]+\MF^{2n}&\bullet=2n-1, \end{array}\right.
$$ 
i.e., in even degrees~$2n$ the coefficients are integral modular forms of weight $-n$ and in odd degrees they are quotients of $\C\llbracket q \rrbracket [q^{-1}]$ by the image of $\Z\llbracket q \rrbracket [q^{-1}]$ and (complex) weak modular forms under the $q$-expansion map (see~\S\ref{appen:MF} for our conventions regarding the grading on~$\MF$). The odd coefficients are well-known  receptacles for torsion invariants, e.g., Laures' $f$-invariant~\cite{LauresPhD} and Bunke and Naumann's secondary invariants of string manifolds~\cite{BunkeNaumann}. In~\S\ref{sec:examples} we sketch how the Bunke--Naumann invariants are primary invariants in $\KMF$, coming from the composition $\MString\to\TMF\to\KMF$ where the first map is the string orientation constructed by Ando--Hopkins--Strickland--Rezk~\cite{AHS,AHR}. The string orientation of~$\KMF$ can be thought of as a refinement of the string orientation of~$\K_\Tate$, as constructed by Witten~\cite{Witten_Dirac}. Notably, in this refinement to~$\KMF$, the underlying (Witten) genus is automatically an integral modular form. 

The picture from~\eqref{square2} gives a physical interpretation for invariants coming from ${\rm MString}\to \KMF$, as we explain in~\S\ref{sec:examples}. This both explains the modularity of the Witten genus within a mathematical framework, and gives a first indication of how the torsion in~TMF might be related to 2-dimensional quantum field theories. However, as it stands the physical interpretation doesn't have a great deal of mathematical content; instead its purpose is to begin development of a dictionary that connects the homotopical ideas to physical ones. The end goal is to continue to use the geometry of supersymmetric sigma models to refine the homotopical invariants further. Specifically, a robust understanding of the supersymmetric sigma model might give an analytic construction of the string orientation of~$\TMF$. 

As a warm-up example to this story, we construct differential K-theory from representations of constant super paths in a manifold. This basically constructs Klonoff's model~\cite{Klonoff} where cocycles are super vector bundles with super connection and an odd degree differential form. This easier case both highlights the formal similarities with the construction of~$\K_\Tate$ and allows us to translate operations on path categories into differential geometry. For example, we show that dilating super paths implements the Bismut--Quillen rescaling on super connections. In physics terminology, this dilation is the \emph{renormalization group flow}. This also makes sense for super annuli, giving a candidate generalization of Bismut--Quillen rescaling in the 2-dimensional case where the geometry is less familiar.

\subsection{Results I: Super annuli and elliptic cohomology at the Tate curve}

The main objects of study in this paper are super Euclidean annuli with maps to $M$
\begin{equation}
\begin{array}{c}
\begin{tikzpicture}[node distance=3.5cm,auto]
  \node (A) {$S^{1|1}_r$};
  \node (B) [node distance= 1.5cm, below of=A] {$S^{1|1}_r$};
  \node (C) [node distance= 3cm, right of=A] {$\null$};
  \node (D) [node distance = .75cm, below of=C] {$A^{2|1}_r$};
  \node (E) [node distance = 2cm, right of=D] {$M$};
  \draw[->,right hook-latex] (A) to node {in} (D);
  \draw[->,right hook-latex] (B) to node [swap] {out} (D);
  \draw[->] (D) to node {$\phi$} (E);
\end{tikzpicture}\end{array}\label{eq:annulus}
\end{equation}
where $S^{1|1}_r=\R^{1|1}/r\Z$ are incoming and outgoing super circles of circumference $r\in \R_{>0}$, and $A^{2|1}_r=\R^{2|1}/r\Z$ is an (infinite) super annulus with circumference $r$; see Figure~\ref{fig1}. By restricting to those super annuli whose maps to $M$ are invariant under the rotation action by the underlying (ordinary) annulus $\R^2/r\Z\subset \R^{2|1}/r\Z$ one can assemble the data~\eqref{eq:annulus} into the morphisms in a super Lie category denoted $\Ann(M)$ whose objects and morphisms form finite-dimensional super manifolds. Such categories have a good notion of a smooth representation (see~\S\ref{sec:Lierep}), with the data being a vector bundle over the objects and maps between vector bundles over morphisms. Motivated by unitary quantum field theories, an orientation-reversing map on annuli naturally leads to a version of \emph{unitary} representations of $\Ann(M)$. 

\begin{figure}\label{fig1}
\begin{center}
\begin{tikzpicture}[scale=1]

\node [draw, cylinder, shape aspect=4, minimum height=8cm, minimum 
width=2cm,dashed] (a) {};

\draw[thick] (-3.25,1) to (3.85,1);

\draw[thick] (-3.25,-1) to (3.85,-1);

\node [draw, thick, ellipse, minimum width=2cm,minimum height=.9cm, rotate=90] (b) at (-1,0){};

\node [draw, thick, ellipse, minimum width=2cm,minimum height=.9cm, rotate=90] (c) at (2,0){};

\node [draw, thick, ellipse, minimum width=2cm,minimum height=.9cm, rotate=90] (d) at (-1.5,-3){};

\node [draw, thick, ellipse, minimum width=2cm,minimum height=.9cm, rotate=90] (e) at (2.5,-3){};
\draw[->,left hook-latex] (2.5,-1.8) to node [right=.1in] {${\rm out}$} (2.2, -1.3);
\draw[->,right hook-latex] (-1.5,-1.8) to node [left=.1in] {${\rm in}$} (-1.2, -1.3);
\node (A) [node distance = 1cm, left of =d] {$S^{1|1}_r$};
\node (B) [node distance = 1cm, right of =e] {$S^{1|1}_r$};
\node (C) [node distance = 3.5cm, left of =b] {$A^{2|1}_r$};
\node (D) [node distance =7cm, right of=b] {$M$};
\draw[->] (5,0) to node [above] {$\phi$} (D);
\end{tikzpicture}
\end{center}
\caption{A rough picture of the infinite super annulus $A^{2|1}_r$ with a pair of embedded super circles~$S^{1|1}_r$ and a map to $M$.}
\end{figure}
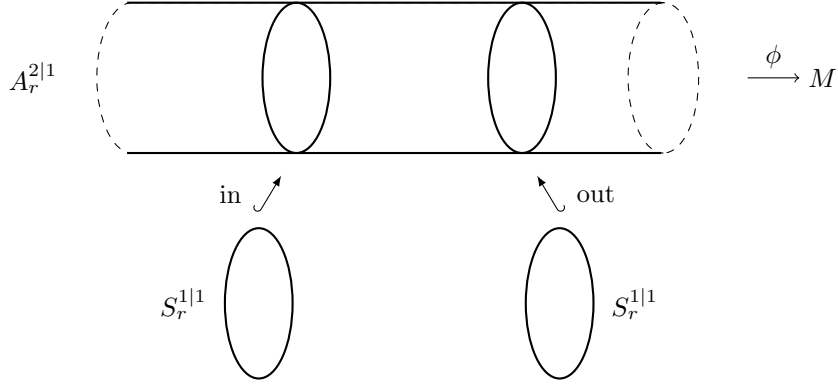

As is familiar in the case of loop group representations, it is important to restrict attention to positive energy representations of $\Ann(M)$. The definition relies on a subgroupoid $\Rot(M)\subset \Ann(M)$ of degenerate annuli that act on super circles by rotation. The action by this subgroupoid decomposes any unitary representation into a direct sum of weight spaces. A unitary representation then has \emph{positive energy} if the weight spaces of the $\Rot(M)$-action are finite-dimensional with weight bounded below. Let $\Rep(\Ann(M))$ denote this category of positive energy representations. 

\begin{thm} \label{thm:Tateeasy}
The Grothendieck group of $\Rep(\Ann(M))$ is $\K_\Tate(M)$, the elliptic cohomology at the Tate curve of~$M$. \end{thm}

\begin{rmk} The quotient defining the Grothendieck group of $\Rep(\Ann(M))$ can be identified with representations that tend to zero under the renormalization group flow; see~\S\ref{sec:Kthymot}. \end{rmk}

\begin{rmk} Viewing the Tate curve as an infinite annulus, the category~$\Ann(M)$ comes from cutting (a super-version) of this curve into pieces and mapping these pieces to~$M$. \end{rmk}

To promote this construction to a \emph{differential} cocycle model, we require a differential form-valued Chern character. The character of a representation of $\Ann(M)$ is a function on super annuli over $M$ whose source and target super circles coincide, i.e., super annuli that determine super tori in~$M$. A rescaling of this character (akin to the rescaled super connections of Quillen~\cite{Quillensuper} and Bismut~\cite{Bismutindex}) yields a function that is (1) invariant under global dilations of the super annulus, (2) invariant under the super translation action of the associated torus on itself, and (3) depends holomorphically on the modulus $q=e^{2\pi i\tau}$ of the super annulus. For finite-dimensional representations, this gives a \emph{rescaled partition function}
$$
Z\colon \Rep_{\rm fd}(\Ann(M))\to \Omega^\ev_\cl(M)\otimes \C[q,q^{-1}]\hookrightarrow \mathcal{O}(\widetilde{\mathcal{L}}^{2|1}_0(M))
$$
where $\widetilde{\mathcal{L}}^{2|1}_0(M)$ is a stack whose objects are super tori in $M$ with a choice of embedded super circle, and whose morphisms are global dilations and super translations of these tori. Characters of arbitrary positive energy representations need not define a function on $\widetilde{\mathcal{L}}^{2|1}_0(M)$, but instead define a formal sum of such functions, giving a \emph{formal rescaled partition function}
\beq
Z\colon \Rep(\Ann(M))\to \Omega^\ev_\cl(M)\otimes \C\llbracket q \rrbracket [q^{-1}]=\widehat{\mathcal{O}}(\widetilde{\mathcal{L}}^{2|1}_0(M)). 
\label{eq:charintro}
\eeq
This furnishes the required differential form-valued Chern character to define a differential cohomology theory. In~\S\ref{sec:groth} we give a general construction of \emph{differential Grothendieck groups} of the representation category of a super Lie category equipped with a specified character map, which in this case is~$Z$. 

\begin{thm}\label{thm:Tate}
The differential Grothendieck group of $\Rep(\Ann(M))$ with respect to the character map~\eqref{eq:charintro} is the differential elliptic cohomology of $M$ at the Tate curve. 
\end{thm}

\begin{rmk} A differential cocycle is a positive energy representation with an extra datum: a smooth homotopy (or better, concordance) of its rescaled partition function. Physically, this can be interpreted as a correction to the partition function from ``higher energy" modes that have been integrated out. In examples from geometry this correction is constructed by a Cheeger--Chern--Simons form that interpolates between the character of a representation where the eigenspaces of $\Rot(M)$ might be infinite-dimensional and a cutoff version where the eigenspaces are finite-dimensional. We explain this in~\S\ref{sec:Kthymot} and~\S\ref{sec:examples}.
\end{rmk}

\subsection{Results II: Field theories and topological $q$-expansion}

The remaining results in the paper compare the above differential cocycle model for $\K_\Tate$ with the one for $\TMF\otimes \C$ developed in~\cite{DBE_WG}. The latter model has as cocycles holomorphic sections $\mathcal{O}(\mathcal{L}^{2|1}_0(M);\omega^{\otimes n/2})$ where $\omega^{\otimes n/2}$ is the $n$th tensor power of a line bundle closely related to the square root of the Hodge bundle on the moduli stack of elliptic curves, and $\mathcal{L}^{2|1}_0(M)$ is a stack whose objects are super tori with a constant map, $\phi\colon \R^{2|1}/\Lambda \to M$. \emph{Constant} means the map is invariant under the precomposition action of $\E^2/\Lambda$ by translations. Morphisms between these objects are super translations and dilations of super tori compatible with~$\phi$. This differs from~$\widetilde{\mathcal{L}}^{2|1}_0(M)$ in that a specified meridian super circle isn't part of the data. As such there is a map that forgets this super circle, $u\colon \widetilde{\mathcal{L}}^{2|1}_0(M)\to {\mathcal{L}}^{2|1}_0(M)$. 
Let $\widetilde{\omega}^{\otimes n/2}$ denote the pullback $u^*\omega^{\otimes n/2}$ to $\widetilde{\mathcal{L}}^{2|1}_0(M)$. 

We obtain rescaled partition functions with values in $\widetilde\omega^{\otimes n/2}$ from positive energy representations of $\Ann(M)$ valued in a category of modules over an algebra called the \emph{$n$-free fermions}, denoted $\Fer_n$. These are a generalization of the Clifford algebras, and enjoy many analogous properties. We form a category denoted $\Rep^n(\Ann(M))$ of such $\Fer_n$-linear representations. The category of \emph{trace class} positive energy representations, denoted $\Rep_{\rm TC}^n(\Ann(M))$ is the full subcategory of $\Rep^n(\Ann(M))$ for which the rescaled partition function~\eqref{eq:charintro} takes values in sections of $\widetilde\omega^{\otimes n/2}$ (rather than formal sums of sections) over $\widetilde{\mathcal{L}}^{2|1}_0(M)$. For such representations, we can ask for a lift
\beq
\begin{array}{c}
\begin{tikzpicture}
  \node (A) {$\Rep_{\rm TC}^n(\Ann(M))$};
  \node (B) [node distance= 4cm, right of=A] {$\Gamma(\widetilde{\mathcal{L}}^{2|1}_0(M);\widetilde{\omega}^{\otimes n/2}).$};
  \node (C) [node distance = 1.5cm, above of=B] {$\Gamma(\mathcal{L}^{2|1}_0(M);\omega^{\otimes n/2})$};
  \draw[->] (A) to node [above] {$Z$} (B);
  \draw[->,dashed] (A) to (C);
  \draw[->] (C) to node [right] {$u^*$} (B);
\end{tikzpicture}\end{array}\nonumber
\eeq
If such a factorization exists it is unique and necessarily defines a holomorphic section $\mathcal{O}(\mathcal{L}^{2|1}_0(M);\omega^{\otimes n/2})\subset \Gamma(\mathcal{L}^{2|1}_0(M);\omega^{\otimes n/2})$. We define a category of differential cocycles, ${\widehat\Rep}{}^{n}_\MF(\Ann(M))$, gotten from $\Fer_n$-linear positive energy representations whose rescaled partition functions have such a (holomorphic) lift. 

\begin{thm}\label{thm:KMF}
The differential Grothendieck group of ${\widehat\Rep}{}^{2n}_\MF(\Ann(M))$ gives a model for $\widehat{\K}_\MF^{2n}(M)$, the degree~$2n$ differential $\KMF$ of $M$. 
\end{thm}

\begin{rmk}
For $n$ odd, the $\Fer_n$-linear representations also provide cocycles. The map is surjective when $M=\pt$, but fails to be surjective generally. This is in complete analogy to how finite-dimensional Clifford module bundles map to $\K^n(M)$, but the map isn't surjective in general, e.g., there is no finite-dimensional Clifford module bundle representative for the generator of $\widetilde{\K}^1(S^1)\cong \Z$. 
\end{rmk}

Using similar ideas to Freed and Lott's~\cite{LottFreed} construction of a pushforward in differential K-theory, we construct a differential string orientation for $\KMF$. 

\begin{thm}\label{thm:diffpush}
A geometric family of rational string manifolds~$X\to M$ with fiber dimension~$2d$ determines a differential cocycle
$$
\widehat{\sigma}(X)\in \widehat{\K}_\MF^{-2d}(M).
$$
When $M=\pt$, $\widehat{\sigma}(X)\in \widehat{\K}_\MF^{-2d}(\pt)=\MF^{-2d}_\Z$ is the Witten genus of $X$ as an integral modular form. 
\end{thm}

We interpret $\widehat{\sigma}(X)$ as a cutoff version of the supersymmetric sigma model. When~$M=\pt$, our methods also give an odd variant of this differential orientation where $\widehat{\sigma}(X)\in \widehat{\K}^\odd_\MF(\pt)$ is a version of the Bunke--Naumann secondary invariant of the Witten genus. 

\subsection{Motivation from physics}\label{sec:physmot}

There are two basic approaches in physics that construct a quantum theory out of a classical one. One follows the Lagrangian (or path integral) formalism, which computes expectation values of observables as a (usually ill-defined) integral over a space of fields. The second approach is the Hamiltonian (or canonical) formalism, which looks to construct a space of states with an action by various operators, e.g., a representation of the Poincar\'e group of the relevant dimension (which includes a time-evolution operator), and creation and annihilation operators associated to particles. In the physics literature these approaches are widely assumed to be equivalent, e.g., see Zee~\cite[I.8]{Zee} for some discussion. 

The local index theorem can be viewed as giving a mathematically precise relationship between these approaches in the context of $1|1$-dimensional quantum mechanics~\cite{susymorse,Alvarez,Witten_Dirac}. The fields for the classical theory are super paths $\gamma\colon \R^{1|1}\to X$ in a Riemannian manifold~$X$, and the action generalizes the usual energy of the path; e.g., see~\cite{5lectures}. One can then apply either the Hamiltonian or Lagrangian formalisms to obtain a quantum theory. 
An important quantum observable is the partition function, which counts the difference between the fermionic and bosonic states. In the Hamiltonian approach, the space of states consists of sections of the spinor bundle, and the time-evolution operator is $e^{-t\slashed{D}{}^2}$ where $\slashed{D}$ is the Dirac operator. In this approach, the partition function is the super trace of $e^{-t\slashed{D}{}^2}$. By the McKean--Singer formula, 
$$
{\rm sTr}(e^{-t\slashed{D}^2})={\rm Ind}(\slashed{D})
$$ 
this is the index of~$\slashed{D}$. In the Lagrangian approach, physical reasoning argues that the path integral that calculates the partition function can be reduced to a $\zeta$-regularized determinant for a family of operators over a moduli space of constant super loops $\R^{1|1}/\Z\to X$. This is a version of \emph{1-loop quantization}, because this determinant calculates the contribution from 1-loop Feynman diagrams in a perturbative expansion around these constant super loops. This constructs the $\hat{A}$-form of~$X$, which can be viewed as a function on the moduli space of constant super loops in~$X$. The assertion that the Lagrangian and Hamiltonian computations of the partition function agree is the local index theorem,
\beq
\Z \ni {\rm sTr}(e^{-t\slashed{D}^2})=(2\pi i)^{n/2}\int_X \hat{A}(X)\in \C.\label{eq:localindex}
\eeq
The Hamiltonian computation on the left hand side is necessarily an integer, whereas Lagrangian computation on the right hand side a priori is a complex number. The equality shows that the $\hat{A}$-genus of a spin manifold is in fact an integer. A version of this works in families, considering a kind of fiberwise sigma model for a proper submersion $\pi\colon X\to M$ with spin structures on the fibers. Then we get an equality of differential forms, coming from a (limit of) the Chern character of Bismut's super connection for a family of Dirac operators on the left hand side and the fiberwise integral of the $\hat{A}$-form on the right hand side; see~\S\ref{sec:FreedLott}. 

Witten generalized this story from physics to the $2|1$-dimensional sigma model. The fields in this classical field theory are super annuli $\gamma\colon \R^{2|1}/r\Z\to X$ in a Riemannian manifold~$X$. The classical action is a super-generalization of the one for which critical points are harmonic maps of ordinary annuli to~$M$. Very roughly, this is $1|1$-dimensional mechanics with target~$LM$, the free loop space of~$M$. Applying a version of Hamiltonian quantization, Witten constructed a space of states and a time-evolution operator. By analogy, he referred to these as the \emph{spinor bundle on loop space} $\$_{LX}$ and the \emph{Dirac operator on loop space}, $\slashed{D}_{LX}$, respectively. Decomposing the space of states according to the $S^1$-action on loop space by loop rotation, the partition function ${\rm sTr}(e^{-t\slashed{D}_{LX}^2})$ can be written as a power series in~$q$. The coefficient of $q^n$ counts the difference between fermionic and bosonic states in the space of states on which the $S^1$-action is by the $n^{\rm th}$ power of the basic representation. This power series with integer coefficients is the \emph{Witten genus of~$X$,} which (continuing the analogy) is the index of the Dirac operator on loop space. 

If we apply the Lagrangian formalism to the $2|1$-dimensional sigma model, physical reasoning again argues that the partition function comes from a $\zeta$-regularized determinant for a family of operators over a moduli space of constant super tori, $\R^{2|1}/\Lambda\to X$. However, in this case there is an \emph{anomaly}, meaning the analog of~$\hat{A}(X)$ is a twisted class that can't necessarily be integrated on this moduli stack~\cite{DBE_WG}. In a bit more detail, the $\zeta$-regularized determinant defines a section of a line bundle over the moduli stack of constant super tori in~$X$. Trivializing this line bundle requires a trivialization of~$p_1(TX)$, i.e., a rational string structure. With such a rational string structure set, the $\zeta$-regularized determinant can be identified with a function on the moduli stack of constant super tori in~$X$. The analog of~\eqref{eq:localindex} is the equality
\beq
\Z\llbracket q \rrbracket\ni {\rm sTr}(e^{-t\slashed{D}_{LX}^2})=(2\pi i)^{n/2}\int_X {\rm Wit}(X)\in \MF^{-d} \label{eq:localLindex}
\eeq
where $d={\rm dim}(X)$ and~${\rm Wit}(X)\in {\rm H}_\MF^0(X)$ is a characteristic class called the \emph{Witten class} of $X$ (see~\S\ref{sec:examples}). As before, the left hand side is the computation of the partition function in the Hamiltonian framework, whereas the right hand side is the computation in the Lagrangian framework. The left side is a power series in $q$ with integer coefficients, and the right hand side is a weak modular form over~$\C$. So for these to be equal, they both must be integral modular forms. This is the argument Witten employs in~\cite{Witten_Dirac} to assert the modularity of the Witten genus. To summarize, the Hamiltonian approach leads one to expect \emph{integrality}, whereas the Lagrangian approach leads to expect \emph{modularity}. This philosophy is now old-hat in the string theory literature; e.g., see~\cite[\S1.1]{Dijkgraaf} for an overview. 

However, relatively unstudied in the physics literature is the failure of the families version of~\eqref{eq:localLindex}. In this case, we compare a Chern character of a Bismut super connection on the left hand side with a fiberwise integral of the Witten class on the right hand side. if the first Pontryagin class of the vertical tangent bundle is zero, the difference between the two sides is an exact form, but is not necessarily. A choice of 3-form $H$ with $dH=p_1(T(X/M))$ presents the difference as~$d$ of a specific form. This is the extra data that is required to construct the string orientation of~$\KMF$ analytically. Its physical relevance is the choice of trivializing anomalies in families. 

Broadly stated, the goal of this paper is to put the above physical ideas in direct contact with algebraic topology. This requires we understand \emph{families} of supersymmetric sigma models parametrized by a smooth manifold in the Lagrangian and Hamiltonian frameworks. The Lagrangian picture has been worked out in~\cite{DBE_WG,DBE_MQ}. In brief, the constructions take place over a stack of \emph{constant super tori} in~$M$. These are maps, 
\beq
\R^{2|1}/\Lambda\to (\R^{2|1}/\Lambda)/\E^2\cong \R^{0|1}\to M\label{eq:constanttori}
\eeq
that are invariant under the translation action of $\E^2$ on a super torus, or equivalently, maps that factor through a super point. These can also be viewed as energy zero maps, i.e., the parametrizing space for perturbative constructions such as the 1-loop partition function described above.

This paper starts work on the Hamiltonian picture. For a family of supersymmetric sigma models parametrized by~$M$, the Hamiltonian formalism can be expected to produce an $M$-family of representations of super annuli. This is a geometric (or bordism) description of the time-evolution operator in 2-dimensional field theories due to Segal~\cite{SegalCFT}. It can be made mathematically precise in several distinct ways. The most complete definition to date takes the annular subcategory of Stolz and Teichner's $2|1$-dimensional Euclidean bordism category~\cite{ST11}. We follow a somewhat different approach, incorporating a few simplifications suggested both directly and indirectly by Witten's picture. Before explaining these ingredients, we overview Stolz and Teichner's framework. 

\subsection{Relation to the Stolz--Teichner program}\label{sec:ST}
For a smooth manifold~$M$, Stolz and Teichner have defined a bordism category denoted $2|1\EBord(M)$ whose objects are closed, collared, $1|1$-dimensional super manifolds with a map to~$M$, and whose morphisms are compact, collared, $2|1$-dimensional super Euclidean manifolds with a map to~$M$. Disjoint union of bordisms gives a symmetric monoidal structure. To incorporate isometries of super manifolds, both the objects and morphisms of $2|1\EBord(M)$ are regarded as symmetric monoidal stacks on the site of super manifolds. 

For a target category $\Vect$ of topological vector spaces, they form a category of $2|1$-dimensional super Euclidean field theories, 
$$
2|1\EFT(M):=\Fun^\otimes(2|1\EBord(M),\Vect).
$$
To incorporate a notion of degree, there is a version of the above functors with values in modules over the free fermion algebra~$\Fer_n$ (see~\S\ref{sec:ferdef}). Call this category of field theories~$2|1\EFT^n(M)$. Stolz and Teichner's main main conjecture is the existence of a higher-categorical refinement of $2|1\EFT^n(M)$ incorporating a fully-extended bordism category and a delooping of~$\Vect$ such that there is a natural ring isomorphism
\beq
{\rm TMF}^n(M)\cong 2|1\EFT^n(M) \quad\quad\quad {(\rm conjectural}).\label{eq:STconj}
\eeq
The 1-categorical version of $2|1\EBord(M)$ is already a very intricate object~\cite{ST11}. Surely some level of intricacy is necessary if one wishes to recover a deep object like~TMF. However, the complexity of $2|1\EFT^n(M)$ has been difficult to characterize in terms of standard geometric or topological objects, which has made progress on~\eqref{eq:STconj} difficult. With an eye toward the topological $q$-expansion principle~(\ref{square1}), we extract simpler pieces from~$2|1\EFT^n(X)$ that are characterized using mild elaborations of standard geometric tools. Our choices in these simplifications are informed by the ingredients that go into Witten's construction of the Witten genus. The hope is a simplified model will make it easier to nail down a higher-categorical refinement leading to a cocycle model for~TMF as in~\eqref{eq:STconj}. 

The first simplification restricts the maps of bordisms (in our case, super annuli) to~$M$. Indeed, all of Witten's analysis uses a stand-in for the free loop space consisting of an infinite-dimensional normal bundle over the \emph{finite-dimensional} space of constant loops. See also the related construction of Bott and Taubes~\cite{BottTaubes}. This suggests we consider infinite-dimensional representations of a suitable finite-dimensional category of \emph{constant} super annuli. With the pre-existing construction of $\TMF\otimes \C$~\cite{DBE_WG} in terms of the constant maps~\eqref{eq:constanttori}, we take the simplest possible option, namely maps from super annuli to $M$
$$
\R^{2|1}/\Z\to (\R^{2|1}/\Z)/\E^2\cong \R^{0|1}\to M
$$
that are invariant under precomposition with the translational $\E^2$-action on $\R^{2|1}/\Z$. As desired, this moduli space of annuli over~$M$ is finite-dimensional. A more technical (but equally important) point is that gluing constant super annuli can be arranged without choosing collars. So, for example, the objects we consider are constant super circles in~$M$, whereas in the Stolz--Teichner framework objects are super circles \emph{together with} the germ of a super annulus in~$M$. Following the usual story of bordism categories, gluing collared manifolds is only defined up to isomorphism, resulting in a non-strict composition law. However, in the category of constant super annuli composition is strict. 

The second simplification follows Witten's suggestion in the second paragraph of the introduction to~\cite{Witten_Dirac}:
\begin{quote}
the topological conjecture in question would follow from certain simple (conjectured) properties of the supersymmetric nonlinear sigma model\dots A cutoff version of the nonlinear sigma model would be adequate.
\end{quote}
An example of a cutoff in index theory is the passage from a family of Dirac operators (acting on infinite-dimensional vector spaces) to a (finite-dimensional) index bundle; see~\S\ref{sec:Kthymot}. For the $2|1$-dimensional supersymmetric sigma model, a cutoff sigma model extracts the analog of an index bundle for each weight space of the $S^1$-action that rotates annuli. We explain this in greater detail in~\S\ref{sec:examples}. The important point is that the resulting representation of super annuli is a (finite-type) positive energy representation: weight spaces for the $S^1$-action are finite-dimensional with weight bounded below. This exactly mimics the definition of positive energy loop group representations (which are expected to appear in equivariant refinements of $\K_\Tate$~\cite{Andopower}). We restrict attention to representations of super annuli of this form. 

In families these cutoffs initially introduce a couple issues, but these turn out to be features rather than bugs. First, it is important to identify the theories resulting from different choices of cutoff. This identification is a version of the Grothendieck construction, as we explain in~\S\ref{sec:Kthymot}. Second, a choice of cutoff in a family can affect the value of the partition function. In our formalism it is important to remember this change in the partition function. A positive energy representation of super annuli together with the data of a modification to the partition function is precisely a Hopkins--Singer style differential cocycle. 

Now we explain how our approach compares to Stolz and Teichner's. To guarantee the existence of cutoffs, we restrict to even degree (see Lemma~\ref{lem:FreedLott}); in Remark~\ref{rmk:odd} we describe some approaches to the odd case. We obtain a version of (\ref{square1}),
\begin{equation}
\begin{array}{c}
\begin{tikzpicture}[decoration=snake]
  \node (A) {$2|1\EFT^{2n}(M)$};
  \node (B) [node distance= 6.9cm, right of=A] {$\Rep^{2n}(\Ann(M))$};
  \node (C) [node distance = 2cm, below of=A] {$\mathcal{O}(\mathcal{L}^{2|1}_0(M);\omega^{\otimes 2n/2})$};
  \node (D) [node distance = 2cm, below of=B] {$\mathcal{O}(\widetilde{\mathcal{L}}^{2|1}_0(M)).$};
  \node (E) [node distance=3cm, right of=A] {$\null$};
  \node (F) [node distance=.75cm, below of=E] {$\twocommute \ \eta$};
  \draw[->,decorate] (A) to node [above] {${\rm restrict+cutoff}$}  (B);
  \draw[->] (A) to node [left] {$\begin{array}{c}{\rm restrict}\\ + \\ {\rm rescale}\end{array}$} (C);
  \draw[->] (C) to node [below] {$({\rm forget})^*$} (D);
  \draw[->] (B) to node [right] {$\begin{array}{c}{\rm rescaled} \\ {\rm partition}\\ {\rm function}\end{array}$} (D);
\end{tikzpicture}\end{array}\label{squareST}
\end{equation}
The left vertical arrow restricts a Stolz--Teichner field theory to a subcategory of super tori (viewed as bordisms from the empty set to itself) and rescales these partition functions in the same manner as in~\eqref{eq:charintro}. This gives a section of a line bundle over the moduli stack of super tori. The upper horizontal arrow restricts to a subcategory of annuli and chooses a cutoff, extracting a positive energy representation of super annuli. The vertical arrow on the right is the rescaled partition function~\eqref{eq:charintro}. The lower horizontal map is the pullback along the forgetful functor $\widetilde{\mathcal{L}}^{2|1}_0(M)\to \mathcal{L}^{2|1}_0(M)$ that forgets a chosen embedded super circle in a super torus. 

The squiggly arrow in~\eqref{squareST} is intended to emphasize that there are choices of cutoff involved in extracting a (finite-type) positive energy representation from a field theory. This choice means that we cannot expect the diagram to commute strictly; instead, there is a concordance~$\eta$ between the character of the positive energy representation of super annuli and the partition function of the input field theory. However, we sketch in~\S\ref{sec:Kthymot} how the induced map to the Grothendieck group of the representation category is independent of these choices. 


\begin{rmk}
A more straightforward comparison between the definitions in this paper and Stolz--Teichner field theories is the warm-up example concerning super paths and K-theory. A bordism category $1|1\EBord(M)$ analogous to $2|1\EBord(M)$ as sketched above gives a definition of $1|1$-Euclidean field theories over $M$, denoted~$1|1\EFT(M)$. There is an elegant relationship between a classifying space of $1|1$-Euclidean field theories over $M=\pt$ and the representing space~$BO\times \Z$ for real K-theory~\cite{HST}. However, as yet concordance classes of the category $1|1\EFT(M)$ of $1|1$-Euclidean field theories over~$M$ have not been calculated. Some good evidence for a relationship to K-theory is Dumitrescu's super parallel transport map~\cite{florin} from super vector bundles with super connection to $1|1\EFT(M)$. However, issues regarding collars have made a complete characterization difficult. In our setup, the category of representations of constant super paths in~$M$ are super vector bundles with super connection on the nose. From this category it is straightforward to build $\K(M)$, e.g., the usual Grothendieck group applied to finite-dimensional representations. A downside of these constant super paths is that they don't connect with extended objects in $M$ coming from K-theory, such as the Bismut--Chern character. Fei Han~\cite{Han} has related this loop space lift of the Chern character to partition functions for~$1|1\EFT(M)$. 
\end{rmk}

\begin{rmk}\label{rmk:odd}
There are several approaches to the odd cohomological degree in~\eqref{squareST}. One is to abandon cutoffs and work with infinite-dimensional objects. Although more technical, this doesn't seem unreasonable, especially if we insist on a finite-dimensional category of (constant) super annuli in~$M$. An option involving finite-dimensional representations of super annuli is to study degree $2n$ field theories over $M\times \R$ with compact support in the $\R$-direction. Then the suspension isomorphisms in $\K_\Tate$ and $\TMF\otimes \C$ identify such an object with a class of degree~$2n-1$. Yet another possibility is to work with families of automorphisms of representations of annuli, which is analogous to studying odd K-theory in terms of (smooth) maps into the (infinite) unitary group. In this paper our focus is on the even degree case, though our methods also lead to partial results in the odd case. 
\end{rmk} 

\subsection{Terminology} 

Throughout, $M$ will denote a smooth, compact, oriented manifold. Unless stated otherwise, all vector spaces and vector bundles are ``super," though sometimes we include this adjective for emphasis. We use $\otimes$ to denote the graded tensor product. We frequently use the functor of points when dealing with super manifolds, and reserve the letter~$S$ for a test super manifold. We refer to Appendix~\ref{appenA} for a bit more background on these (and other) ingredients. 

\subsection{Outline}

The next section,~\S\ref{sec:lie}, introduces some minor twists on standard background material. This develops the framework in which we can make sense out of smooth representations of super paths and super annuli. Sections~\ref{sec:Quillenconn} and~\ref{sec:per} construct the (non-differential) models for K-theory and elliptic cohomology at the Tate curve, respectively, out of categories of these representations. These constructions are a bit more straightforward than the differential versions and can be read with~\S\ref{sec:lie} alone as background. Sections~\ref{sec:diffKmain} and~\ref{sec:diffellmain} construct the differential refinements~$\widehat\K$ and~$\widehat\K_\Tate$, which amounts to understanding the appropriate character theory for representations of constant super paths and constant super annuli. Section~\ref{sec:freeferKMF} constructs the differential version of~$\KMF$, which refines this character theory for super annuli even further to take values in a differential model for~${\rm H}_\MF$. Finally, in~\S\ref{sec:examples} we discuss the string orientation of~$\KMF$ and its relation to the supersymmetric sigma model. The parts of~\S\ref{sec:examples} regarding the orientation (but not its physical interpretation) can be read independent of the rest of the paper. 

\subsection{Acknowledgements}
I thank Ralph Cohen, Kevin Costello, Chris Douglas, Mike Freedman, Owen Gwilliam, Vesna Stojanoska, Stephan Stolz, Peter Teichner, and Arnav Tripathy for helpful conversations.  Lastly, I am appreciative of the well-timed encouragement from Matt Ando and Gerd Laures. 

\section{Super Lie categories, representations, and Grothendieck groups}\label{sec:lie}

This section overviews the framework in which we define the main objects of study. First we introduce super Lie categories and their unitary representations. These are modest generalizations of standard ideas in the theory of Lie groupoids for which~\cite{Mackenzie} is a reference. Second, for super Lie categories natural in manifold parameter, we define a differential Grothendieck group of the representation category with respect to a character map. This definition is intended to mimic Hopkins--Singer differential cocycles~\cite{HopSing}. 

\subsection{Internal categories and super Lie categories}\label{sec:internal}

We follow~\cite[MII.1]{MacLane} as a reference for internal categories. Our thinking is also deeply influenced by Stolz and Teichner~\cite[\S2.2]{ST11} in their approach to smooth field theories. 

\begin{defn} A \emph{category object in ${\sf E}$} or \emph{an internal category} denoted $\CC=\{\CC_1\toto\CC_0\}$ is a pair of objects $\CC_0$ and $\CC_1$ in ${\sf E}$ and arrows in ${\sf E}$ 
$$
s,t\colon\CC_1\to \CC_0\quad u\colon \CC_0\to \CC_1\quad c\colon \CC_1\times_{\CC_0}\CC_1\to \CC_1
$$
with the notation standing for source, target, unit and composition, respectively. We further require that these data satisfy the usual axioms for a category, namely, the equalities
$$
s\circ u=\id_{\CC_0}=t\circ u
$$
specify the source and target of identity arrows, the commutative diagram
\begin{equation}
\begin{array}{c}
\begin{tikzpicture}
  \node (A) {$\CC_0\times_{\CC_0}\CC_1$};
  \node (B) [node distance= 3.5cm, right of=A] {$\CC_1\times_{\CC_0}\CC_1$};
  \node (C) [node distance = 3.5cm, right of=B] {$\CC_1\times_{\CC_0}\CC_0$};
  \node (D) [node distance = 1.25cm, below of=B] {$\CC_1$};
  \draw[->] (A) to node [above] {$u\times \id_{\CC_1}$} (B);
  \draw[->] (A) to node [below] {$p_2$} (D);
  \draw[->] (C) to node [above] {$\id_{\CC_1}\times u$} (B);
  \draw[->] (B) to node [right] {$c$} (D);
    \draw[->] (C) to node [below] {$p_1$} (D);
\end{tikzpicture}\end{array}\nonumber
\end{equation}
specifies the source and target of the composition, and the commutative diagrams
\begin{equation}
\begin{array}{c}
\begin{tikzpicture}
  \node (A) {$\CC_1$};
  \node (B) [node distance= 2cm, right of=A] {$\CC_1\times_{\CC_0}\CC_1$};
  \node (C) [node distance = 2cm, right of=B] {$\CC_1$};
  \node (D) [node distance = 1.5cm, below of=A] {$\CC_0$};
  \node (E) [node distance = 1.5cm, below of=B] {$\CC_1$};
  \node (F) [node distance = 1.5cm, below of=C] {$\CC_0$};
  \draw[->] (B) to node [above] {$p_1$} (A);
  \draw[->] (B) to node [above] {$p_2$} (C);
  \draw[->] (A) to node [left] {$t$} (D);
  \draw[->] (B) to node [right] {$c$} (E);
  \draw[->] (C) to node [right] {$s$} (F);
  \draw[->] (E) to node [below] {$t$} (D);
  \draw[->] (E) to node [below] {$s$} (F);
\end{tikzpicture}\end{array}\nonumber\qquad
\begin{array}{c}
\begin{tikzpicture}
  \node (A) {$\CC_1\times_{\CC_0}\CC_1\times_{\CC_0}\CC_1$};
  \node (B) [node distance= 4cm, right of=A] {$\CC_1\times_{\CC_0}\CC_1$};
  \node (C) [node distance = 1.5cm, below of=A] {$\CC_1\times_{\CC_0}\CC_1$};
  \node (D) [node distance = 1.5cm, below of=B] {$\CC_1$};
  \draw[->] (A) to node [above] {$c\times \id_{\CC_1}$} (B);
  \draw[->] (A) to node [left] {$\id_{\CC_1}\times c$} (C);
  \draw[->] (B) to node [right] {$c$} (D);
  \draw[->] (C) to node [below] {$c$} (D);
\end{tikzpicture}\end{array}\nonumber
\end{equation}
require that identity arrows act as the identity and that composition is associative. 
\end{defn}

In the above definition we require that the fibered products exist in~${\sf E}$. This can be relaxed using categories internal to presheaves on ${\sf E}$, as we discuss in~\S\ref{sec:genLie}. 

\begin{defn}
An \emph{internal functor} $F\colon \CC\to\DD$ between internal categories consists of morphisms $F_0\colon\CC_0\to \DD_0$ and $F_1\colon \CC_1\to \DD_1$ in ${\sf E}$ such that the diagrams commute:
\begin{equation}
\begin{array}{c}
\begin{tikzpicture}
  \node (A) {$\CC_1$};
  \node (B) [node distance= 2cm, right of=A] {$\CC_0$};
  \node (C) [node distance = 2cm, right of=B] {$\CC_1$};
  \node (D) [node distance = 1.5cm, below of=A] {$\DD_1$};
  \node (E) [node distance = 1.5cm, below of=B] {$\DD_0$};
  \node (F) [node distance = 1.5cm, below of=C] {$\DD_1$};
    \draw[->] (A) to node [above] {$s$} (B);
  \draw[->] (B) to node [above] {$u$} (C);
    \draw[->] (D) to node [below] {$s$} (E);
  \draw[->] (E) to node [below] {$u$} (F);
  \draw[->] (A) to node [right] {$F_1$} (D);
  \draw[->] (B) to node [right] {$F_0$} (E);
  \draw[->] (C) to node [right] {$F_1$} (F);
\end{tikzpicture}\end{array}
\nonumber\qquad 
\begin{array}{c}
\begin{tikzpicture}
  \node (A) {$\CC_1$};
  \node (B) [node distance= 2cm, right of=A] {$\CC_0$};
  \node (C) [node distance = 2cm, right of=B] {$\CC_1$};
  \node (D) [node distance = 1.5cm, below of=A] {$\DD_1$};
  \node (E) [node distance = 1.5cm, below of=B] {$\DD_0$};
  \node (F) [node distance = 1.5cm, below of=C] {$\DD_1$};
  \draw[->] (A) to node [above] {$t$} (B);
  \draw[->] (B) to node [above] {$u$} (C);
  \draw[->] (D) to node [below] {$t$} (E);
  \draw[->] (E) to node [below] {$u$} (F);
  \draw[->] (A) to node [right] {$F_1$} (D);
  \draw[->] (B) to node [right] {$F_0$} (E);
  \draw[->] (C) to node [right] {$F_1$} (F);
\end{tikzpicture}\end{array}
\end{equation}
\begin{equation}
\begin{array}{c}
\begin{tikzpicture}
  \node (A) {$\CC_1\times_{\CC_0}\CC_1$};
  \node (B) [node distance= 3cm, right of=A] {$\DD_1\times_{\DD_0}\DD_1$};
  \node (C) [node distance = 1.5cm, below of=A] {$\CC_1$};
  \node (D) [node distance = 1.5cm, below of=B] {$\DD_1.$};
  \draw[->] (A) to node [above] {$F_1\times F_1$} (B);
  \draw[->] (B) to node [right] {$c_\DD$} (D);
  \draw[->] (C) to node [below] {$F_1$} (D);
  \draw[->] (A) to node [left] {$c_\CC$} (C);
\end{tikzpicture}\end{array}\nonumber
\end{equation}
\end{defn}

\begin{defn}
An \emph{internal natural transformation} $\eta \colon F\Rightarrow G$ between internal functors is a morphism in ${\sf E}$, $\eta \colon \CC_0\to \DD_1$ satisfying 
\begin{equation}
\begin{array}{c}
\begin{tikzpicture}
  \node (A) {$\CC_0$};
  \node (B) [node distance= 1.5cm, below of=A] {$\DD_1$};
  \node (C) [node distance = 2cm, right of=B] {$\DD_0$};
  \node (D) [node distance = 2cm, left of=B] {$\DD_0$};
  \draw[->] (A) to node [left] {$\eta$} (B);
  \draw[->] (B) to node [above] {$s$} (C);
  \draw[->] (B) to node [above] {$t$} (D);
  \draw[->] (A) to node [above=.1cm] {$F_0$} (C);
  \draw[->] (A) to node [above=.1cm] {$G_0$} (D);
\end{tikzpicture}\end{array}\nonumber\qquad
\begin{array}{c}
\begin{tikzpicture}
  \node (A) {$\CC_1$};
  \node (B) [node distance= 4cm, right of=A] {$\DD_1\times_{\DD_0} \DD_1$};
  \node (C) [node distance = 1.5cm, below of=A] {$\DD_1\times_{\DD_0}\DD_1$};
  \node (D) [node distance = 1.5cm, below of=B] {$\DD_1.$};
  \draw[->] (A) to node [above] {$G_1\times \eta\circ s$} (B);
  \draw[->] (A) to node [left] {$\eta\circ t \times F_1$} (C);
  \draw[->] (B) to node [right] {$c_\DD$} (D);
  \draw[->] (C) to node [below] {$c_\DD$} (D);
\end{tikzpicture}\end{array}\nonumber
\end{equation}

\end{defn}

\begin{defn}
A \emph{Lie category} is a category internal to manifolds, and a \emph{super Lie category} is a category internal to super manifolds.
\end{defn}

This generalizes the standard notion of a (super) \emph{Lie groupoid}, which is a category internal to (super) manifolds in which all arrows are invertible, with a smooth inversion map included as part of the data. 

\begin{ex} A (super) Lie category with a single object is a unital (super) semigroup. \end{ex}

\begin{ex}
Consider a manifold $M$ with an $\R$-action. This determines an action groupoid $M\sq \R= \{M\times \R\toto M\}$ whose source map is the projection and target map is the action map. Composition comes from the group structure on~$\R$. Restricting to $M\times \R_{\ge 0}\subset M\times \R$ the action by a semigroup defines a Lie category, $\{M\times \R_{\ge 0}\toto M\}$. More generally, for a Lie group $G$ acting on a manifold~$M$, we get a Lie groupoid $M\sq G$. Restriction of the action to a submanifold of $G$ containing the identity element and that is closed under multiplication (but not necessarily inversion) defines a Lie category. 
\end{ex}

\begin{defn} A \emph{smooth functor} is an internal functor between Lie categories or super Lie categories. \end{defn}

\begin{ex} Let $G$ and $H$ be semigroups, $M$ a manifold with a $G$-action and $N$ a manifold with an $H$-action. Then a smooth functor $M\sq G\to N\sq H$ is a homomorphism $G\to H$ of semigroups and an equivariant map $M\to N$ with respect to this homomorphism. \end{ex}

\begin{ex}\label{ex:parityfunctor} Every super manifold has a canonical $\Z/2$-action generated by the parity involution on its sheaf of functions, acting by $+1$ on even functions and $-1$ on odd functions. Furthermore, any map between super manifolds is equivariant for this $\Z/2$-action. This gives a parity endofunctor ${\sf P}_\CC\colon \CC\to \CC$ for any super Lie category $\CC$ that applies the parity functor to source, target, unit, and composition of~$\CC$. \end{ex}

\begin{defn} A \emph{smooth natural transformation} is an internal natural transformation between smooth functors. \end{defn}

\begin{rmk} There is a 2-functor from the strict 2-category of super Lie categories, smooth functors and smooth natural transformations to a bicategory defined by Stolz and Teichner whose objects are weak category objects internal to stacks on the site of super manifolds~\cite{ST11}. On objects, this 2-functor identifies a super Lie category with its corresponding category internal to smooth stacks, which amounts to viewing the object and morphism super manifolds as (representable) stacks. Using this 2-functor, the super Lie categories considered below of constant super paths and super annuli in a manifold can be viewed as smooth subcategories of Stolz and Teichner's super Euclidean bordism categories. \end{rmk}

\subsection{Generalized, reduced, and conjugate Lie categories}\label{sec:genLie}

For objects $S,C\in {\sf E}$, let $C(S)$ denote the set of maps~$S\to C$, i.e., identify $C$ with the (representable) presheaf $C\colon {\sf E}^\op\to {\sf Set}$. Similarly, for a category $\CC$ internal to ${\sf E}$, let $\CC(S)$ denote the small category whose objects are $\CC_0(S)$ and morphisms are $\CC_1(S)$. Hence, an internal category gives a category object in presheaves on ${\sf E}$. By the usual Yoneda argument, $\CC$ is completely determined by this category object in presheaves. Similarly, smooth functors between super Lie categories coincide with functors between the category objects in presheaves. 

The functor of points is a standard way of performing constructions in super manifolds, so we often use this perspective when analyzing super Lie categories and smooth functors. In some cases we will encounter category objects in presheaves on super manifolds that are not representable. These presheaves are sometimes called \emph{generalized} objects. 

\begin{defn}\label{defn:genLiecat} A \emph{generalized super Lie category} is a category object in presheaves on super manifolds. \end{defn}

There is a faithful functor from manifolds to super manifolds, where we take the usual smooth functions on a manifold as the structure sheaf (purely in even degree). We often abuse notation, letting $M$ denote both a manifold and the associated super manifold. There is also a functor from super manifolds to manifolds called the \emph{reduced manifold} functor that on objects takes the quotient of the structure sheaf of a super manifold by the ideal generated by nilpotent elements. For a super manifold $S$, let $S_{\rm red}$ denote this reduced manifold, and observe that the quotient map gives a morphism of super manifolds $S_{\rm red}\hookrightarrow S$, where we have identified the manifold $S_{\rm red}$ with a super manifold. 

These functors can also be applied to super Lie categories.

\begin{defn} The \emph{super Lie category associated to a Lie category $\CC$} regards the data defining $\CC$ as super manifolds and maps of super manifolds; in particular, it views the object and morphism manifolds of ~$\CC$ as super manifolds. The \emph{reduced Lie category of a super Lie category $\CC$}, denoted $\CC_{\rm red}$ applies the reduced manifold functor to a Lie category. In particular, it has as objects $(\CC_0)_{\rm red}$ and as morphisms $(\CC_1)_{\rm red}$. \end{defn}

Finally, there is a complex conjugation endofunctor on the category of super manifolds $N\mapsto \overline{N}$ that reverses the complex structure on sheaves of functions, i.e., $\overline{N}$ has the same real sheaf of functions but complex numbers act by precomposing with complex conjugation. 

\begin{defn} A \emph{real structure} on $N$ is an isomorphism of super manifolds $r_N\colon N\to \overline{N}$ such that $r_N\circ \overline{r}_N=\id_{\overline{N}}$ and $\overline{r}_N\circ r_N=\id_N$. \end{defn}

\begin{ex} For a super manifold~$N$ whose sheaf of $\Z/2$-graded algebras is the exterior bundle of a complex vector bundle $E$, i.e., $C^\infty(N)=\Gamma(N_{\rm red},\Lambda^\bullet E^*)$, a real structure on~$E$ induces a real structure on $N$. \end{ex}

\begin{ex}
In particular, an ordinary manifold regarded as a super manifold has a real structure coming from complex conjugation on its complex valued functions. 
\end{ex}

See~\cite{DM} pages 92-94 and~\cite{HST} page~37 for more background on real super manifolds. These ideas carry over to super Lie categories immediately.

\begin{defn} For a super Lie category $\CC$, let $\overline{\CC}$ denote the super Lie category gotten by applying the conjugation functor to~$\CC$. In particular, the object and morphism super manifolds are $\overline{\CC}_0$ and $\overline{\CC}_1$. A \emph{real structure} on a super Lie category is a functor ${\sf r}_\CC\colon \CC\to \overline{\CC}$ and natural isomorphisms ${\sf r}_\CC\circ \overline{\sf r}_\CC\cong \id_\CC$.  \end{defn}

\begin{ex} Any super Lie category that comes from an ordinary Lie category has a canonical real structure. \end{ex}

\subsection{Representations of super Lie categories}\label{sec:Lierep}
There is an evident super manifold generalization of the standard definition of a representation of a Lie groupoid~\cite{Mackenzie}. For $V\to N$ a super vector bundle over a super manifold, consider the super vector bundle of invertible maps, $\Hom(p_1^*V,p_2^*V)^\times\to N\times N$, between fibers for $p_1,p_2\colon N\times N\to N$ the projections. The frame groupoid has $N$ as its super manifold of objects, and (roughly) for a pair of points $x,y\in N$ the morphisms are invertible linear maps $V_x\to V_y$ between the fibers of~$V$. The source of such a morphism is $x$ and the target is~$y$. We'll make this precise after a technical remark on vector bundles over super manifolds. 

\begin{rmk}\label{rmk:frame} In the category of super manifolds there are some subtleties regarding vector bundles and maps between vector bundles. The basic point is that a vector bundle over a super manifold is a module over is structure sheaf, and this module is typically different than a would-be set of maps from the base super manifold into a candidate total space of the vector bundle. The usual intuition applies, however, provided we work with the functor of points: an $S$-point of a vector bundle $V$ over $N$ is a map $f\colon S\to N$ and an element of the pullback module associated to~$V$. We continue to use much of the standard notation and terminology for vector bundles. For example, $\Gamma(V)$ denotes the $C^\infty(N)$-module associated with a vector bundle~$V$ over~$N$, and $V\to N$ denotes the natural transformation between the $S$-points of $V$ and the $S$-points of $N$. With this in mind, we make the following definition. 
\end{rmk}

\begin{defn}\label{defn:GLV} For $V\to N$ a super vector bundle, the \emph{frame groupoid} is a generalized super Lie category,
$$
{\sf GL}(V):=\left(\begin{array}{c} \Hom(p_1^*V,p_2^*V)^\times \\ \pi_1 \downarrow \downarrow \pi_2 \\ N\end{array}\right)
$$
where $\Hom(p_1^*V,p_2^*V)^\times$ are invertible vector bundle maps, and $\pi_1,\pi_2$ are the two compositions $\Hom(p_1^*V,p_2^*V)^\times\to N\times N\toto N$ of projections. Composition in ${\sf GL}(V)$ is composition of maps of vector bundles. 
\end{defn}

\begin{defn} 
A \emph{representation} of a super Lie groupoid ${\sf G}$ is a smooth functor ${\sf G}\to {\sf GL}(V)$ for $V\to {\sf G}_0$ a super vector bundle over the objects of ${\sf G}$. 
\end{defn}

\begin{ex}
When $M=\pt$ and $V$ is a trivial bundle, ${\sf GL}(V)$ is the one-object groupoid associated to the group $\GL(V)$. For a Lie group~$G$, a representation of the associated single-object Lie groupoid ${\sf G}=\{G\toto \pt\}$ on ${\sf GL}(V)$ is the same as a homomorphism $G\to {\rm GL}(V)$, recovering the usual definition of a representation. 
\end{ex}

There is an obvious generalization to a representation of a super Lie category taking values in \emph{arbitrary} linear maps between fibers. 

\begin{defn} For $V\to M$ a super vector bundle, define a generalized Lie category called the \emph{frame category} by
$$
\End(V):=\left(\begin{array}{c} \Hom(p_1^*V,p_2^*V) \\ \pi_1 \downarrow \downarrow \pi_2 \\ M\end{array}\right)
$$
where $\pi_1$ and $\pi_2$ are the projection  $\Hom(p_1^*V,p_2^*V)\to M\times M$ postcomposed with the projections~$p_1$ or~$p_2$ to~$M$. Composition comes from composition of vector bundle maps. 
\end{defn}

\begin{defn} A representation of a super Lie category ${\sf C}$ is a super vector bundle $V\to {\sf C}_0$ and a smooth functor ${\sf C}\to \End(V).$ An \emph{isomorphism} between representations $\rho$ and $\rho'$ is an isomorphism of vector bundles $V\to V'$ over ${\sf C}_0$ that induces an isomorphism of super Lie categories $\End(V)\to \End(V')$ and makes the following diagram commute 
\begin{equation}
\begin{array}{c}
\begin{tikzpicture}[node distance=3.5cm,auto]
  \node (A) {${\sf C}$};
  \node (B) [node distance= 3cm, right of=A] {$\null$};
  \node (C) [node distance = .75cm, below of=B] {$\End(V').$};
  \node (D) [node distance = .75cm, above of=B] {$\End(V)$};
  \draw[->] (A) to node [swap] {$\rho'$} (C);
  \draw[->] (A) to node {$\rho$} (D);
  \draw[->] (D) to (C);
\end{tikzpicture}\end{array}\nonumber
\end{equation}
\end{defn}

\subsection{Super adjoints}\label{sec:superadjoint}

To phase the definition of a unitary representation, we require a super-generalization of the adjoint of a linear map between Hilbert spaces. 
We now attend to the various sign issues involved, replicating pages~89-91 of~\cite{strings1}. 
 
 A \emph{$*$-structure} on a super algebra $A$ over $\C$ is an involutive $\C$-antilinear isomorphism $(-)^*\colon A\to A^\op$ satisfying
\beq
(ab)^*=(-1)^{p(a)p(b)}b^*a^*. \label{eq:adjointantiho}
\eeq
A \emph{super hermitian form} on a complex super vector space $H$ is a map $\langle-,-\rangle \colon H\otimes H\to \C$ of real super vector spaces that is $\C$-antilinear in the first variable, $\C$ linear in the second, and
\beq
\langle x,y\rangle =(-1)^{p(x)p(y)} \langle y,x\rangle.\label{eq:superherm}
\eeq
It follows that $\langle x,x\rangle$ is real if $x$ is even, purely imaginary if $x$ is odd, and $\langle x,y\rangle=0$ when $x$ and $y$ are of different parities. We call $\langle-,-\rangle$ \emph{positive} if $\langle x,x\rangle >0$ for $x$ even and $-i\langle x,x\rangle>0$ for $x$ odd. From this we extract an ordinary positive-definite inner product $(-,-)\colon H\otimes H\to \C$ first by writing $H=H_0\oplus H_1$ as a direct sum of even and odd subspaces, and then setting $(x,y)=0$ if $x$ and $y$ are of opposite parity, $(x,y)=\langle x,y\rangle$ for $x$ and $y$ both even, and $(x,y)=-i\langle x,y\rangle$ for $x$ and $y$ both odd. We call $(H,\langle-,-\rangle)$ a \emph{super Hilbert space} if $H\cong H_0\oplus H_1$ is complete with respect to the (ordinary) inner product $(-,-)$ inherited  from $\langle-,-\rangle$. 

As the above discussion shows, the difference between super hermitian forms and ordinary hermitian forms on $\Z/2$-graded vector spaces is slight. As such, we adopt the following terminology. 

\begin{defn} A \emph{super hermitian vector space} is a super vector space equipped with a super hermitian pairing $\langle-,-\rangle$ as in~\eqref{eq:superherm}. A \emph{$\Z/2$-graded hermitian vector space} is a hermitian pairing $(-,-)$ on a $\Z/2$-graded vector space so that the odd and even subspaces are orthogonal. \end{defn}

The super adjoint of a hermitian form preserving linear map $a\colon H\to H'$ is defined by 
\beq
\langle x,a y\rangle_{H'}=(-1)^{p(x)p(y)} \langle a^*x,y\rangle_{H}. \label{eq:supersa}
\eeq
In finite dimensions this completely determines $a^*$. There are the usual caveats in infinite dimensions which can be addressed in the standard way by translating back to usual Hilbert spaces as above. Indeed, under this translation the super adjoint compares with the usual adjoint $(-)^\dagger$ on the inner product spaces $(H,(-,-)_H)$ and $(H',(-,-)_{H'})$ as
\beq
a^*=\left\{\begin{array}{cc} a^\dagger & a \ {\rm even}\\ ia^\dagger & a \ {\rm odd.}\end{array}\right.\label{eq:superadjointnormaladjoint}
\eeq
For $a,b\in \End(H)$, we have $(ab)^*=(-1)^{p(a)p(b)}b^*a^*$, so that $(-)^*$ defines a star structure on the super algebra~$\End(H)$. We also get a contravariant functor from the category of super Hilbert spaces to itself. Note that the endofunctor $((-)^*)^*$ is the parity involution on $\End(H)$ that is the identity on even operators and $-1$ on odd operators. 

Super adjoints have a straightforward generalization to the frame category of a hermitian super vector bundle $V\to N$. Let ${\sf End}(\overline{V})$ denote the frame category with the conjugate complex structure on the base super manifold and on the module defining~$V$. 

\begin{defn} 
Let $V\to N$ be a hermitian super vector bundle. Define the \emph{super adjoint} $(-)^*\colon {\sf End}(V)\to {\sf End}(\overline{V})^\op$ as the fiberwise super adjoint on morphisms,
$$
(-)^*\colon \Hom(p_1^*V,p_2^*V) \to \Hom(p_2^*\overline{V},p_1^*\overline{V}),
$$
using the pullback hermitian form on $p_1^*V$ and $p_2^*V$. By inspection, this restricts to a smooth functor $(-)^*\colon {\sf GL}(V)\to {\sf GL}(\overline{V})^\op$ on the frame groupoid. 
\end{defn}

\subsection{Unitary representations of super Lie categories}\label{sec:unitaryreps}

A unitary representation of a Lie group~$G$ is a homomorphism $\rho \colon G\to \GL(V)$ such that for all $g\in G$, $\rho(g)^*=\rho(g^{-1})$. Repackaged, this is a commutative diagram 
\begin{equation}
\begin{array}{c}
\begin{tikzpicture}
  \node (A) {$G$};
  \node (B) [node distance= 2.5cm, right of=A] {$\GL(V)$};
  \node (C) [node distance = 1.5cm, below of=A] {$G^\op$};
  \node (D) [node distance = 1.5cm, below of=B] {$\GL(V)^\op$};
  \draw[->] (A) to node [above] {$\rho$} (B);
  \draw[->] (A) to node [left] {$(-)^{-1}$} (C);
  \draw[->] (C) to node [below] {$\rho^\op$} (D);
  \draw[->] (B) to node [right] {$(-)^*$} (D);
\end{tikzpicture}\end{array}\nonumber
\end{equation}
where $G^\op$ denotes the opposite super Lie group. This standard definition is for a \emph{real} Lie group, so below for super manifolds we'll need to use care when reversing complex structures. With this in mind, there is a straightforward generalization to super Lie groupoids, using the super adjoint on the frame groupoid defined above. 

\begin{defn} Let ${\sf G}=\{G_1\toto G_0\}$ be a super Lie groupoid with real structure ${\sf r}_{\sf G}\colon {\sf G}\to \overline{{\sf G}}$ and $V\to G_0$ be a vector bundle with hermitian metric. A \emph{unitary representation of a super Lie groupoid} is a functor $\rho\colon {\sf G}\to {\sf GL}(V)$ that makes the diagram commute
\begin{equation}
\begin{array}{c}
\begin{tikzpicture}
  \node (A) {${\sf G}$};
  \node (B) [node distance= 2.5cm, right of=A] {${\sf GL}(V)$};
  \node (C) [node distance = 1.5cm, below of=A] {$\overline{{\sf G}}^\op$};
  \node (D) [node distance = 1.5cm, below of=B] {${\sf GL}(\overline{V})^\op$};
  \draw[->] (A) to node [above] {$\rho$} (B);
  \draw[->] (A) to node [left] {$ (-)^{-1}\circ r_G$} (C);
  \draw[->] (C) to node [below] {$\overline{\rho}^\op$} (D);
  \draw[->] (B) to node [right] {$(-)^*$} (D);
\end{tikzpicture}\end{array}\nonumber
\end{equation}
where $(-)^{-1}$ is the inversion on the Lie groupoid. 
\end{defn}

Super Lie categories do not have the data of an inversion functor, so to define a notion of unitary representation we require a choice of functor $\sigma\colon \CC\to \overline{\CC}^\op$. 

\begin{defn} A \emph{anti-involution} of a super Lie category is a functor $\sigma\colon \CC\to \overline{\CC}^\op$ such that $\sigma\circ\bar\sigma^\op= {\sf P}_{\overline{\CC}}$ and $\bar\sigma^\op\circ \sigma= {\sf P}_\CC$ for ${\sf P}$ the parity automorphism on super Lie categories from Example~\ref{ex:parityfunctor}. \end{defn}

\begin{defn}\label{defn:unitaryrep} Let $\CC=\{\CC_1\toto \CC_0\}$ be a super Lie category and $V\to \CC_0$ be a vector bundle with hermitian metric, and $\sigma$ an anti-involution of $\CC$. A \emph{unitary representation} on $\End(V)$ has as data functors $\sigma\colon \CC\to \overline{\CC}^\op$ and $\rho\colon \CC\to \End(V)$ that make the diagram commute
\begin{equation}
\begin{array}{c}
\begin{tikzpicture}
  \node (A) {$\CC$};
  \node (B) [node distance= 2.5cm, right of=A] {$\End(V)$};
  \node (C) [node distance = 1.5cm, below of=A] {$\overline{\CC}^\op$};
  \node (D) [node distance = 1.5cm, below of=B] {$\End(\overline{V})^\op$};
  \draw[->] (A) to node [above] {$\rho$} (B);
  \draw[->] (A) to node [left] {$\sigma$} (C);
  \draw[->] (C) to node [below] {$\overline{\rho}^\op$} (D);
  \draw[->] (B) to node [right] {$(-)^*$} (D);
\end{tikzpicture}\end{array}\nonumber
\end{equation}
When $\sigma$ is fixed, we use the notation $\Rep({\sf C})$ to denote the category of unitary representations of $\CC$ and their isomorphisms. The direct sum and tensor product of vector bundles endow $\Rep({\sf C})$ with a pair of monoidal structures, denoted~$\oplus$ and~$\otimes$. 

\end{defn}

\subsection{Representations valued in $A$-modules}\label{sec:Amod}

For a Lie category $\CC$, a \emph{bundle of algebras} ${\sf A}\to \CC$ is an algebra bundle (in the usual sense) $A\to \CC_0$ and an isomorphism of algebra bundles, $s^*A\to t^*A$ over $\CC_1$ compatible with composition and units. 

\begin{rmk} A more general (and conceivably more interesting) theory arises from asking for Morita equivalences between $s^*A$ and $t^*A$ over $\CC_1$ rather than algebra isomorphisms. However, in our examples of interest the above suffices. \end{rmk}

\begin{defn} Fix an algebra bundle ${\sf A}\to \CC$. A \emph{representation of $\CC$ in ${\sf A}$-modules} is a smooth functor $\CC\to \End(V)$ together with a fiberwise $A$-action on $V\to \CC_0$ such that the maps $s^*V\to t^*V$ are $A$-linear using the specified isomorphism $s^*A\to t^*A$. Such a representation in ${\sf A}$-modules is \emph{unitary} with respect to a super hermitian form on $V$ when the underlying representation $\CC\to \End(V)$ is unitary, and the $A$-action is \emph{self-adjoint}: any section of the bundle $A$ over $\CC_0$ defines a super self-adjoint endomorphism of $V$ in the sense of~\eqref{eq:supersa}. 
\end{defn}

Let $\Rep^{\sf A}(\CC)$ denote the category of unitary representations of $\CC$ in ${\sf A}$-modules. 

\subsection{Characters of representations of super Lie categories}\label{sec:chartheory}

Let $\CC_1^\times\subset \CC_1$ denote the sub-presheaf of invertible morphisms in $\CC$. Define the \emph{inertia groupoid} $\Lambda({\sf C})$ of a super Lie category ${\sf C}$ as a generalized super Lie groupoid whose objects are endomorphisms of objects in ${\sf C}$, i.e., the equalizer of the source and target maps,
$$
\Lambda({\sf C})_0:={\rm Eq}\left(s,t\colon {\sf C}_1\toto {\sf C}_0\right)
$$
and whose morphisms are the fibered product 
$$
\Lambda({\sf C})_1:={\rm Eq}\left(s,t\colon {\sf C}_1^\times\toto {\sf C}_0\right)\times_{{\sf C}_0} {\rm Eq}\left(s,t\colon {\sf C}_1\toto {\sf C}_0\right)
$$
consisting of an automorphism and an endomorphism of the same object. The source map is the obvious projection, and the target map conjugates the endomorphism by the automorphism. 

\begin{rmk}
In the examples of interest below, the equalizers and fibered products above turn out to be representable so that $\Lambda(\CC)$ is an honest super Lie groupoid rather than just a generalized one. Explicitly, the inertia groupoid will consist of closed super paths (i.e., super loops) and closed super annuli (i.e., super tori). 
\end{rmk}

\begin{rmk}
A bit more succinctly (but less concretely) we have
$$
\Lambda({\sf C}):=\underline{\sf Fun}^\times(\pt\sq \N,{\sf C})
$$
where $\pt\sq \N$ is the free Lie category on a single object and non-identity morphism, and $\underline{\sf Fun}^\times(\pt\sq \N,{\sf C})$ denotes the (generalized) super Lie groupoid whose objects are smooth functors and morphisms are smooth natural isomorphisms.
\end{rmk}

There is a variant of the inertia category that we'll need below that restricts further to non-identity endomorphisms in $\CC$. This shows up in our examples of interest as a restriction to super loops and super tori with strictly positive volume. With this in mind, we define the \emph{nondegenerate inertia groupoid} $\Lambda^\nd(\CC)\subset \Lambda(\CC)$ as the full subcategory with objects
$$
\Lambda^\nd({\sf C})_0:={\rm Eq}\left(s,t\colon ({\sf C}_1\setminus \CC_1^\times) \toto {\sf C}_0\right)
$$
where $\nd$ stands for \emph{nondegenerate} and $\CC_1\setminus \CC_1^\times$ consists of the sub presheaf of non-invertible morphisms. We again regard $\Lambda^\nd({\sf C})$ as a generalized super Lie category: its objects and morphisms need not be representable. 

A representation of ${\sf C}$ defines sections of the endomorphism bundles
\beq
\rho|_{\Lambda(\CC)_0} &\in& \Gamma(\Lambda(\CC)_0,{\rm Hom}(s^*V,t^*V)\cong \Gamma(\Lambda(\CC)_0,{\rm End}(s^*V))\nonumber
\eeq
where $s$ and $t$ are the restriction of the source and target of $\CC$ to $\Lambda(\CC)_0$, and~$V\to {\sf C}_0$ is the vector bundle defining the representation. Here we emphasize that ${\rm End}(V)$ denotes the usual endomorphism bundle of a vector bundle (not the frame category). 

\begin{defn}\label{defn:Liechar}
When it is defined, the \emph{character} of a representation of a super Lie category is the super trace ${\rm sTr}(\rho|_{\Lambda(\CC)_0})$
of the section of the endomorphism bundle determined by the restriction of $\rho$ to $\Lambda(\CC)_0$. By the cyclic property of the trace, this is invariant under the conjugation action by automorphisms in $\CC$, so descends to a function on the inertia groupoid
$$
{\rm sTr}(\rho|_{\Lambda(\CC)_0})\in C^\infty(\Lambda(\CC)).
$$
Furthermore, characters of isomorphic representations are equal, so the character is a functor
$$
\Rep({\sf C})\to C^\infty(\Lambda(\CC))
$$
viewing the target as a discrete category. When the context is clear, we will also call the restriction of ${\rm sTr}(\rho|_{\Lambda(\CC)_0})$ to the nondegenerate inertia groupoid the character of the representation. 
\end{defn}

\begin{ex} When ${\sf C}=\{G\toto \pt\}$ is the Lie groupoid associated to a Lie group, a representation of ${\sf C}$ is the same as a representation of $G$. The inertia groupoid is the adjoint quotient, $\Lambda(\pt\sq G)\cong G\sq G$, and the character of a representation $\rho$ of ${\sf C}$ reduces to the usual character as a smooth function on~$G$ invariant under conjugation. Note in this case $\Lambda^\nd(\pt\sq G)$ is the empty groupoid. 
\end{ex}

The restriction of an $A$-linear representation to $(\Lambda\CC)_0$ gives a section of the bundle of $A$-linear endomorphisms, ${\rm End}_A(V)$. To obtain a character, we need a \emph{super trace} on this bundle, meaning a bundle map from ${\rm End}_A(V)$ to a line bundle $\omega$ over~$\Lambda\CC$ that vanishes on commutators. 

\begin{defn} For a choice of super trace on ${\rm End}_A(V)$, the character of a representation of $\CC$ in $A$-modules is the supertrace of its restriction to $(\Lambda \CC)_0$. By the definition of a super trace, this descends to a section of a line bundle $\omega$ on~$\Lambda \CC$. When the context is clear, we will often call the restriction of this section to $\Lambda^\nd(\CC)$ the character. 
\end{defn}

In our applications, characters will often take values in a subalgebra of~$C^\infty(\Lambda^\nd({\sf C}))$. Typically this comes from characters taking values in functions on a groupoid ${\sf G}$ with the same objects as $\Lambda^\nd({\sf C})$, but more morphisms. 

\begin{defn}\label{defn:charfactor}
Fix a super Lie category ${\sf C}$ and an algebra $\Omega$ with an injective map $i\colon \Omega\hookrightarrow C^\infty(\Lambda^\nd {\sf C})$. Representations of ${\sf C}$ in ${\sf A}$-modules have \emph{$\Omega$-valued characters} if there is a map $\chi$ making the diagram commute
\beq
\begin{array}{c}
\begin{tikzpicture}
  \node (A) {$\Rep^{\sf A}(\CC)$};
  \node (B) [node distance= 4cm, right of=A] {$\Gamma(\Lambda^\nd(\CC);\omega)$};
  \node (C) [node distance = 1.5cm, above of=B] {$\Omega$};
  \draw[->] (A) to  (B);
  \draw[->,dashed] (A) to node [above] {$\chi$} (C);
  \draw[->,left hook-latex] (C) to node [right] {$i$} (B);
\end{tikzpicture}\end{array}\nonumber
\eeq
where the lower horizontal arrow is the character map on $\Rep^{\sf A}({\sf C})$. In this case, we also say that $\Rep^{\sf A}(\CC)$ has \emph{an $\Omega$-valued character theory}.
\end{defn}

\subsection{Super Lie categories that are natural in a manifold parameter}

Let $\mathcal{C}\colon {\sf Mfld}\to {\sf LieCat}$ be a functor from manifolds to (the 1-category of) super Lie categories, i.e., for each smooth manifold~$M$, $\mathcal{C}(M)$ is a super Lie category, a smooth map $f\colon M\to M'$ determines a smooth functor $\mathcal{C}(f)\colon \mathcal{C}(M)\to \mathcal{C}(M')$ and these compose strictly, $\mathcal{C}(f\circ g)=\mathcal{C}(f)\circ \mathcal{C}(g)$. On categories of representations, this yields pullback functors
$$
f^*\colon \Rep(\mathcal{C}(M'))\to \Rep(\mathcal{C}(M)).
$$
Hence, the assignment $M\mapsto \Rep(\mathcal{C}(M))$ (viewing the target as a groupoid) is a prestack. Furthermore, using that $f\colon M\to M'$ induces a smooth functor $\mathcal{C}(f)$, we obtain smooth functors between inertia categories
$$
\Lambda(\mathcal{C}(M))\to \Lambda(\mathcal{C}(M'))
$$
and hence the character map $\chi\colon \Rep(\mathcal{C}(M))\to C^\infty(\Lambda(\mathcal{C}(M)))$ can be promoted to a morphism of prestacks where the target is a discrete prestack, i.e., a presheaf. 

We observe that an algebra bundle on $\mathcal{C}(\pt)$ can be pulled back to an algebra bundle over~$\mathcal{C}(M)$, resulting in a prestack $M\mapsto \Rep^A(\mathcal{C}(M))$ of representations valued in $A$-modules. A choice of super trace over $\mathcal{C}(\pt)$ valued in a line~$\omega$ over $\Lambda(\mathcal{C}(\pt))$ results in a super trace valued in the pullback of $\omega$ to $\Lambda(\mathcal{C}(M))$. 

In our examples of interest, the categories $\Rep^A(\mathcal{C}(M))$ turn out to have a character theory with more symmetry that is also natural in $M$ in the sense of a presheaf of algebras, leading to the following refinement of Definition~\ref{defn:charfactor}. In the following, suppose we have fixed a bundle of algebras ${\sf A}\to \mathcal{C}(\pt)$ with a super trace on $A$-modules (as described above).

\begin{defn}\label{defn:charfactor2}
Let $\Omega\colon {\sf Mfld}^\op\to {\sf Alg}$ be a presheaf of algebras, $\mathcal{C}$ a prestack of Lie categories, and $i\colon \Omega\to C^\infty(\Lambda{\sf C}(-))$ a morphism of presheaves of algebras. Representations of $\mathcal{C}(M)$ in $A$-modules have \emph{$\Omega$-valued characters} if there is a morphism $\chi$ of prestacks making the diagram commute
\beq
\begin{array}{c}
\begin{tikzpicture}
  \node (A) {$\Rep^{\sf A}(\mathcal{C}(M))$};
  \node (B) [node distance= 4cm, right of=A] {$\Gamma(\Lambda^\nd(\mathcal{C}(M));\omega)$};
  \node (C) [node distance = 1.5cm, above of=B] {$\Omega(M)$};
  \draw[->] (A) to  (B);
  \draw[->,dashed] (A) to node [above] {$\chi$} (C);
  \draw[->,left hook-latex] (C) to node [right] {$i$} (B);
\end{tikzpicture}\end{array}\nonumber
\eeq
where the lower horizontal arrow is the character map on $\Rep^{\sf A}(\mathcal{C}(M))$.
\end{defn}

\subsection{Grothendieck groups of representations}\label{sec:groth}

When the assignment $M\mapsto \Rep(\mathcal{C}(M))$ is a stack, concordance classes of representations provide a set-valued topological invariant of~$M$ (see~\S\ref{appen:conc}). The operations of direct sum~$\oplus$ and parity reversal~$\Pi$ can be used to turn this into a group-valued invariant via a Grothendieck group. 

\begin{defn}[Grothendieck group] \label{defn:diffGG}
Let $\mathcal{F}(\Rep(\mathcal{C}(M))$ denote the free abelian group on concordance classes of representations of $\mathcal{C}(M)$. Define $\mathcal{Z}(\Rep(\mathcal{C}(M)))$ to be the subgroup generated by elements of the form
\beq
\rho+ \rho'-(\rho\oplus\rho')\qquad \rho\oplus \Pi \rho.\label{eq:Grothequiv}
\eeq
Then the \emph{Grothendieck group} of $\Rep(\mathcal{C}(M))$ is 
$$
\K(\Rep(\mathcal{C}(M)):=\mathcal{F}(\Rep(\mathcal{C}(M))/\mathcal{Z}(\Rep(\mathcal{C}(M)).
$$
The tensor product of representations endows this Grothendieck group with a ring structure. 
\end{defn}

For representations with an $\Omega$-valued character theory, there is a differential refinement of~$\K(\Rep(\mathcal{C}(M)))$ that mimics the definition of Hopkins--Singer differential cocycles~\cite{HopSing}. This requires one fix a natural character map
\beq
\chi\colon \Rep(\mathcal{C}(-)) \to \Omega(-).\label{eq:genchar}
\eeq

\begin{defn}[Differential cocycles and concordance]\label{defn:diffconc}
Given a character map $\chi$ as in~\eqref{eq:genchar}, the prestack of \emph{differential cocycles} is
$$
\widehat{\Rep}(\mathcal{C})(M):=\{\rho\in \Rep(\mathcal{C}(M)), \alpha\in \Omega(M\times \R), \mid i_0^*\alpha=\chi(\rho)\}.
$$
Define the \emph{character} of a differential cocycle to be $\alpha_1=i^*_1\alpha$ (the target of the concordance $\alpha$), and let $\widehat{\chi}\colon \widehat{\Rep}(\mathcal{C})(M)\to \Omega(M)$ be the character map.

A \emph{differential concordance} is a concordance $(\widetilde\rho,\widetilde\alpha)\in \widehat{\Rep}(\mathcal{C})(M\times \R)$ for $\widetilde{\rho}\in \Rep(\mathcal{C}(M\times \R))$ and $\widetilde{\alpha}\in \Omega(M\times\R^2)$ with the property that that $i_1^*\widetilde\alpha=p^*\alpha_1$ is the constant concordance for $i_1\colon M\times \R\times\{1\}\hookrightarrow M\times \R^2$ the inclusion, $p\colon M\times \R\to M$ the projection, and $\alpha_1\in \Omega(M)$.
\end{defn}

Note that the source and target of a differential concordance are differential cocycles with the same character. Using direct sum $\oplus$ and parity reversal $\Pi$, we define a differential Grothendieck group. 

\begin{defn}[Differential Grothendieck group] \label{defn:Diffgroth}
Suppose $\Rep(\mathcal{C}(-))$ and $\Omega(-)$ are stacks, and we are given a character map~\eqref{eq:genchar}. Let $\mathcal{F}(\widehat{\Rep}(\mathcal{C}(M)))$ denote the free abelian group on differential concordance classes. Define $\mathcal{Z}(\widehat{\Rep}(\mathcal{C}(M))$ to be the subgroup generated by elements of the form
$$
(\rho,\alpha)+(\rho',\alpha')-(\rho\oplus\rho',\alpha+\alpha')\qquad (\rho\oplus \Pi \rho,0).
$$
Then the \emph{differential Grothendieck group} of $\Rep(\mathcal{C}(M))$ is 
$$
\K(\widehat{\Rep}(\mathcal{C}(M))):=\mathcal{F}(\widehat{\Rep}(\mathcal{C}(M)))/{\mathcal{Z}}(\widehat{\Rep}(\mathcal{C}(M))).
$$
If the character map~$\chi$ sends tensor products of representations to products of characters, this differential Grothendieck group has the structure of a ring with product $(\rho,\alpha)\cdot (\rho',\alpha')=(\rho\otimes \rho',\alpha\alpha').$
\end{defn}

We can extend the character map $\widehat\chi$ additively to $\mathcal{F}(\widehat{\Rep}(\mathcal{C}(M)))$. By basic properties of the super trace, $ {\mathcal{Z}}(\widehat{\Rep}(\mathcal{C}(M)))$ contains differential cocycles in the kernel of $\chi$. Hence, we get a well-defined map
$$
\widehat{\chi}\colon \K(\widehat{\Rep}(\mathcal{C}(M)))\to \Omega(M) \qquad [\rho,\alpha]\mapsto i^*_1\alpha.
$$

\section{Super path representations, Quillen super connections and K-theory}\label{sec:Quillenconn}

In this section, we study a super Lie category $\Path(M)$ of constant super Euclidean paths in~$M$. Orientation reversal on super paths leads to a definition of a unitary representation of~$\Path(M)$. The key computation is the following. 

\begin{thm}\label{thm:11unitaryreps} The category $\Rep^{\Cl_n}(\Path(M))$ of finite-dimensional $\Cl_n$-linear unitary representations of $\Path(M)$ is equivalent to hermitian super vector bundles with self-adjoint fiberwise $\Cl_n$-action and $\Cl_n$-linear unitary Quillen super connection on~$M$. This equivalence is natural in~$M$. 
\end{thm}

The correspondence is explicit, with the representation of $\Path(M)$ associated to a super connection $\A$ given by the super semigroup action 
$$
\rho\colon \R^{1|1}_{\ge 0}\times \Omega^\bullet(M;V)\to \Omega^\bullet(M;V)\qquad  \rho(t,\theta)=e^{-t\A^2+\theta\A}
$$
on $\Omega^\bullet(M;V)$. Unitarity of the super connection then corresponds to unitarity of~$\A$, and Clifford linearity of the representation corresponds to a Clifford action on~$V$ and Clifford linearity of~$\A$. The following is an easy consequence.

\begin{cor}\label{cor:Kthy}
 The Grothendieck group of $\Rep(\Path(M))$ is the K-theory of $M$. 
 \end{cor}

\subsection{Motivation: effective supersymmetric mechanics and K-theory}\label{sec:Kthymot}

Before jumping into our model for K-theory, we explain how one might arrive at these constructions from applying methods of effective field theory to $1|1$-dimensional supersymmetric quantum mechanics. The main ideas are: (1) choosing a cutoff in quantum mechanics leads to a finite-dimensional space of states, and (2) underlying a supersymmetric quantum mechanical system there is a K-theory class that can be computed in terms of a choice of cutoff, but is independent of the choice. This follows ideas of Kitaev~\cite{Kitaev}, while also leveraging the standard supersymmetric cancellation argument.

For now we will be vague about our precise notion of $1|1$-dimensional field theory, and instead work with a prototypical example: an even-dimensional compact spin manifold determines a quantum mechanical system whose space of states is sections of the spinor bundle $\Gamma(\$^+\oplus\$^-)$, and whose time evolution operator is~$\exp(-t\slashed{D}{}^2)$ for $\slashed{D}$ the self-adjoint, odd Dirac operator. 

The $\lambda$-eigenspace of $\slashed{D}{}^2$ consists of the \emph{energy $\lambda$ states}, $V_\lambda\subset \Gamma(\$^+\oplus\$^-)$. A \emph{cutoff} theory considers (the linear span of) states with energy less than a chosen~$\lambda\in \R_{>0}$, giving a subspace~$V_{<\lambda}\subset \Gamma(\$^+\oplus\$^-)$. Since $M$ is compact, this subspace of states is finite-dimensional. The restriction of the time-evolution operator to~$V_{<\lambda}$ gives a finite-dimensional quantum system, which is a \emph{cutoff theory}. 

The difference between spaces of states for cutoffs~$\lambda$ and~$\lambda'$ is the addition of a finite-dimensional vector space $V_{[\lambda,\lambda')}\subset\Gamma(\$^+\oplus\$^-)$ on which~$\slashed{D}$ is \emph{invertible}. Since $\slashed{D}$ is odd, its invertibility gives an isomorphism $V_{[\lambda,\lambda')}^\ev\to V_{[\lambda,\lambda')}^\odd$ between the even and odd subspaces of $V_{[\lambda,\lambda')}$. In particular, varying the cutoff amounts to taking a direct sum $V_{<\lambda}\oplus V_{[\lambda,\lambda')}\cong V_{<\lambda'}$ with a vector space of super dimension zero. Using the standard identification between $\Z/2$-graded vector bundles and virtual vector bundles, this operation is the familiar stabilization in K-theory. Therefore, the original quantum theory determines an underlying K-theory class that can be represented by \emph{any} cutoff theory, and varying the cutoff ranges through these representatives. 

The above ideas can also be applied in families. For example, from a bundle of spin manifolds with base~$M$, the fiberwise spinor bundle and Dirac operator determine an $M$-family of state spaces and time-evolution operators. If there is a $\lambda>0$ not in the spectrum of the square of any Dirac operator in the family, then~$\lambda$ gives a cutoff and an associated family of effective field theories. Different cutoffs will differ by a vector bundle on which the Dirac operators are invertible, and hence there is a single underlying K-theory class for any family of effective theories extracted from the initial data. In fact, (see Lemma~\ref{lem:FreedLott}) \emph{any} family of Dirac operators has a finite-dimensional subbundle of the fiberwise spinors containing the kernel of the family of Dirac operators. We think of this subbundle as defining a family of cutoff theories over~$M$. Going the other way, if $V=W\oplus \Pi W$ then the identity map from $W$ to $\pi W$ defines an invertible odd linear map that can be incorporated as the degree zero part of a super connection, allowing us to regard this subspace (possible after a deformation) as coming from a different choice of cutoff. 

The ethos of effective field theory is that one never need work directly with the infinite-dimensional objects. Instead we analyze finite-dimensional cutoff versions and relations between different cutoffs that modify the space of states by stabilization. 
We formalize these ideas via representations of super paths in~$M$. Such a representation determines a vector bundle over~$M$, and we think of the fiber at each point in~$M$ as a (finite-dimensional) state space of a quantum system. The super path representation also determines a (super) time-evolution operator on each vector space, which is the analog of the cutoff of $\exp(-t\slashed{D}^2)$ from above. The reason to take super paths rather than all paths is exactly the existence of the odd square root of the generator of time-evolution, which is crucial for identifying different choices of cutoff as different representatives of the same class in K-theory. 

These ideas will carry over directly to the $2|1$-dimensional case, where at each weight of an $S^1$-action from rotating annuli we choose a cutoff.

\subsection{Super Euclidean paths}

Define the \emph{$1|1$-dimensional super (Euclidean) translation group} $\E^{1|1}$ to be the super manifold $\R^{1|1}$ endowed with the multiplication
\beq
(t,\theta)\cdot (t',\theta')=(t+t'+\theta\theta',\theta+\theta'), \quad (t,\theta),(t',\theta')\in \R^{1|1}(S),\label{eq:E11}
\eeq
There is an obvious left action of super translations $\E^{1|1}$ on $\R^{1|1}$. 

An \emph{$S$-family of super Euclidean paths in $M$} is an $S$-point $(t,\theta)\in \Rge(S)$ and a map $\gamma\colon S\times \R^{1|1}\to M$. The \emph{source super point} of this super path is the composition
$$
S\times \R^{0|1}\stackrel{\iota}{\to} S\times \R^{1|1}\stackrel{\gamma}{\to} M
$$
where $\iota$ is determined by the standard inclusion $\R^{0|1}\hookrightarrow \R^{1|1}$. The \emph{target super point} of the super path is
$$
S\times \R^{0|1}\stackrel{\iota}{\to} S\times \R^{1|1}\stackrel{T_{t,\theta}}{\longrightarrow} S\times \R^{1|1}\stackrel{\gamma}{\to} M
$$
where $T_{t,\theta}$ is the translation action on $S\times \R^{1|1}$ by the given $S$-point $(t,\theta)\in \R^{1|1}_{\ge 0}(S)$. A super path is \emph{constant} if the map $\gamma$ is invariant under the precomposition action of $\E^1<\E^{1|1}$. 

\begin{defn} The \emph{presheaf of super paths in $M$,} denoted ${\rm sP}(M)$, is the presheaf whose value at~$S$ is the set of pairs $(t,\theta)\in \Rge(S)$ and $\gamma\colon S\times \R^{1|1}\to M$. The \emph{presheaf of constant super paths in $M$}, denoted ${\rm sP}_0(M)$, is the sub-presheaf where $\gamma$ is a constant super path.  \end{defn}

The source and target super points of a super path give morphisms of presheaves $s,t\colon {\rm sP}(M)\toto \SM(\R^{0|1},M)$ and $s,t\colon {\rm sP}_0(M)\toto \SM(\R^{0|1},M)$. 

\begin{lem}\label{lem:constantsuperpath}
The presheaf of constant super paths in $M$ is representable, ${\rm sP}_0(M)\stackrel{\sim}{\to} \Rge\times \SM(\R^{0|1},M)$, with the isomorphism determined by the super length map ${\rm sP}_0(M)\to \Rge$ and the source super point map, $s\colon {\rm sP}_0(M)\to \SM(\R^{0|1},M)$.
\end{lem}
\bp
An $\E^1$-invariant map $\gamma\colon S\times \R^{1|1}\to M$ can be identified with the composition
$$
\gamma\colon S\times \R^{1|1}\stackrel{\rm pr}{\to} S\times \R^{1|1}/\E^1\cong S\times \R^{0|1}\stackrel{\gamma_0}{\to} M. 
$$
Hence, we can identity a constant super path with an $S$-point $(t,\theta)\in \Rge$ and an $S$-point $\gamma_0\in \SM(\R^{0|1},M)(S)$. Since ${\rm pr}\circ \iota=\id_{S\times \R^{0|1}}$, $\gamma_0$ is the source super point as claimed. 
\ep

\begin{rmk} Below we will denote an $S$-point of $\SM(\R^{0|1},M)$ by the pair $(x,\psi)$. Somewhat loosely, $x\colon S\to M$ is an ordinary $S$-point of $M$, and $\psi\in \Gamma(S,x^*\Pi TM)$ is a section of the pullback odd tangent bundle. More precisely, $x$ corresponds to a super algebra map $x^*\colon C^\infty(M)\to C^\infty(S)$ and $\psi$ is an odd derivation $\psi^*\colon C^\infty(M)\to C^\infty(S)$ with respect to~$x^*$. Then $x^*+\theta\psi^*\colon C^\infty(M)\to C^\infty(S)[\theta]\cong C^\infty(S\times \R^{0|1})$ determines a map of super manifolds $x+\theta\psi\colon S\times \R^{0|1}\to M$. \end{rmk}

\subsection{An orientation-reversing automorphism}
There is an orientation-reversing automorphism of~$\R^{1|1}$ given by
$$
{\sf or}\colon \R^{1|1}\to \R^{1|1} \qquad (t,\theta)\mapsto (-t,i\theta)\qquad (t,\theta)\in \R^{1|1}(S). 
$$
We explain how to promote ${\sf or}$ to an orientation-reversing action on super paths that exchanges the source and target super points. Suppose we are given an input super path determined by $(t,\theta)\in \Rge(S)$ and $\gamma\colon S\times \R^{1|1}\to M$. Applying ${\sf or}$ to $S\times \R^{1|1}$, we get a new pair of inclusions 
\begin{equation}
\begin{array}{c}
\begin{tikzpicture}
  \node (A) {$S\times \R^{0|1}$};
\node (B) [node distance=4cm, right of=A] {$S\times \R^{1|1}$};
\node (C) [node distance=3cm, right of=B] {$M.$};
\draw[->] (A) to[bend left=10] node [above] {$\iota$} (B);
\draw[->] (A) to[bend right=10] node [below] {$T_{-t,i\theta}\circ \iota$} (B);
\draw[->] (B) to node [above] {$\gamma\circ {\sf or}^{-1}$} (C);
\end{tikzpicture}
\end{array}\nonumber
\end{equation}
To turn this data into a super path in the sense of the previous subsection, we translate $S\times \R^{1|1}$ by $T_{t,-i\theta}$ and get
\begin{equation}
\begin{array}{c}
\begin{tikzpicture}
  \node (A) {$S\times \R^{0|1}$};
\node (B) [node distance=4cm, right of=A] {$S\times \R^{1|1}$};
\node (C) [node distance=4cm, right of=B] {$M$};
\draw[->] (A) to[bend left=10] node [above] {$T_{t,-i\theta}\circ \iota$} (B);
\draw[->] (A) to[bend right=10] node [below] {$\iota$} (B);
\draw[->] (B) to node [above] {$\gamma\circ {\sf or}^{-1}\circ T^{-1}_{t,-i\theta}$} (C);
\end{tikzpicture}
\end{array}\nonumber
\end{equation}
We observe the super length of this new super path is $(t,-i\theta)\in \Rge(S)$. 

\begin{defn} Define the \emph{time-reversal} automorphism ${\rm sP}(M)\to {\rm sP}(M)$ by the map on $S$-points
$$
(t,\theta,\gamma)\mapsto (t,-i\theta,\gamma\circ {\sf or}^{-1}\circ T^{-1}_{t,-i\theta}),\qquad (t,\theta)\in \Rge(S), \ \ \gamma\colon S\times \R^{1|1}\to M.
$$
\end{defn}

The following is a simple calculation. 

\begin{lem}\label{lem:11sigma} The restriction of the time-reversal automorphism to ${\rm sP}_0(M)$ is determined by the formula
$$
(t,\theta,x,\psi)\mapsto (t,-i\theta,x+\theta\psi,-i\psi) \qquad (t,\theta)\in \Rge(S), \ \ (x,\psi)\in \SM(\R^{0|1},M)(S). 
$$
under the isomorphism ${\rm sP}_0(M)\cong \Rge\times \SM(\R^{0|1},M)$.
\end{lem}

\subsection{The super Lie category of constant super paths}\label{sec:11cat}

Consider a pair of super paths 
\begin{equation}
\begin{array}{c}
\begin{tikzpicture}
  \node (A) {$S\times \R^{0|1}$};
  \node (B) [node distance= 4cm, right of=A] {$S\times \R^{1|1}$};
  \node (C) [node distance = 3cm, right of=B] {$M$};
  \node (D) [node distance = 1.5cm, below of=A] {$S\times \R^{0|1}$};
  \node (E) [node distance = 1.5cm, below of=B] {$S\times \R^{1|1}$};
  \node (F) [node distance = 1.5cm, below of=C] {$M$};
  \draw[->] (A) to[bend right=10] node [below] {$T_{t,\theta}\circ \iota$} (B);
    \draw[->] (A) to[bend left=10] node [above] {$\iota$} (B);
  \draw[->] (B) to node [above] {$\gamma$} (C);
  \draw[->] (D) to[bend right=10] node [below] {$T_{t',\theta'}\circ \iota$} (E);
    \draw[->] (D) to[bend left=10] node [above] {$\iota$} (E);
  \draw[->] (E) to node [below] {$\gamma'$} (F);
\end{tikzpicture}\end{array}\nonumber
\end{equation}
for which the target super point of the first super path is  the same as the source super point of the second, i.e.,~$\gamma\circ T_{t,\theta}\circ \iota=\gamma'\circ \iota$, or equivalently, $\gamma\circ \iota=T_{t,\theta}^{-1}\circ \gamma'\circ \iota$. For such paths to have a smooth concatenation, we require an equality of maps $S\times \R^{1|1}\to M$ in a neighborhood of these candidate gluing points. Explicitly, we translate the second super path by $T_{t,\theta}$ to get
\begin{equation}
\begin{array}{c}
\begin{tikzpicture}
  \node (A) {$S\times \R^{0|1}$};
  \node (B) [node distance= 4cm, right of=A] {$S\times \R^{1|1}$};
  \node (C) [node distance = 3cm, right of=B] {$M$};
  \draw[->] (A) to[bend right=10] node [below] {$T_{t,\theta}\circ T_{t',\theta'}\circ \iota$} (B);
    \draw[->] (A) to[bend left=10] node [above] {$T_{t,\theta}\circ \iota$} (B);
  \draw[->] (B) to node [above] {$ \gamma'\circ T_{t,\theta}^{-1}$} (C);
\end{tikzpicture}\end{array}\nonumber
\end{equation}
and we require that $\gamma$ and $\gamma'\circ T^{-1}_{t,\theta}$ agree when restricted to $S\times (t-\epsilon,t+\epsilon)^{1|1}\subset S\times \R^{1|1}$ where $(t-\epsilon,t+\epsilon)^{1|1}$ is the restriction of the structure sheaf of~$\R^{1|1}$ to $(t-\epsilon,t+\epsilon)\subset \R$. If this is the case, the concatenation is 
\begin{equation}
\begin{array}{c}
\begin{tikzpicture}
  \node (A) {$S\times \R^{0|1}$};
  \node (B) [node distance= 4cm, right of=A] {$S\times \R^{1|1}$};
  \node (C) [node distance = 3cm, right of=B] {$M$};
  \draw[->] (A) to[bend right=10] node [below] {$T_{t+t'+\theta\theta',\theta+\theta'}\circ \iota$} (B);
    \draw[->] (A) to[bend left=10] node [above] {$\iota$} (B);
  \draw[->] (B) to node [above] {$\gamma'*\gamma$} (C);
\end{tikzpicture}\end{array}\nonumber
\end{equation}
where $T_{t,\theta}\circ T_{t',\theta'}=T_{t+t'+\theta\theta',\theta+\theta'}$ and $\gamma'*\gamma$ is the map whose restriction to $(-\infty,t+\epsilon)^{1|1}$ is $\gamma$ and whose restriction to $(t-\epsilon,\infty)^{1|1}$ is $\gamma'\circ T^{-1}_{t,\theta}$. 

Concatenation is only defined for super paths that agree in a neighborhood of a gluing point, so this only determines a partially defined composition on $\{{\rm sP}(M)\toto \SM(\R^{0|1},M)\}$. It takes some work to promote this to an honest category of super paths, e.g., the use of collars in~\cite{ST11}. However, for constant super paths concatenation is always defined. Indeed, when $\gamma$ and $\gamma'\circ T^{-1}_{t,\theta}$ factor through $S\times \R^{0|1}$, if $\gamma\circ \iota=T_{t,\theta}^{-1}\circ \gamma'\circ \iota\colon S\times \R^{0|1}\to M$ then $\gamma=\gamma'\circ T^{-1}_{t,\theta}\colon S\times \R^{1|1}\to M$. In particular, these super paths agree on $S\times (t-\epsilon,t+\epsilon)^{1|1}$, and we have that $\gamma'*\gamma=\gamma$.

\begin{defn} 
Define the category of \emph{constant super paths in $M$} as
$$
\Path(M):=\left(\begin{array}{c}{\rm sP}_0(M) \\ \downarrow \downarrow \\ \SM(\R^{0|1},M)\end{array}\right)\cong \left(\begin{array}{c} \Rge\times \SM(\R^{0|1},M) \\ \downarrow \downarrow \\ \SM(\R^{0|1},M)\end{array}\right),
$$
where objects are $S$-families of super points in~$M$, and morphisms are constant super paths in $M$. The source and target maps take the source and target super point of a constant super path. Explicitly, the source map is the projection and the target map on $S$-points is
$$
{\rm target}(t,\theta,x,\psi)=(x+\theta\psi,\psi)\qquad (t,\theta)\in \R^{1|1}_{\ge 0}(S), \ (x,\psi)\in \SM(\R^{0|1},M)(S). 
$$ 
The unit section picks out the constant super path associated with $(0,0)\in \Rge$. Composition is concatenation of constant super paths, which is determined by the restriction of the group structure~\eqref{eq:E11} to $\R^{1|1}_{\ge 0}\subset \R^{1|1}$. 
\end{defn}

\begin{lem}\label{lem:subgrpd1} The time-reversal map extends to a smooth functor ${\sf or}\colon \Path(M)\to \Path(M)^\op$ whose map on morphisms is the one from Lemma~\ref{lem:11sigma}, and whose map on objects is 
\beq
&&\SM(\R^{0|1},M)\to \SM(\R^{0|1},M)\quad (x,\psi)\mapsto (x,-i\psi),\ (x,\psi)\in \SM(\R^{0|1},M)(S). \label{eq:idegmap}
\eeq
\end{lem}
\bp 
The definition of ${\sf or}_0$ along with Lemma~\ref{lem:11sigma} shows that ${\sf or}$ is compatible with source and target maps, since
\beq
{\rm source}(t,-i\theta,x+\theta\psi,-i\psi)&=&(x+\theta\psi),\nonumber\\
 {\rm target}(t,-i\theta,x+\theta\psi,-i\psi)&=&(x+\theta\psi+(-i\theta)(-i\psi),-i\psi)=(x,-i\psi). \nonumber
\eeq
We observe that $\Rge\to (\Rge)^\op$ given by $(t,\theta)\mapsto (t,-i\theta)$ is a homomorphism of super semigroups. This implies that ${\sf or}$ is compatible with units and composition, and so defines a functor. 
\ep

\subsection{Unitary and Clifford linear representations of $\Path(M)$}

Complex conjugation on $C^\infty(\SM(\R^{0|1},M))$ and $C^\infty(\R^{1|1}_{\ge 0})$ come from natural isomorphisms
\beq
C^\infty(\SM(\R^{0|1},M))&\cong& C^\infty(\Pi TM)\cong \Omega^\bullet(M)=\Omega^\bullet(M;\C)\cong \Omega^\bullet(M;\R)\otimes_\R \C.\nonumber\\
C^\infty(\Rge)&\cong& C^\infty(\R_{\ge 0})[\theta]\cong (C^\infty(\R_{\ge 0};\R)\otimes_\R \C)[\theta].\nonumber
\eeq
This defines a real structure ${\sf r}\colon \Path(M)\to \overline{\Path(M)}$, and we consider the composite
\beq
\sigma\colon \Path(M)\stackrel{{\sf or}}{\to} \Path(M)^\op\stackrel{{\sf r}}{\to} \overline{\Path(M)}^{\op}.\nonumber\label{eq:11unitaryinv}
\eeq
By inspection, $\sigma$ defines an anti-involution of $\Path(M)$. 

\begin{defn} \label{defn:11unitary} A \emph{unitary representation} of $\Path(M)$ is a unitary representation with respect to the functor $\sigma$ in the sense of Definition~\ref{defn:unitaryrep}. 
\end{defn}

Now we turn to Clifford linear representations. 

\begin{defn} 
Let ${\sf Cl}_n\to \pt={\rm Ob}(\Path(\pt))$ be the trivial algebra bundle with fiber $\Cl_n$, and take the identity algebra isomorphism $s^*\Cl_n\cong t^*\Cl_n$ over ${\rm Mor}(\Path(\pt))\cong \R^{1|1}_{\ge 0}$, where we use that $s=t$ is the projection. We also let ${\sf Cl}_n$ denote the pullback of this algebra bundle and algebra isomorphisms to $\Path(M)$. 
\end{defn}


We observe that the notion of a self-adjoint Clifford module (e.g.,~\cite{BGV}) coincides with the notion of a self-adjoint ${\sf Cl}_n$-module from~\S\ref{sec:Amod}: when translating the condition to ordinary inner product spaces, the action by the generators of the Clifford algebra is through skew-adjoint operators. 

\begin{defn} A \emph{Clifford-linear representation} of $\Path(M)$ is a unitary representation in self-adjoint ${\sf Cl}_n$-modules, $\Hom(\null_{{\sf Cl}_n}V,\null_{{\sf Cl}_n}V)$ using the (identity) isomorphisms between algebra bundles $s^*{\Cl}_n\to t^*{\Cl}_n$ specified above. Let $\Rep^{\Cl_n}(\Path(M))=:\Rep^n(\Path(M))$ denote the category of $\Cl_n$-linear representations. 
\end{defn}

The Morita equivalences $\Cl(2n)\simeq \C$ and $\Cl(2n+1)\simeq \Cl(1)$ give equivalences of categories, 
$$
\Rep^{2n}(\Path(M))\simeq \Rep^0(\Path(M))=\Rep(\Path(M))\qquad \Rep^{2n+1}(\Path(M))\simeq \Rep^1(\Path(M)),
$$
with explicit functors gotten by tensoring a vector bundle over~$M$ on which a representation is defined with a bimodule implementing the Morita equivalence.

\subsection{The proof of Theorem~\ref{thm:11unitaryreps}}
In the next few lemmas we characterize unitary representations of $\Path(M)$. 

\begin{lem}\label{lem:semigroup}
A representation of $\Path(M)$ is determined by a super semigroup representation $\rho(t,\theta)\colon \R^{1|1}_{\ge 0}\to {\rm End}(\Omega^\bullet(M;V))$ for $V\to M$ a super vector bundle, with the additional condition that
$$
\rho(t,\theta)(f\cdot v)=(f+\theta df)\rho(t,\theta)(v)\qquad v\in \Omega^\bullet(M;V), \ f\in \Omega^\bullet(M), 
$$ 
for the $\Omega^\bullet(M)$-module structure on $\Omega^\bullet(M;V)$. An isomorphism between representations of $\Path(M)$ is determined by an element of $\Omega^\bullet(M;{\rm Hom}(V,W))^\times$ that intertwines the semigroup representations of $\Rge$. 
\end{lem}
\bp
Let $V'\to \SM(\R^{0|1},M)$ be a super vector bundle and $\rho$ be a representation on $\End(V')$. Trivializing this bundle along the odd fibers of the projection $\SM(\R^{0|1},M)\cong \Pi TM\to M$ is an isomorphism
$$
\Gamma(\SM(\R^{0|1},M);V')\cong \Omega^\bullet(M;i^*V').
$$
for $i\colon M\hookrightarrow \SM(\R^{0|1},M)$ the canonical inclusion of the reduced manifold. Let $V=i^*V'$ denote the resulting super vector bundle over~$M$. 
The pullback along the source map $s=p_2\colon \Rge\times \SM(\R^{0|1},M)\to \SM(\R^{0|1},M)$ specifies an isomorphism
$$
\Gamma({\rm Mor}(\Path(M));s^*V')=\Gamma(\Rge\times \SM(\R^{0|1},M),s^*V')\cong C^\infty(\Rge)\otimes \Omega^\bullet(M;V),
$$
where the right hand side has the obvious $C^\infty({\rm Mor}(\Path(M))\cong C^\infty(\Rge)\otimes \Omega^\bullet(M)$-module structure. 
We similarly obtain an isomorphism
$$
\Gamma({\rm Mor}(\Path(M));t^*V')=\Gamma(\Rge\times \SM(\R^{0|1},M);t^*V')\cong  C^\infty(\Rge)\otimes \Omega^\bullet(M;V),
$$
as vector spaces, but the $C^\infty(\Rge)\otimes \Omega^\bullet(M)$-module structure is twisted by the algebra automorphism 
$$
C^\infty(\Rge)\otimes \Omega^\bullet(M)\to C^\infty(\Rge)\otimes \Omega^\bullet(M),\quad g(t,\theta)\mapsto g(t,\theta), \quad f\mapsto f+\theta df,
$$
for $g\in C^\infty(\Rge)$ and $f\in \Omega^\bullet(M)\cong C^\infty(\SM(\R^{0|1},M))$. 

With the above identifications in place, a representation $\rho$ is a map of $C^\infty(\Rge)\otimes \Omega^\bullet(M)$-modules,
$$
C^\infty(\Rge)\otimes \Omega^\bullet(M;V) \to C^\infty(\Rge)\otimes \Omega^\bullet(M;V),
$$
where the target has the aforementioned twisted $\Omega^\bullet(M)$-module structure. Explicitly, $v\in C^\infty(\Rge)\otimes \Omega^\bullet(M;V)$, $g\in C^\infty(\Rge)$, and $f\in \Omega^\bullet(M)$, we have
\beq
\rho(g\cdot v)=g\rho(v),\quad \rho(f\cdot v)=(f+\theta df)\rho(v). \label{eq:R11linearmod}
\eeq
The crucial point is that $\rho$ is linear over~$C^\infty(\Rge)$, so determines a map between (trivial) vector bundles on~$\Rge$ with fiber $\Omega^\bullet(M;V)$. In this repackaging,~$\rho$ is a function $\rho(t,\theta)$ on~$\Rge$ with values in module maps $\Omega^\bullet(M;V)\to \Omega^\bullet(M;V)$. This function must also satisfy the second condition in~\eqref{eq:R11linearmod} and be compatible with composition,
$$
\rho(t,\theta)\circ \rho(t',\theta')=\rho(t+t'+\theta\theta',\theta+\theta'). 
$$
Compatibility with composition shows we get a semigroup representation $\rho(t,\theta)\colon \R^{1|1}_{\ge 0}\to {\rm End}(\Omega^\bullet(M;V))$. The second condition in~\eqref{eq:R11linearmod} is the additional condition in the statement of the lemma. 

As for isomorphisms of representations, by definition these are isomorphisms of super vector bundles $V'\to W'$ over $\SM(\R^{0|1},M)$ compatible with the representations. With our identifications on modules of sections in place, this is an isomorphism of $\Omega^\bullet(M)$-modules
$$
\Omega^\bullet(M;V)\to \Omega^\bullet(M;W)
$$
for $W=i^*W'$. But being $\Omega^\bullet(M)$-linear means this is equivalent data to a section of the bundle of fiberwise maps, i.e., an invertible element of $\Omega^\bullet(M;{\rm Hom}(V,W))$ as claimed. 
\ep

\begin{lem} \label{lem:superconnect}
The semigroup representation from Lemma~\ref{lem:semigroup} (and hence a representation of $\Path(M)$) is determined by the formula 
$$\rho(t,\theta)=e^{-t\A^2+\theta\A}$$
where $\A$ is a super connection viewed as an odd derivation, $\A\colon \Omega^\bullet(M;V)\to \Omega^\bullet(M;V)$ over $\Omega^\bullet(M)$. 
Isomorphisms between representations of $\Path(M)$ are in bijection with super connection preserving isomorphisms of super vector bundles. 
 \label{lem:11geo} \end{lem}

\begin{proof}[Proof of~Lemma~\ref{lem:11geo}]  

We start by restricting $\rho$ to the subspace
$$
(\Lambda(\Path(M)))_0=\R_{\ge 0}\times \SM(\R^{0|1},M)\subset \Rge\times \SM(\R^{0|1},M)=(\Path(M))_0
$$ 
on which source and target maps are both the projection to $\SM(\R^{0|1},M)$. By the previous lemma, this restriction of~$\rho$ is determined by a semigroup representation of $\R_{\ge 0}$ on $\Omega^\bullet(M;V)$ that is $\Omega^\bullet(M)$-linear, i.e., a section of the endomorphism bundle, $\Omega(M;{\rm End}(V))$. By the existence and uniqueness of solutions to ordinary differential equations we obtain a generator~$H$,
$$
\rho(t,0)= e^{-tH},\quad H\in \Omega^\bullet(M,{\rm End}(V)), \quad t\in \R_{\ge 0}.
$$
Returning to the representation on the whole of $\Rge\times \SM(\R^{0|1},M)$, we have
$$
\rho(t,\theta)=e^{-tH}(1+\theta \A),
$$
where~$\A$ is an odd linear map, and we have used that $(t,\theta)=(t,0)\cdot (0,\theta)$ in the semigroup $\Rge$. Compatibility with composition further demands that 
$$
e^{-tH}(1+\theta \A)e^{-t'H}(1+\theta'\A)=e^{-(t+t'+\theta\theta')H}(1+(\theta+\theta')\A)
$$
which is equivalent to~$\A^2=H$. Therefore~$\A$ completely determines the representation by the formula 
$$
\exp(-t\A^2+\theta\A)\in \Gamma(\Rge\times \SM(\R^{0|1},M),\Hom(s^*V,t^*V )).
$$
By the previous lemma and the fact that $e^{-tH}=e^{-t\A^2}$ is a section of the endomorphism bundle, we have
$$
(1+\theta \A)  e^{-t\A^2}(fv)=e^{-t\A^2+\theta \A}(fv)=(f+\theta df )e^{-t\A^2+\theta \A}v=f(1+\theta \A)e^{-t\A^2}v+\theta df e^{-t\A^2}v. 
$$
where the middle equality is required by the definition of a representation, and the outer equalities use $e^{-t\A^2+\theta \A}=(1+\theta \A)e^{-t\A^2}$ and $\theta^2=0$. Setting $t=0$ and reading off the equation associated with the component of $\theta$, we find that a representation requires~$\A$ satisfy a graded Leibiniz rule. Hence a representation both determines and is determined by a super connection~$\A$. 

Finally, we observe a super vector bundle isomorphism~$\varphi\in \Omega^\bullet(M;\Hom(V,W))^\times$ is compatible with the representations $\rho$ and $\rho'$ associated with super connections $\A$ and $\A'$ if and only if $\varphi(\A)=\A'$. 
\ep

\begin{proof}[Proof of Theorem~\ref{thm:11unitaryreps}]
It remains to identify unitary representations with unitary super connections. So promote the super vector bundle $V\to M$ of the previous lemmas to a super hermitian vector bundle; we will identify this with its associated $\Z/2$-graded vector bundle with (ordinary) hermitian form; see~\S\ref{sec:superadjoint} for this translation. 

From Definition~\ref{defn:11unitary}, a representation $\rho$ is unitary if the composition $\overline{\rho}^\op\circ \sigma$ is equal to the adjoint representation~$(-)^*\circ \rho$. Let $i^{\deg}$ denote the operator on differential forms $f\mapsto i^{\deg}(f)=i^kf$ for $f\in \Omega^k(M)\subset C^\infty(\SM(\R^{0|1},M))$. Using the characterization of representations afforded by Lemma~\ref{lem:superconnect}, a unitary representation satisfies 
$$
i^{\deg}(e^{-t\A^2+i\theta \A})=e^{-ti^{\deg}(\A^2)+i\theta i^{\deg}(\A)}=(e^{-t\A^2+\theta \A})^*=e^{t(\A^*)^2+\theta \A^* }
$$
This uses Lemma~\ref{lem:11sigma} and properties of the super adjoint~\eqref{eq:adjointantiho}. This equality holds if and only if the super connection satisfies 
\beq
i (i^{\deg}(\A))=\A^*.\label{eq:11unitary}
\eeq
as an equality of odd linear maps $\A\colon \Omega^\bullet(M;V)\to \Omega^\bullet(M;V)$. We can write a super connection as
\beq
\A=\sum_j \A(j),\qquad \A(j)\colon \Omega^\bullet(M;V)\to \Omega^{\bullet+k}(M;V)\label{eq:superconnZ}
\eeq
where the terms in the sum have differential form degree~$j$. Then we calculate
\beq
i(i^{\deg}(\A))&=&i(\A(0)+i\A(1)+i^2 \A(2)+i^3\A(3)+i^4 \A(4)+\dots)\nonumber \\
&=&i\A(0)- \A(1)-i\A(2)+ \A(3)+i\A(4)+\dots \nonumber
\eeq
By the characterization of the super adjoint~\eqref{eq:superadjointnormaladjoint} in terms of ordinary adjoints, we have
$$
\A^*=i\A(0)^\dagger+ \A(1)^\dagger+ i\A(2)^\dagger+ \A(3)^\dagger+ i\A(4)^\dagger+\dots
$$
Then for~\eqref{eq:11unitary} to hold, we require 
\beq
\A(j)^\dagger=\A(j)\ \  j=0,3\mod 4\quad\quad  \A(j)^\dagger=-\A(j), \ \ j=1,2\mod 4.\label{eq:11unitaryconcl}
\eeq
But~\eqref{eq:11unitaryconcl} is the definition of a unitary super connection~\cite{BGV}. 
\end{proof}

\begin{prop}\label{prop:Cliffordlineargeo} A Clifford linear unitary representation of $\Path(M)$ is a bundle of $\Cl_n$-modules over $M$ together with a $\Cl_n$-linear unitary super connection. 
\end{prop}

\bp
A Clifford-linear representation determines an ordinary representation on $V$ by forgetting the $\Cl_n$-action, so by Theorem~\ref{thm:11unitaryreps}, this representation can be described as a super vector bundle with unitary super connection. Then being $\Cl_n$-linear requires that $V$ carry a fiberwise Clifford action and
$$
e^{-t\A^2+\theta \A}{\rm cl}_v={\rm cl}_v e^{-t\A^2+\theta \A}
$$
for ${\rm cl}_v$ the action by a generator $v\in \Cl_n$ of the Clifford algebra. But this is equivalent to demanding that $\A$ be Clifford linear. 
\ep

\subsection{Grothendieck groups of representations}

The following is an immediate corollary to Theorem~\ref{thm:11unitaryreps}; it implies that concordance of representations is an equivalence relation. 

\begin{cor} The prestack $M\mapsto \Rep^n(\Path(M))$ is a stack. 
\end{cor}

Now we compute Grothendieck groups of representations of super paths.

\begin{proof}[Proof of Corollary~\ref{cor:Kthy}]
By Theorem~\ref{thm:11unitaryreps}, a concordance class of a unitary representation of $\Path(M)$ is the concordance class of a unitary super connection on a super vector bundle. Since the space of unitary super connections is affine, the set of concordance classes is the same as isomorphism classes of super vector bundles. The quotient of the free abelian group on super vector bundles by the subgroup generated by~\eqref{eq:Grothequiv} is exactly the K-theory of~$M$. 
\ep

\begin{prop} There is a natural map from $\Rep^n(\Path(M))$ to $\K^n(M)$ that is surjective when~$n$ is even or $M=\pt$. \end{prop}
\bp
The degree zero part of a Clifford linear super connection is an easy example of a Clifford linear Fredholm operator, giving the claimed map. In even degrees the map is surjective, e.g., using the Morita equivalence $\Cl(2n)\simeq \C$ and Bott periodicity. However, in odd degrees this map is typically not surjective (e.g., it fails to see the generator of $\widetilde{\K}^1(S^1)\cong \Z$). When $M=\pt$, the claimed surjection is the Atiyah--Bott--Shapiro map. 
\ep

\section{Rescaled partition functions and differential K-theory}\label{sec:diffKmain}

In this section we study the character theory of $\Rep(\Path(M))$. A priori, this takes values in functions on the inertia groupoid of $\Path(M)$, which consists of constant super paths in~$M$ that start and end at the same super point. The algebra of functions on this stack is the wrong one from the perspective of differential K-theory, e.g., it includes all differential forms, not just closed ones. This requires we refine the character theory. We construct a rescaled partition function, 
\beq
Z\colon \Rep^n(\Path(M))\to \Gamma(\mathcal{L}^{1|1}_0(M);\omega^{\otimes n/2})\cong \left\{\begin{array}{ll} \Omega^{\ev}_\cl(M) & n \ {\rm even} \\ \Omega^{\odd}_\cl(M)& n \ {\rm odd}\end{array}\right. \label{eq:chartheoryK}
\eeq
that takes values is sections of a line bundle over a constant super loop stack, $\mathcal{L}^{1|1}_0(M)$. This constant super loop stack has objects $\R^{1|1}/r\Z\to \R^{0|1}\to M$ and morphisms come from precomposition with super rotations and dilations of super loops. This stack was studied in~\cite{DBE_WG} where sections were computed as~\eqref{eq:chartheoryK}. Applying the machinery of differential Grothendieck groups to~\eqref{eq:chartheoryK}, we obtain the following. 

\begin{thm}\label{prop:diffKthy} The differential Grothendieck group of $\Rep(\Path(M))$ for the character theory~\eqref{eq:chartheoryK} is the differential K-theory of~$M$. 
\end{thm}

The rescaled partition function makes use of a Bismut--Quillen rescaling of super connections. Geometrically, this rescaling comes from dilating the super length of a super path. Physically, this is the renormalization group (RG). When $M=\pt$, the usual trace is automatically invariant under the RG action by the supersymmetric cancelation argument, e.g., see~\cite{BGV}. In families this need not be the case. However, the \emph{rescaled} partition function is a systematic way of extracting a RG-invariant function from a family of representations. 

We give a partial result for the odd cohomological degree. 

\begin{prop} \label{prop:higherdegK}
The differential cocycles $\widehat{\Rep}{}^n(\Path(M))$ with respect to the character theory~\eqref{eq:chartheoryK} map to $\widehat{\K}^n(M)$, the $n^{\rm th}$ differential K-theory group of~$M$. This map is a surjection when~$n$ is even or when $M=\pt$. \end{prop}

We finish the section by explaining how Freed and Lott's analytic orientation in differential K-theory~\cite[\S3, \S7]{LottFreed} gives a $1|1$-dimensional version of our Theorem~\ref{thm:diffpush}. Their construction can be interpreted as choosing cutoffs for a family of Dirac operators as in~\S\ref{sec:Kthymot}, and then choosing eta forms that mediate between the Chern character of Bismut super connection of this family and the cutoff super connection, as described in the next subsection below. In brief, to a family of spin manifolds $\pi\colon X\to M$, they construct a differential cocycle in $\widehat{\Rep}(\Path(M))$. This can be interpreted as defining a cutoff version of supersymmetric quantum mechanics in families. 

\subsection{Motivation: partition functions in effective field theory}\label{eq:11effpart} 

In this section we study how partition functions of $1|1$-dimensional field theories behave under cutoffs. For this, it is important to consider 1-parameter families of field theories that interpolate between different choices of cutoff. These arise from the action of the renormalization group flow on field theories, which we introduce by way of a basic example. 

Consider a quantum mechanical system given by the spinors $\Gamma(\$^+\oplus\$^-)$ of an even-dimensional Riemannian manifold with time-evolution operator $\exp(-t\slashed{D}^2)$. The \emph{renormalization group (RG) flow} dilates time by $\mu^2\in \R_{>0}$, modifying the time-evolution operator as $e^{-t\slashed{D}{}^2}\mapsto e^{-t\mu^2\slashed{D}{}^2}$, or equivalently, replacing~$\slashed{D}$ by $\mu \slashed{D}$. 

Now consider a pair of cutoff theories (in the sense of~\S\ref{sec:Kthymot}) with state spaces $V_{<\lambda}$ and $V_{<\lambda'}\subset\Gamma(\$^+\oplus\$^-)$ with $V_{<\lambda}\oplus V_{[\lambda,\lambda')}\cong V_{<\lambda'}$. We get time-evolution operators on these finite-dimensional state spaces by restricting $\exp(-t\slashed{D}^2)$. On $V_{[\lambda,\lambda')}$ (where~$\slashed{D}$ is invertible) the time evolution operator approaches the zero operator in the limit $\mu\to\infty$ of the RG-flow. This gives a homotopy (meaning, a 1-parameter family of field theories) interpolating between the cutoff theory with state space $V_{<\lambda}$ and the cutoff theory with state space $V_{<\lambda'}$. 

To relate observables in this pair of cutoff theories, the idea from Wilsonian effective field theory is to study how they behave under the renormalization group flow, then taking~$\mu\to \infty$. This procedure is sometimes called ``integrating out" the higher energy contribution. The observable of interest in our case is the partition function. This is the super trace of the time-evolution operator. The super symmetric cancellation argument shows that this quantity is in fact \emph{invariant} under the renormalization group flow: it is the index of the Dirac operator~$\slashed{D}$. 

However, in families the situation isn't quite as simple. By the work of Fei Han~\cite{Han}, the partition function in the case of finite-dimensional state spaces is the differential form-valued Chern character of a super connection. For vector bundles $V_{<\lambda}$ and $V_{<\lambda'}$ with $V_{<\lambda}\oplus V_{[\lambda',\lambda)}\cong V_{<\lambda'}$ and an invertible super connection on $V_{[\lambda,\lambda')}$, the difference in the Chern character is measured by a Chern--Simons form
\beq
\eta:=\int_0^\infty {\rm Tr}\left(\frac{d\A^\mu}{d\mu}e^{-(\A^\mu)^2}\right)ds,\label{eq:BCetaform}
\eeq
where $\A^\mu=\A_{<\lambda} \oplus \A^\mu_{[\lambda',\lambda)}$ is a 1-parameter family of super connections with $\A_{<\lambda}$ the given super connection on $V_{<\lambda}$ and on $V_{[\lambda',\lambda)}$ we take
\beq
&&\A^\mu_{[\lambda',\lambda)}= \mu \A_{[\lambda',\lambda)}(0)+\A_{[\lambda',\lambda)}(1)+\mu^{-1}\A_{[\lambda',\lambda)}(2)+\dots \mu^{-j+1}\A_{[\lambda',\lambda)}(j).\label{eq:Bismutrescale}
\eeq
for $\A_{[\lambda',\lambda)}(i)\colon \Omega^\bullet(M;V)\to \Omega^{\bullet+i}(M;V)$ the degree $i$ piece of the given super connection. Under the equivalence afforded by Theorem~\ref{thm:11unitaryreps}, the renormalization group action on representations of super paths coincides exactly with~\eqref{eq:Bismutrescale}. Hence, a Chern--Simons form measures the difference in families of partition functions gotten from different choices of cutoff, and the Chern---Simons form itself is constructed using the renormalization group flow. Conversely, given a bundle $W\oplus \Pi W$ with ordinary connection $\nabla$, we can form a super connection $\A_0+\nabla$ where $\A_0$ is the identity map $W\to \Pi W$, viewed as an odd linear map. The form~\eqref{eq:BCetaform} is identically zero for this, meaning that such families of finite-dimensional state spaces can be removed without any affect on the partition function. 

There is an infinite-dimensional version of this as well, relating the Chern character of a family of Dirac operators to the Chern character of a cutoff. The easiest version takes as input a family of Dirac operators over~$M$ whose fiberwise kernel is a vector bundle on~$M$ (though this assumption can be dropped; see Lemma~\ref{lem:FreedLott}). Then there is a Bismut--Cheeger eta form~\cite{BismutCheeger} that measures the difference between the Chern character of the family of Dirac operators (as defined by Bismut~\cite{Bismutindex}, see also~\cite[Chapter 9]{BGV}) and the Chern character of the index bundle. The formula for this eta form is essentially the same as~\ref{eq:BCetaform}. 

With the above ideas in mind, if we want our cutoff theory to remember the true value of the partition function (before cutting off) we need the extra data of an eta form~\eqref{eq:etaform}. As we'll see below, the data of a representation of $\Path(M)$ and such an eta form is exactly a differential K-theory cocycle.

\subsection{Dilation of super paths and Bismut--Quillen rescaling of super connections}\label{sec:11Bismut--Quillen}

There is a dilation action on $\R^{1|1}$, 
$$
(t,\theta)\mapsto (\mu^2 t,\mu\theta),\qquad (t,\theta)\in \R^{1|1}(S), \ \mu\in \R_{>0}(S)
$$
that descends to an action on $\E^{1|1}$ through group homomorphisms. We promote this to a functor on constant super paths, which is the \emph{renormalization group} (RG) action. Below, $\R_{>0}$ is the discrete super Lie category associated with the manifold $\R_{>0}$. 

\begin{defn} \label{defn:11RG}
Define a functor $\RG\colon \R_{>0}\times \Path(M)\to \Path(M)$ whose value on $S$-points of objects and morphisms is
$$
(\mu,x,\psi)\mapsto (x,\mu^{-1}\psi),\qquad (\mu,t,\theta,x,\psi)\mapsto (\mu^2 t,\mu\theta,x,\mu^{-1}\psi)
$$
where $\mu\in \R_{>0}(S)$, $(t,\theta)\in \R^{1|1}_{\ge 0}(S)$, and $(x,\psi)\in \SM(\R^{0|1},M)(S).$
\end{defn}

We observe that the diagram commutes,
\beq
\begin{array}{c}
\begin{tikzpicture}
  \node (A) {$\R_{>0}\times \R_{>0}\times \Path(M)$};
  \node (B) [node distance= 6cm, right of=A] {$\R_{>0}\times \Path(M)$};
  \node (C) [node distance = 1.5cm, below of=A] {$\R_{>0}\times \Path(M)$};
\node (D) [node distance = 1.5cm, below of=B] {$\Path(M),$};
  \draw[->] (A) to node [above] {$\id_{\R_{>0}}\times \RG$} (B);
  \draw[->] (A) to node [left] {$m\times \id_{\Path(M)}$} (C);
  \draw[->] (B) to node [right] {$\RG$} (D);
  \draw[->] (C) to node [below] {$\RG$} (D);
\end{tikzpicture}\end{array}\nonumber
\eeq
where $m\colon \R_{>0}\times\R_{>0}\to \R_{>0}$ is multiplication. Hence, $\RG$ defines a strict $\R_{>0}$-action on $\Path(M)$. Let $\RG_\mu$ be the restriction of $\RG$ to the subcategory $\{\mu\}\times \Path(M)$, so $\RG_\mu\colon \Path(M)\to \Path(M)$ and $\RG_\mu\circ \RG_\lambda=\RG_{\mu\lambda}$.

Precomposing a representation with~$\RG_\mu$ leads to an $\R_{>0}$-action on representations. We characterize it in terms of an action on super connections, using Theorem~\ref{thm:11unitaryreps}. 

\begin{lem} \label{lem:11Bismut--Quillen}
The action of $\RG_\mu$ on a $\Cl_n$-linear representation of $\Path(M)$ associated with a super connection $\A$ is precisely the Bismut--Quillen rescaling~\cite[Chapter 9]{BGV} by~$\mu$, 
$$
\A\stackrel{\RG_\mu}{\mapsto} \mu \A(0)+\A(1)+\mu^{-1}\A(2)+\dots \mu^{-j+1}\A(j),
$$
for $\A(i)\colon \Omega^\bullet(M;V)\to \Omega^{\bullet+i}(M;V)$ the degree $i$ piece of the super connection. 
\end{lem}
\bp 
The action on the semigroup representation is
\beq
\exp(-t\A^2+\theta\A)
&\stackrel{\RG_\mu}{\mapsto}& \exp\Big(-\mu^2 t(\A(0)+\mu^{-1}\A(1)+\dots +\mu^{-j}\A(j))^2\nonumber\\&&\phantom{\exp(}+\mu \theta(\A(0)+\mu^{-1}\A(1)+\dots +\mu^{-j}\A(j))\Big).\nonumber
\eeq
In terms of the super connection this $\R_{>0}$-action is the claimed Bismut--Quillen rescaling. 
\ep

\subsection{The inertia groupoid and the constant super loop stack}

Following Defintion~\ref{defn:Liechar}, characters of representations of $\Path(M)$ are functions on the inertia groupoid,~$\Lambda(\Path(M))$. 

\begin{lem} The inerita groupoid $\Lambda(\Path(M))$ is the discrete groupoid, 
$$
\{\R_{\ge 0}\times \SM(\R^{0|1},M)\toto \R_{\ge 0}\times \SM(\R^{0|1},M)\},
$$ 
and the nondegenerate inertia groupoid is the subgroupoid with objects $\R_{> 0}\times \SM(\R^{0|1},M)\}\subset \R_{\ge 0}\times \SM(\R^{0|1},M)\}$. 

\end{lem}

\bp The inertia groupoid of $\Path(M)$ has as objects those super paths with the same start and endpoint, so by the description in Lemma~\ref{lem:subgrpd1},
$$
\Lambda(\Path(M))_0=\R_{\ge 0}\times \SM(\R^{0|1},M).
$$
Morphisms of $\Lambda(\Path(M))$ are \emph{invertible} super paths with the same start and endpoint, but the only such path is the identity path. Hence, $\Lambda(\Path(M))$ is a discrete groupoid with objects as above. The nondegenerate inertia groupoid is the subgroupoid corresponding to noninvertible endomorphisms, and these are precisely the super paths with strictly positive super length~$t>0$, giving the claimed description. 
\ep

We take a moment to spell out the inertia groupoid explicitly in terms of the geometry of super paths in~$M$. An $S$-point of the objects is a map
$$
\phi\colon S\times \R^{1|1}\stackrel{{\rm pr}}{\to} S\times \R^{0|1}\to M
$$
with the same source and target super point, which is equivalent to invariance of the map~$\phi$ under the $\Z$-action generated by the family of translations $t\in \R_{\ge 0}(S)$ that act on $S\times \R^{1|1}$. This means that the map above descends to the quotient,
\beq
(S\times \R^{1|1})/t\Z \to S\times \R^{0|1}\to M.\label{eq:markedloop}
\eeq
In the inertia groupoid, there are no non-identity isomorphisms between these super circles. 

However, there are interesting super Euclidean isometries (see~\S\ref{sec:supertrans}) between super circles coming from super rotations and a $\Z/2$-action on the odd line bundle. These symmetries are determined by $S$-points of $\E^{1|1}\rtimes \Z/2$. The renormalization group also acts by dilation, which combines with the super Euclidean group to give maps between super circles for $S$-points of $\E^{1|1}\rtimes \R^\times$. We will ask that our (rescaled) partition functions are invariant under these additional symmetries, leading to the following definition. 

\begin{defn} 
Define the \emph{stack of constant super loops in $M$,} denoted $\mathcal{L}^{1|1}_0(M)$, as the super Lie groupoid,
$$
\mathcal{L}^{1|1}_0(M):=\left(\begin{array}{c} (\E^{1|1}\rtimes \R^\times \times \R_{>0})/\Z\times \SM(\R^{0|1},M))\\ \downarrow\downarrow \\
\R_{>0} \times \SM(\R^{0|1},M))\end{array}\right),
$$
where the quotient $(\E^{1|1}\rtimes \R^\times \times \R_{>0})/\Z$ comes from the $\Z$-action generated by
$$
(s,\eta,\mu,t)\mapsto (s+t,\eta,\mu,t)\qquad (s,\eta)\in \E^{1|1}(S), \ \mu\in \R^\times(S), \ t\in \R_{>0}(S). 
$$
The source map for the groupoid is the projection, and the target map comes from an~$\E^{1|1}\rtimes \R^\times$-action. On $\SM(\R^{0|1},M)$ this action it is through the homomorphism $\E^{1|1}\rtimes \R^\times\to \E^{0|1}\rtimes \R^\times$ and then the precomposition action on $\SM(\R^{0|1},M)$. The $\E^{1|1}\rtimes \R^\times$-action on $\R_{>0}$ is through the homomorphism $\E^{1|1}\rtimes \R^\times\to \R^\times$ followed by the dilation action,
$$
\R^\times\times \R_{>0}\to \R_{>0}, \quad (\mu,s)\mapsto \mu^2s. 
$$
\end{defn}

We observe that an $S$-point of objects of $\mathcal{L}^{1|1}_0(M)$ gives a family of super circles with a map to $M$ as in~\eqref{eq:markedloop}. An $S$-point of morphisms gives a commuting triangle
\beq
\begin{tikzpicture}[baseline=(basepoint)];
\node (A) at (0,0) {$(S\times \R^{1|1})/t\Z$};
\node (B) at (6,0) {$(S'\times \R^{1|1})/t'\Z$};
\node (C) at (3,-1.5) {$M$};
\draw[->] (A) to node [above=1pt] {$\cong$} (B);
\draw[->] (A) to node [left=1pt]{$\phi$} (C);
\draw[->] (B) to node [right=1pt]{$\phi'$} (C);
\path (0,-.75) coordinate (basepoint);
\end{tikzpicture}\label{11anntriangle2}
\eeq
where the horizontal arrow is determined by an $S$-point of $\E^{1|1}\rtimes \R^\times$ which acts on the family of super tori by super translations and global dilations.

There is an odd line bundle $\omega^{1/2}\to \mathcal{L}^{1|1}_0(M)$ coming from the functor $\mathcal{L}^{1|1}_0(M)\to \pt\sq \Z/2$ that sends all objects to $\pt$, and to a morphism assigns~$\{\pm 1\}\cong \Z/2$ according to whether the morphism preserves or reverses the orientation of the odd line bundle over the family of super circles. A bit more explicitly, this functor is induced by the homomorphism
$$
\E^{1|1}\rtimes \R^\times\to \R^\times \to \{\pm 1\}.
$$
and then $\omega^{1/2}$ is the pullback of the odd line bundle over $\pt\sq \Z/2$. 
Let $\omega^{\otimes n/2}=(\omega^{1/2})^{\otimes n}$. 

\begin{lem}\label{lem:superloopfuns} There are natural isomorphisms of vector spaces 
$$
\Omega^\ev_{\cl}(M)\stackrel{\sim}{\to} \Gamma(\mathcal{L}^{1|1}_0(M);\omega^{{\otimes 2n}/2}),\qquad \Omega^\odd_\cl(M)\stackrel{\sim}{\to} \Gamma(\mathcal{L}^{1|1}_0(M);\omega^{\otimes(2n+1)/2})
$$
given by $f\mapsto t^{j/2}\otimes f\in C^\infty(\R_{>0}\times \SM(\R^{0|1},M))$ for $f\in \Omega^{j}_{\cl}(M)$. The graded multiplication of sections agrees with the graded multiplication on differential forms. 
\end{lem}

\bp
We identify functions on $\mathcal{L}^{1|1}_0(M)$ with elements of $C^\infty(\R_{>0}\times \SM(\R^{0|1},M))\cong C^\infty(\R_{>0})\otimes \Omega^\bullet(M)$ invariant under the $\E^{1|1}\rtimes \R^\times$-action. The $\E^{1|1}$-action factors through the standard $\E^{0|1}$-action generated by the de~Rham differential on $\Omega^\bullet(M)\cong C^\infty(\SM(\R^{0|1},M))$ so the differential form components of invariant function must be closed. 

The $\R^\times$-action on $f\in \Omega^k(M)$ is by $f\mapsto \mu^{-n}f$, and on $\R_{>0}$ is by $t\mapsto \mu^2t$ where $t$ is the standard coordinate. The action by $-1\in \R^\times$ therefore demands that the differential form component be even or odd, depending on the parity of~$n$. For a function $F$ to define a section, we further require that $F$ be invariant under the action of $\mu\in \R_{>0}(S)\subset \R^\times(S)$. These invariant functions are generated by $t^{j/2}\otimes f$ with $f\in \Omega^{j}_\cl(M)$, as claimed. 
\ep
\subsection{The rescaled partition function of a $\Cl_n$-linear representation}

The character of a Clifford linear representation of $\Path(M)$ is the Clifford super trace (see~\S\ref{sec:ClsTr}) applied to the endomorphism of Clifford modules gotten by restriction of the representation to the objects of the non-degenerate inertia groupoid, ${\rm Ob}(\Lambda^{\rm nd}(\Path(M)))$. 

\begin{lem}\label{lem:11boringchar}
The character of a Clifford linear representation of $\Path(M)$ associated with a super connection~$\A$ is
$$
{\rm sTr}_{\Cl_n}(e^{-t\A^2})\in C^\infty(\Lambda^\nd(\Path(M))\cong C^\infty(\R_{>0}\times \SM(\R^{0|1},M)),
$$
as a function on the nondegenerate inertia groupoid.
\end{lem}
\bp This follows directly from Lemma~\ref{lem:superconnect}, Proposition~\ref{prop:Cliffordlineargeo}, and the definition of the Clifford super trace. \ep

The character ${\rm sTr}_{\Cl_n}(e^{-t\A^2})$ is automatically invariant under super loop rotation; this corresponds to the Chern character being a closed form. However, it is typically not automatically invariant under super loop dilation, so does not descend to a function on the stack $\mathcal{L}^{1|1}_0(M)$. This can be repaired by applying a rescaling of super loops from~\S\ref{sec:11Bismut--Quillen}, which translates into a Bismut--Quillen rescaling of the super connection. 

Consider the composition 
\beq
&&\R_{>0}\times \SM(\R^{0|1},M)\to \R_{>0}\times \R_{>0}\times \SM(\R^{0|1},M)\stackrel{\RG}{\to} \R_{>0}\times \SM(\R^{0|1},M)\label{eq:11rescale}
\eeq
where the first arrow is determined by $t\mapsto (1/t,t)$ for $t$ a coordinate on $\R_{>0}$, and $\RG$ denotes the restriction of the functor $\RG$ to the subset 
$$
\R_{>0}\times \SM(\R^{0|1},M)={\rm Ob}(\Lambda(\Path(M))) \subset {\rm Mor}(\Path(M))=\R^{1|1}_{\ge 0}\times \SM(\R^{0|1},M).
$$

\begin{defn}
For a representation $\rho$, consider the pullback of the section of the endomorphism bundle determined by $\rho$ along the composition~\eqref{eq:11rescale}. Define the \emph{rescaled partition function} $Z(\rho)$ as the Clifford super trace of this pullback,
$$
Z(\rho)\in C^\infty(\R_{>0}\times \SM(\R^{0|1},M)).
$$
At a fixed $t\in \R_{>0}$ we observe that the value of $Z(\rho)$ is ${\rm sTr}_{\Cl_n}((\rho\circ \RG_{1/t})(t)),$ i.e., it is the super trace on super paths of length~$t$ of the image of $\rho$ under the renormalization group flow by $1/t$.
\end{defn}

\begin{rmk} A rescaling of the Chern character is also built into Bismut's definition for a family of Dirac operators~\cite{Bismutindex}. See~\cite[Ch. 9]{BGV}, especially \S9.1 which treats the finite-dimensional case.
\end{rmk}

\begin{lem} \label{prop:11charthy} 
For a $\Cl_n$-linear representation~$\rho$, the function $Z(\rho)$ descends to a section of $\omega^{\otimes n/2}$ over the stack $\mathcal{L}^{1|1}_0(M)$ and so defines a map $Z\colon \Rep^n(\Path(M))\to \Gamma(\mathcal{L}^{1|1}_0(M);\omega^{\otimes n/2})$. \end{lem}
\bp
By Lemma~\ref{lem:11Bismut--Quillen} and the definition of the Clifford super trace (see~\S\ref{sec:ClsTr}), we have
$$
{\rm sTr}_{\Cl_n}\left(\exp\left(-t\left(\sum_j (1/t)^{(-j+1)/2}\A_k\right)^2\right)\right)={\rm sTr}\left(\Gamma\exp\left(-\left(\sum_j t^{j/2}\A_k)\right)^2\right)\right)
$$
as functions on $C^\infty(\R_{>0}\times \SM(\R^{0|1},M))$, where $\Gamma$ is the chirality operator on the Clifford algebra. When $n$ is even (respectively odd), $\Gamma$ is even (respectively odd). The super trace of an odd endomorphism is zero, and so the super trace above takes values in even or odd forms according to the parity of~$n$. 

By Lemma~\ref{lem:superloopfuns}, the factors of $t$ in the expression for the rescaled partition function are precisely the ones required so that $Z(\rho)$ descends to a section of $\omega^{\otimes n/2}$ on the constant super loop stack $\mathcal{L}^{1|1}_0(M)$. 
\ep

For the sake of being explicit, when $n=0$ we get 
$$
Z(\rho)={\rm sTr}\left(\exp\left(-\left(\sum_k t^{n/2}\A_k\right)^2\right)\right)\mapsto {\rm sTr}(\exp(-\A^2))\in \Omega^\ev_{\cl}(M). 
$$
where the map applies the isomorphism in Lemma~\ref{lem:superloopfuns} on functions. This is the usual Chern character of the super connection~$\A$. When $n=1$, we get 
$$
Z(\rho)={\rm sTr}\left(\Gamma\exp\left(-\left(\sum_k t^{n/2}\A_k\right)^2\right)\right)\mapsto {\rm sTr}(\Gamma\exp(-\A^2))\in \Omega^\odd_\cl(M)
$$
where the map applies the isomorphism in Lemma~\ref{lem:superloopfuns} for sections of $\omega^{1/2}$. This agrees with the differential form representative of the odd Chern character of a $\Cl_1$-linear super connection as constructed by Quillen~\cite[\S5]{Quillensuper}. Computations in the other cases can be reduced to the ones above. 

To promote the map $Z$ to a refinement of the character theory of $\Path(M)$, we make the following definition.

\begin{defn}
Define the injective algebra map $i\colon \Gamma(\mathcal{L}^{1|1}_0(M);\omega^{\otimes n/2}) \to C^\infty(\Lambda(\Path(M)))$ by
$$
t\otimes 1\mapsto t\otimes 1, \  \ \  1\otimes f\mapsto t^{(-k+1)/2}\otimes f,\quad t\in C^\infty(\R_{>0}), \ f\in \Omega^{k}_\cl(M).
$$
\end{defn}

By inspection, this makes the diagram commute,
\beq
\begin{array}{c}
\begin{tikzpicture}
  \node (A) {$\Rep^n(\Path(M))$};
  \node (B) [node distance= 4cm, right of=A] {$C^\infty(\Lambda(\Path(M)))$};
  \node (C) [node distance = 1.5cm, above of=B] {$\Gamma(\mathcal{L}^{1|1}_0(M);\omega^{\otimes n/2})$};
  \draw[->] (A) to  (B);
  \draw[->,dashed] (A) to node [above] {$Z$} (C);
  \draw[->,left hook-latex] (C) to node [right] {$i$} (B);
\end{tikzpicture}\end{array}\nonumber
\eeq
and hence representations of $\Path(M)$ have a $C^\infty(\mathcal{L}^{1|1}_0(M))$-valued character theory. 

\subsection{Differential grothendieck groups of representations}

\begin{proof}[Proof of Theorem~\ref{prop:diffKthy}]
We spell out differential concordance classes with respect to the $C^\infty(\mathcal{L}^{1|1}_0(M))$-characters in the sense of Definition~\ref{defn:diffconc}. We start with the collection of pairs
$$
{\widehat\Rep}{}(\Path(M))=\{\rho\in \Rep(\Path(M)),\ \alpha\in C^\infty(\mathcal{L}^{1|1}_0(M\times \R)) \mid i_0^*\alpha=Z(\rho)\}. 
$$
By Theorem~\ref{thm:11unitaryreps}, $\rho$ is equivalent to a unitary super connection $\A$, and by Lemma~\ref{lem:superloopfuns}, we identify $\alpha$ with a closed, even differential form $\alpha\in \Omega^\ev_{\cl}(M\times \R)$. By Proposition~\ref{prop:11charthy}, the restriction of $\alpha$ to $M\times \{0\}$ is the Chern form of $\A$, 
$$
i^*_0\alpha={\rm sTr}(e^{-\A^2})\in \Omega^\ev_{\cl}(M).
$$
So hereafter we make the identification
$$
{\widehat\Rep}{}(\Path(M))=\{\A {\rm \ unitary\ super\ connection,\ }  \alpha\in \Omega^\ev_\cl(M\times \R)\mid {\rm sTr}(e^{-\A^2})=i_0^*\alpha\}
$$

We similarly identify the data of a differential concordance $(\widetilde\rho,\widetilde\alpha)$ (see Definition~\ref{defn:diffconc}) with $(\widetilde\A,\widetilde\alpha)$ for a unitary super connection $\widetilde\A$ on a bundle over $M\times \R$, and $\widetilde\alpha\in \Omega^\bullet_\cl(M\times \R^2)$.
First we consider the quotient by differential concordances where $\widetilde\A$ is the constant concordance, meaning the bundle and super connection on $M\times \R$ pullback from~$M$. This restricted equivalence relation is precisely the one from Definition~\ref{defn:relconc}. We computed equivalence classes in Lemma~\ref{lem:relconc}, finding them to be determined by the integral
$$
\beta=\int_{M\times I/M} \alpha,\qquad [\beta]\in \Omega^\odd(M)/d\Omega^\ev(M). 
$$
As such, for this part of the equivalence relation we have~$(\A,\alpha)\sim (\A,\alpha')$ if $\alpha$ and $\alpha'$ define the same equivalence class $\beta\in \Omega^\odd(M)/d\Omega^\ev(M)$ using the formula above. This allows us to work with pairs~$(\A,\beta)$ for the remainder of the proof. 

For an arbitrary differential concordance (without a restriction on $\widetilde{\A}$), by Stokes theorem the fiberwise integral of $\widetilde{\alpha}\in\Omega^\ev_\cl(M\times \R^2)\cong C^\infty(\mathcal{L}^{1|1}_0(M\times\R^2))$ along $M\times I^2\to M$ satisfies
\beq
d\int_{M\times I^2/M} \widetilde\alpha=\int_{M\times I/M} \alpha'-\int_{M\times I/M}\alpha+\int_{M\times I/M}{\rm sTr}(e^{-\widetilde\A^2}). \label{Kdiffconc}
\eeq
The third term is exactly the Chern--Simons form for the super connection $\widetilde{\A}$, 
$$
d\int_{M\times I/M} {\rm sTr}(e^{-\widetilde\A^2})={\rm sTr}(e^{-\A^2_1})-{\rm sTr}(e^{-\A^2_0})=d\CS(\A_1,\A_0) 
$$
Using that the left hand side is exact and any pair of super connections are concordant, from Equation~\ref{Kdiffconc} we find that the relation of differential concordance is exactly 
$$
(\A,\beta)\sim (\A',\beta')\iff \CS(\A,\A')+\beta'=\beta\in \Omega^\odd(M)/d\Omega^\ev(M).
$$
Then the differential Grothendieck group (Definition~\ref{defn:diffGG}) is the free abelian group on these equivalence classes modulo subgroup generated by
$$
(\A,\beta)+(\A',\beta')-(\A\oplus \A',\beta+\beta'),\qquad (\A\oplus\Pi \A,0),
$$
where $\Pi \A$ is the super connection $\A$ on the parity reversed super vector bundle.
But this is precisely the presentation of differential K-theory given by Klonoff~\cite[Proposition~4.64]{Klonoff} (see~\S\ref{sec:backKthy}).

The Chern character is multiplicative, so we get a ring structure on the differential Grothendieck group from Definition~\ref{defn:Diffgroth}. By inspection, this agrees with the ring structure defined in~\cite[pg~49]{Klonoff}.
\ep

\begin{proof}[Proof of Proposition~\ref{prop:higherdegK}]

This is the differential version of Proposition~\ref{prop:Cliffordlineargeo}. Objects in $\widehat{\Rep}{}^n(\Path(M))$ are finite-dimensional Clifford module bundles over~$M$ with a Clifford linear super connection, together with a concordance of closed forms whose source is the Chern character of this Clifford-linear super connection. But these give maps to~$\K^n(M)$ and to $\Omega^{\ev/\odd}_\cl(M)$ with a compatible homotopy, and so define classes in $\widehat{\K}^n(M)$ (in the sense of Hopkins--Singer differential K-theory) by the universal property. That this map is a surjection in the even degrees follows from Theorem~\ref{prop:diffKthy} and the fact that $\Cl(2n)$ is Morita equivalent to~$\C$. 
\ep

\subsection{The Freed--Lott analytic orientation and a cutoff version of supersymmetric quantum mechanics}\label{sec:FreedLott}

We now explain how Freed and Lott's differential analytic pushforward~\cite{LottFreed} gives the K-theory variant of Theorem~\ref{thm:diffpush}, namely a differential cocycle in $\widehat{\Rep}(\Path(M))$ for a family $\pi\colon X\to M$ of spin manifolds. This requires a few background results in index theory, which we review first. The original reference is~\cite{Bismutindex}; see also~\cite[Chapter 9]{BGV} for a more expansive discussion or~\cite[\S3]{LottFreed} for a condensed one.

Let $\pi\colon X\to M$ be a proper submersion of relative dimension~$d$, or equivalently, $\pi \colon X\to M$ is a smooth fiber bundle with compact fibers of dimension~$d$. Let $T(X/M)={\rm ker}(d\pi)\subset TX$ denote the vertical tangent bundle. A \emph{spin structure} on $\pi$ is a spin structure on $T(X/M)$. A \emph{Riemannian structure} on~$\pi$ is a metric on $T(X/M)$ and a horizontal distribution~$H(X/M)$ on~$X$. This permits the construction of a Levi-Civita connection on the fibers of~$\pi$~\cite[Definition~1.6]{Bismutindex}. When $\pi$ has a spin and Riemannian structure, we call $\pi\colon X\to M$ a \emph{geometric family of spin manifolds}. 

Let $E\to X$ be a real vector bundle with metric and compatible connection $\nabla^E$ on a geometric family of spin manifolds. We can form the fiberwise Dirac operator on~$\pi$ twisted by~$E$, denoted~$\slashed{D}$. This is the degree zero part of the Bismut super connection,
$$
A_\mu:=\mu\slashed{D}+\nabla^H+\frac{1}{4\mu}c(T)
$$
where $\nabla^H$ is the unitary connection on the fiberwise spinors coming from the horizontal distribution $H(X/M)\subset TX$, and $c(T)$ is Clifford multiplication by the curvature 2-form of the horizontal distribution. Then Bismut showed
\beq
{\rm Ch}(\slashed{D}):=\lim_{\mu\to 0}  {\rm sTr}(e^{-A_\mu^2})=(2\pi i)^{d/2}\int_{X/M} \hat{A}(X/M)\wedge {\rm Ch}(\nabla^E)\in \Omega^\bullet_\cl(M)\label{eq:BismutChernchar}
\eeq
where $\hat{A}(X/M)$ is the $\hat{A}$-form of the vertical bundle. The above is a lift of the Chern character of $\pi_![E]\in \K^{-d}(X)$ to a differential form, where $\pi_!$ is the K-theory pushforward.\footnote{This is the complexification of the KO-pushforward.} 

To extract a finite-dimensional (or cutoff) version of the family of Dirac operators, naively one could take the kernel of $\slashed{D}$ itself. However, this kernel might not be a vector bundle on $M$: the dimension can jump. Freed and Lott use a lemma of Mischenko--Fomenko~\cite{MischenkoFomenko} to get around this problem; we summarize their construction as follows, combining their Lemma~7.11 with their Equations~7.23 and~7.24. 

\begin{lem}[Mischenko--Fomenko, Freed--Lott]\label{lem:FreedLott}
Given a proper family of geometric spin manifolds $X\to M$ over a compact base and a vector bundle $E\to X$ with connection, there is a finite-dimensional smooth subbundle~$V$ of the $E$-twisted fiberwise spinor bundle that contains the fiberwise kernel of the twisted Dirac operator. Moreoever,~$V$ has a super connection~$\A$ whose degree zero piece is the restriction of the fiberwise Dirac operator to~$V$ and there is a class $\eta\in \Omega^{\odd}(X)/d\Omega^{\ev}(X)$ with 
\beq
d\eta={\rm Ch}(\slashed{D})-{\rm Ch}(\A)\label{eq:etaform}
\eeq
where ${\rm Ch}(\slashed{D})$ is the Chern character of Bismut's super connection~\eqref{eq:BismutChernchar}.
\end{lem}

Now, for a geometric family of spin manifolds $\pi\colon X\to M$ with even fiber dimension $2d$ and an associated family of Dirac operators~$\slashed{D}$, we apply Lemma~\ref{lem:FreedLott} to get
\beq
(\slashed{D},\$_{X/M})\rightsquigarrow  (\A,V,\eta).\label{eq:cutoffK}
\eeq
This yields a finite-dimensional vector bundle $V\to M$ with super connection $\A$ and $\eta\in \Omega^{\odd}(X)/d\Omega^{\ev}(X)$ such that
$$
d\eta={\rm Ch}(\slashed{D})-{\rm Ch}(\A).
$$
Identifying $\eta$ with its associated concordance, define
\beq
&&\widehat{\alpha}(X):=\Big(V,\A ,\eta\Big)\in \widehat{\Rep}(\Path(M))\to \widehat{\K}(M).\nonumber
\eeq
\begin{rmk}
We think of $\widehat{\alpha}(X)$ as cutoff version of supersymmetric quantum mechanics in the fibers $\pi\colon X\to M$. Indeed, Lemma~\ref{lem:FreedLott} finds a subbundle containing the energy zero states, so defines a cutoff energy~$\lambda$ that might vary with $M$. Then the form $\eta$ remembers the contribution to the partition function from the higher energy states. 
\end{rmk}

\section{Positive energy representations of super annuli}\label{sec:per}

In this section we introduce the super Lie category of constant super Euclidean annuli in~$M$, denoted~$\Ann(M)$. Orientation reversal of super annuli leads to a definition of unitary representations of this category. There is a subgroupoid $\Rot(M)\subset \Ann(M)$ consisting of ``thin" annuli, which act on super circles by rotation. Decomposing a unitary representation using this circle action leads to a definition of positive energy representations of~$\Ann(M)$. In parallel to the situation in K-theory, the first key computation in analyzing this category of representations is a characterization in terms of more familiar geometric quantities. This is the main result of the section. 

\begin{thm}\label{thm:geocharof21rep}\label{prop:Cliffordlineargeo21} The category of positive energy unitary representations of $\Ann(M)$ is equivalent to the category whose objects are $\Z$-graded (possibly infinite-dimensional) super hermitian vector bundles $V\to M$, and each homogenous piece for the $\Z$-grading is a finite-dimensional super vector bundle $V_k\to M$ with super connection~$\A_k$. Morphisms in the category are isomorphisms $(V,\A)\to (V',\A')$ of super vector bundles compatible with the $\Z$-gradings and super connections. 
\end{thm}

The correspondence is as explicit as in~\S\ref{sec:Quillenconn} for K-theory. Given geometric data as above, we obtain a representation of super annuli on the vector bundle $V:=\bigoplus_k V_k$ is determined by the semigroup representation
$$
\H^{2|1}\times \Omega^\bullet(M;V)\to \Omega^\bullet(M;V) \qquad \rho(\tau,\bar\tau,\theta)=\bigoplus_k e^{2\pi i k \tau}e^{-2\im(\tau)\A_k^2+\theta\A}
$$
where $\H^{2|1}\subset \R^{2|1}$ is the (super) closed upper half plane. As before, unitarity of the representation corresponds to unitarity of the connections. With this result in hand, identifying the Grothendieck group of $\Rep(\Ann(M))$ with $\K_\Tate(M)$ (proving Theorem~\ref{thm:Tateeasy}) follows directly. 

The motivation from field theories runs in parallel to~\S\ref{sec:Quillenconn}, with the new feature that we chose cutoffs for each weight space of the $\Rot(M)$-action. This corresponds to a choice of cutoff for each power in~$q$. As before, any choice of cutoff defines the same underlying class in $\K_\Tate(M)$ simply by applying the argument for K-theory to each power of~$q$. 

\subsection{Super Euclidean annuli in $M$} Define the \emph{$2|1$-dimensional super (Euclidean) translation group}, denoted $\E^{2|1}$, to be the super manifold $\R^{2|1}$ equipped with the multiplication
\beq
&&(z,\bar z,\eta)\cdot (z',\bar z',\eta')=(z+z',\bar z+\bar z'+i\eta\eta',\eta+\eta'), \quad (z,\bar z,\eta),(z',\bar z,\eta')\in \R^{2|1}(S).\label{eq:E21}
\eeq
See~\S\ref{sec:smfld} and~\S\ref{sec:supertrans} for an explanation of the complex coordinate notation above. 

Before defining super annuli, we observe some features of the $2|1$-dimensional super Euclidean model geometry. There is a standard inclusion $\iota\colon \R^{1|1}\hookrightarrow \R^{2|1}$ generalizing the inclusion $\R\subset \C\cong \R^2$. In terms of $S$-points, this is 
\beq
\R^{1|1}(S) \ni (t,\theta)\stackrel{\iota}{\mapsto} (t,\bar t,\theta)\in \R^{2|1}(S)\label{eq:11to21}
\eeq
where $\bar t$ uses the real structure on $\R^{1|1}$. Post-composition of $\iota$ by left translation by $(\tau,\bar \tau,\theta)\in \E^{2|1}(S)$ gives a different embedding
\beq
\R^{1|1}(S)\stackrel{\iota }{\hookrightarrow }\R^{2|1}(S)\stackrel{T_{\tau,\bar \tau,\theta}}{\longrightarrow} \R^{2|1}(S)\qquad (t,\theta)\mapsto (t+\tau,\bar t+\bar\tau+i\theta\eta,\theta+\eta). \label{eq:11to212}
\eeq
Below, we will restrict to those embeddings gotten by \emph{positive} translations $(\tau,\bar\tau,\theta)\in \H{}^{2|1}(S)$ where $\H{}^{2|1}\subset \R^{2|1}$ is the super manifold gotten by restriction of the structure sheaf of $\R^{2|1}$ to the closed upper half plane, $\H\subset \C\cong \R^2\subset \R^{2|1}$. We observe that this standard embedding~$\iota\colon \R^{1|1}\hookrightarrow \R^{2|1}$ does \emph{not} induce a homomorphism of super Lie groups from~$\E^{1|1}$ to~$\E^{2|1}$, in contrast to the non-super case. However,  restricting~\eqref{eq:11to21} to~$\eta=0$ gives a homomorphism $\E\hookrightarrow \E^{2|1}$. For a fixed choice of $r\in \R_{>0}(S)\subset \E(S)$, this homomorphism determines a $\Z$-action on $S\times \R^{2|1}$ that on $S$-points is 
\beq
(z,\bar z,\theta)\mapsto (z+r,\bar z+\bar r,\theta), \label{eq:Zactionforannuli}
\eeq
where $\bar r$ uses the real structure on $\R_{>0}$. 
\begin{defn}
For a choice of $r\in \R_{>0}(S)$, define the \emph{infinite super annulus with circumference $r$} as the quotient $A_r^{2|1}:=(S\times \R^{2|1})/r\Z$ for $\Z$-action generated by \eqref{eq:Zactionforannuli}. Let $S^{1|1}_r:=(S\times \R^{1|1})/r\Z$ be the \emph{super circle of circumference $r$}. 
\end{defn}
Since the $\E$-action \eqref{eq:Zactionforannuli} on $\R^{2|1}$ commutes with the $\E^{2|1}$-action by super translations, \eqref{eq:11to21} and \eqref{eq:11to212} descend to maps into the infinite super annulus with circumference $r$
$$
S^{1|1}_r\stackrel{\iota}{\hookrightarrow} A_r^{2|1} \stackrel{T_{\tau,\bar\tau,\theta}\circ\iota}{\hookleftarrow}S^{1|1}_r
$$
that we call the \emph{standard embedding of a super circle} and the \emph{embedding of a super circle at $(\tau,\bar\tau,\theta)$}. Because of the $\Z$-quotient, the translation $T_{\tau,\bar\tau,\theta}$ is determined by a section of the bundle $(S\times \H{}^{2|1})/r\Z\to S$ where the $\Z$-action is~\eqref{eq:Zactionforannuli}. Equivalently, we have $(r,\tau,\bar\tau,\theta)\in ((\R_{>0}\times \H{}^{2|1})/\Z)(S)$ for the $\Z$-action 
$$
(r,\tau,\bar\tau,\theta)\mapsto (r,\tau+nr,\bar\tau+n\bar r,\theta)\quad\quad r\in \R_{>0}(S),\ (\tau,\bar\tau,\theta)\in \H{}^{2|1}(S).
$$

\begin{defn} A \emph{super annulus in $M$} is the data $(r,\tau,\bar\tau,\theta,\gamma)$ where $(r,\tau,\bar\tau,\theta)\in ((\R_{>0}\times \H{}^{2|1})/\Z)(S)$ is called the \emph{super modulus} and determines $A_r^{2|1}=(S\times \R^{2|1})/r\Z$ together with a pair of embedded super circles, and 
$$
\gamma\colon A_r^{2|1}\to M
$$
is a map. The \emph{source} and \emph{target super circles in $M$} are the compositions
$$
S^{1|1}_r\stackrel{\iota}{\hookrightarrow} A_r^{2|1} \stackrel{\gamma}{\to} M\qquad S^{1|1}_r\stackrel{T_{\tau,\bar\tau,\theta}\circ \iota}{\hookrightarrow} A_r^{2|1} \stackrel{\gamma}{\to} M.
$$
A \emph{constant super annulus in $M$} is one for which the map $\gamma$ is invariant under the precomposition action by $S$-points of $\E^2<\E^{2|1}$ on $A_r^{2|1}$. 
\end{defn}

\begin{defn} The \emph{presheaf of super annuli in $M$, denoted ${\rm sAnn}(M)$} is the presheaf whose value at~$S$ is the set of pairs $(r,\tau,\bar \tau,\theta)\in ((\R_{>0}\times \H{}^{2|1})/\Z)(S)$ and $\gamma\colon A^{2|1}_r\to M$. The \emph{presheaf of constant super annuli in $M$}, denoted ${\rm sAnn}_0(M)$, is the sub-presheaf where $\gamma$ is a constant super annulus in~$M$.  \end{defn}

\begin{lem}\label{lem:constantsuperann}
The presheaf of constant super annuli in $M$ is represented by the super manifold ${\rm sAnn}_0(M)\cong (\R_{>0}\times \H{}^{2|1})/\Z\times \SM(\R^{0|1},M)$. The source and target super circles determine morphisms of presheaves ${\rm sAnn}_0(M)\toto \R_{>0}\times \SM(\R^{0|1},M)$.
\end{lem}
\bp
An $\E^2$-invariant map $\gamma\colon A_r^{2|1}\to M$ can be identified with the composition
$$
\gamma\colon A_r^{2|1}= (S\times \R^{2|1})/r\Z\stackrel{\rm pr}{\to} S\times \R^{0|1}\stackrel{\gamma_0}{\to} M. 
$$
Hence, we can identity a constant super path with an $S$-point $(r,\tau,\bar\tau,\theta)\in (\R_{>0}\times \H{}^{2|1})/\Z$ and an $S$-point $\gamma_0\in \SM(\R^{0|1},M)(S)$.

The source and target super circles are given by $r\in \R_{>0}(S)$ and a map with a factorization,
$$
S^{1|1}_r \to A_r^{1|1} \to S\times \R^{0|1}\to M,
$$
and so are determined by $r\in \R_{>0}(S)$ and $S\times\R^{0|1}\to M$, as claimed. 
\ep

\subsection{An orientation-reversing map}
There is an anti-homomorphism
\beq
{\sf or}\colon \E^{2|1}\to (\overline{\E}^{2|1})^\op \qquad (z,\bar z,\eta)\mapsto (\bar z,z,-i\eta),\label{eq:antihomo}
\eeq
where $\overline\E^{2|1}$ has the conjugate group structure. 
This descends to a map of super annuli
$$
{\sf or}\colon A^{2|1}_r\to \bar{A}^{2|1}_r. 
$$
In a similar fashion to the $1|1$-dimensional case, we promote ${\sf or}$ to an orientation-reversing map on super annuli that exchanges the source and target super circles. Given a super annulus determined by $(r,\tau,\bar\tau,\theta)\in ((\R_{>0}\times \H{}^{2|1})/\Z)(S)$ and $\gamma\colon (S\times \R^{2|1})/r\Z\to M$, applying ${\sf or}$ gives a new pair of inclusions and a map to~$M$,
\begin{equation}
\begin{array}{c}
\begin{tikzpicture}
  \node (A) {$S^{1|1}_r$};
\node (B) [node distance=4cm, right of=A] {$\overline{A}^{2|1}_r$};
\node (C) [node distance=3cm, right of=B] {$M$};
\draw[->] (A) to[bend left=10] node [above] {$\iota$} (B);
\draw[->] (A) to[bend right=10] node [below] {$T_{\bar \tau,\tau,-i\theta}\circ \iota$} (B);
\draw[->] (B) to node [above] {$\bar\gamma\circ {\sf or}^{-1}$} (C);
\end{tikzpicture}
\end{array}\nonumber
\end{equation}
where $(r,\bar\tau,\tau,-i\theta)\in (\R_{>0}\times \overline{\E}{}^{2|1})/\Z(S)$, and we use real structures on $\R_{>0}$ and $\R^{1|1}$ to identify $\R_{>0}\cong \overline{\R}_{>0}$ and $S^{1|1}_r\cong \overline{S}^{1|1}_r$, respectively. To turn this data into a super annulus, we translate $(S\times \R^{2|1})/r\Z$ by $T_{-\bar \tau,-\tau,-i\theta}$, and we get
\begin{equation}
\begin{array}{c}
\begin{tikzpicture}
  \node (A) {$S^{1|1}_r$};
\node (B) [node distance=4cm, right of=A] {$\overline{A}^{2|1}_r$};
\node (C) [node distance=4cm, right of=B] {$M$};
\draw[->] (A) to[bend left=10] node [above] {$T_{-\bar \tau,-\tau,i\theta}\circ \iota$} (B);
\draw[->] (A) to[bend right=10] node [below] {$\iota$} (B);
\draw[->] (B) to node [above] {$\bar \gamma\circ {\sf or}^{-1}\circ T^{-1}_{-\bar \tau,-\tau,i\theta}$} (C);
\end{tikzpicture}
\end{array}\nonumber
\end{equation}
Then the new super annulus has modulus $(r,-\bar \tau,-\tau,i\theta)\in ((\R_{>0}\times \overline{\H}{}^{2|1})/\Z)(S)$. 

\begin{defn} Define the \emph{orientation-reversal} map ${\rm sAnn}(M)\to \overline{{\rm sAnn}(M)}$ that on $S$-points is 
$$
(r,\tau,\bar\tau,\theta,\gamma)\mapsto (r,-\bar \tau,-\tau,i\theta,\bar \gamma\circ {\sf or}^{-1}\circ T^{-1}_{-\bar\tau,-\tau,i\theta}),
$$
for $(r,\tau,\bar\tau,\theta)\in ((\R_{>0}\times \H{}^{2|1})/\Z)(S)$ and $\gamma\colon (S\times \R^{2|1})/r\Z\to M$.
\end{defn}

Specializing to constant super annuli, we observe the following. 

\begin{lem}\label{lem:21sigma1} The restriction of the time-reversal map to constant super annuli, ${\rm sAnn}_0(M)\to \overline{{\rm sAnn}_0(M)}$, is determined on $S$-points by the formula
\beq
&&(r,\tau,\bar\tau,\theta,x,\psi)\mapsto (r,-\bar\tau,-\tau,i\theta,\bar x+\theta\bar \psi,i\bar \psi) \label{eq:sigma211}
\eeq
for $(r,\tau,\bar\tau,\theta)\in ((\R_{>0}\times \H{}^{2|1})/\Z)(S)$ and $(x,\psi)\in \SM(\R^{0|1},M)(S)$, using the isomorphism ${\rm sAnn}_0(M)\cong (\R_{>0}\times \H{}^{2|1})/\Z\times \SM(\R^{0|1},M)$. 
\end{lem}

\subsection{The super Lie category of constant super annuli} Concatenation of super annuli works in an identical fashion to super paths, as in~\S\ref{sec:11cat}. In brief, for a pair of super annuli in~$M$ of the same circumference $r\in \R_{>0}(S)$, 
\begin{equation}
\begin{array}{c}
\begin{tikzpicture}
  \node (A) {$S^{1|1}_r$}; 
  \node (B) [node distance= 4cm, right of=A] {$A_r^{2|1}$};
  \node (C) [node distance = 3cm, right of=B] {$M$};
  \node (D) [node distance = 1.5cm, below of=A] {$S^{1|1}_r$}; 
  \node (E) [node distance = 1.5cm, below of=B] {$A_r^{2|1}$};
  \node (F) [node distance = 1.5cm, below of=C] {$M$};
  \draw[->] (A) to[bend right=10] node [below] {$T_{\tau,\bar\tau,\theta}\circ \iota$} (B);
    \draw[->] (A) to[bend left=10] node [above] {$\iota$} (B);
  \draw[->] (B) to node [above] {$\gamma$} (C);
  \draw[->] (D) to[bend right=10] node [below] {$T_{\tau',\bar\tau',\theta'}\circ \iota$} (E);
    \draw[->] (D) to[bend left=10] node [above] {$\iota$} (E);
  \draw[->] (E) to node [below] {$\gamma'$} (F);
\end{tikzpicture}\end{array}\nonumber
\end{equation}
and if $\gamma$ and $T^{-1}_{\tau,\bar\tau,\theta}\circ \gamma'$ agree in a neighborhood of the image of $\iota$, the concatenation is
\begin{equation}
\begin{array}{c}
\begin{tikzpicture}
  \node (A) {$S^{1|1}_r$}; 
  \node (B) [node distance= 4cm, right of=A] {$A_r^{2|1}$};
  \node (C) [node distance = 3cm, right of=B] {$M$};
  \draw[->] (A) to[bend right=20] node [below] {$T_{\tau+\tau'+i\theta\theta',\theta+\theta'}\circ \iota$} (B);
    \draw[->] (A) to[bend left=20] node [above] {$\iota$} (B);
  \draw[->] (B) to node [above] {$\gamma'*\gamma$} (C);
\end{tikzpicture}\end{array}\label{diag:concatannuli}
\end{equation}
where $\gamma'*\gamma$ is the map whose restriction to the appropriate open submanifolds of $A_r^{2|1}$ agrees with $\gamma$ and $T^{-1}_{\tau,\bar\tau,\theta}\circ \gamma'$. 

Below we only need to consider concatenation of constant super annuli. In this case, a pair of super annuli can be concatenated provided that their source and target super circles match, i.e., $T_{\tau,\bar\tau,\theta}\circ\gamma\circ \iota=\gamma\circ \iota$, with no additional condition. When this is the case, the concatenation is determined by~\eqref{diag:concatannuli} with~$\gamma'*\gamma=\gamma$. 

\begin{defn} 
Define the category of \emph{constant super annuli in $M$} as
$$
\Ann(M):=\left(\begin{array}{c}{\rm sAnn}_0(M) \\ \downarrow \downarrow \\ \R_{>0}\times \SM(\R^{0|1},M)\end{array}\right)\cong \left(\begin{array}{c} (\R_{>0}\times \H{}^{2|1})/\Z\times \SM(\R^{0|1},M) \\ \downarrow \downarrow \\ \R_{>0}\times \SM(\R^{0|1},M)\end{array}\right),
$$
where morphisms are constant super annuli in $M$. The source and target maps take the source and target super circle of a constant super annulus as in Lemma~\ref{lem:constantsuperann}. Composition is concatenation of constant super annuli, and the unit section picks out the constant super annulus associated with~$(r,0,0,0)\in (\R_{>0}\times \H{}^{2|1})/\Z$. 
\end{defn}

\begin{lem}\label{lem:21sigma} Let $\sigma_0\colon \R_{>0}\times \SM(\R^{0|1},M)\to \overline{\R_{>0}\times \SM(\R^{0|1},M)}$ be the map determined by~\eqref{eq:idegmap} composed with the real structure on $\R_{>0}\times\SM(\R^{0|1},M)$. Together with the time-reversal map from Lemma~\ref{lem:21sigma1}, this defines an anti-involution $\sigma\colon \Ann(M)\to \overline{\Ann(M)}^\op$. 
\end{lem}
\bp 
By the same argument as in the $1|1$-dimensional case, the description in Lemma~\ref{lem:21sigma1} shows that $\sigma_1$ and $\sigma_0$ have the claimed compatibility with source, target and unit maps. Compatibility with composition and units follows from~\eqref{eq:antihomo} being a homomorphism. This proves $\sigma$ is a functor. Lemma~\ref{lem:21sigma1} and~\eqref{eq:idegmap} show that $\sigma\circ \bar\sigma$ and $\bar\sigma\circ\sigma$ are both equal to the parity automorphism.
\ep

\subsection{Positive energy representations of super annuli}

\begin{defn}\label{defn:21sigma}
A \emph{unitary representation} of $\Ann(M)$ is a unitary representation (in the sense of Definition~\ref{defn:unitaryrep}) with respect to the functor $\sigma$ in Lemma~\ref{lem:21sigma}.
\end{defn}

\begin{defn} Define the \emph{rotation subgroupoid} $\Rot(M)\subset \Ann(M)$ as
$$
\Rot(M):=\left(\begin{array}{c} (\R_{>0}\times \E)/\Z\times M \\ \downarrow \downarrow \\ \R_{>0}\times M\end{array}\right)\subset \left(\begin{array}{c} (\R_{>0}\times \H{}^{2|1})/\Z\times \SM(\R^{0|1},M) \\ \downarrow \downarrow \\ \R_{>0}\times \SM(\R^{0|1},M)\end{array}\right).
$$
\end{defn}

Geometrically, $\Rot(M)$ is the groupoid whose objects are ordinary (not super) metrized circles with a constant map to $M$, and whose morphisms are rotations of those circles. Any representation of $\Ann(M)$ can be restricted to $\Rot(M)$, and this gives an action by these circle groups on the fibers of a vector bundle over $\R_{>0}\times M$. We observe that $\Rot(M)$ is a (non-super) Lie groupoid, so has a real structure and an inversion functor. This gives an unambiguous meaning to unitary representations of~$\Rot(M)$ without any additional choices. It turns out to be compatible with the restriction of unitary representation of $\Ann(M)$. 

\begin{lem} The restriction of $\sigma$ to $\Rot(M)$ coincides with the real structure composed with the inversion functor $\Rot(M)\to \overline{\Rot(M)}^\op$. \end{lem}

\bp The restriction of the action $(z,\bar z,\theta)\to (-\bar z,-z,-i\theta)$ to $\E\hookrightarrow \H^{2|1}$ is $t\mapsto -\overline{t}$, which is exactly inversion on~$\E$ composed with conjugation. This descends to fiberwise inversion on the family of circles $(\R_{>0}\times \E)/\Z$. Since the source and target maps are both projections, this coincides with inversion on the groupoid composed with conjugation. 
\ep

\begin{cor} The restriction of a unitary representation of $\Ann(M)$ to $\Rot(M)$ determines a unitary representation of the circles groups $\E/r\Z$ on the fibers of the hermitian super vector bundle~$V\to\R_{>0}\times M$. This decomposes~$V$ into an orthogonal sum of weight spaces, i.e., vector bundles~$V(k)\to \R_{>0}\times M$ on which the action is by~$e^{2\pi i kt/r}$. \end{cor}

\begin{defn} A \emph{positive energy representation} of $\Ann(M)$ is a unitary representation whose restriction to $\Rot(M)$ has finite dimensional weight spaces $V_k$ with weight bounded below, $V_k=\{0\}$ for $k\ll 0$. 
\end{defn}

\subsection{The proof of Theorem~\ref{thm:geocharof21rep}}

We are now in a place to give a geometric characterization of positive energy representations of $\Ann(M)$. We start be repackaging a representation in terms of maps between modules over $\Omega^\bullet(M)$ with specified properties.

\begin{lem}\label{lem:21modulemaps}
A representation $\rho$ of $\Ann(M)$ is determined by a representation of the super Lie category $\{(\R_{>0}\times \H^{2|1})/\Z\toto \R_{>0}\}$ on the trivial bundle over $\R_{>0}$ with fiber $C^\infty(\R_{>0})\otimes \Omega^\bullet(M;V)$ for $V\to M$ a super vector bundle, satisfing the additional compatibility 
\beq
\rho(r,\tau,\bar \tau,\theta)(f\cdot v)&=&(f+\theta df)\rho(r,\tau,\bar\tau,\theta)(v)\nonumber \\
 v\in \Omega^\bullet(M;V), && f\in \Omega^\bullet(M), \nonumber
\eeq
with respect to the $\Omega^\bullet(M)$-module structure on $\Omega^\bullet(M;V)$. In particular, for each $r\in \R_{>0}$, a representation of $\Ann(M)$ determines a representation of the super semigroup $\H^{2|1}/r\Z$ with the above property. 

An isomorphism between representations of $\Ann(M)$ is determined by a section of the trivial bundle over $\R_{>0}$ with fiber $\Omega^\bullet(M;{\rm Hom}(V,W))^\times$ that intertwines the representations of $\{(\R_{>0}\times \H^{2|1})/\Z\toto \R_{>0}\}$. 
\end{lem}

\bp
Let $V'\to \R_{>0}\times \SM(\R^{0|1},M)$ be a super vector bundle and $\rho$ a representation on $\End(V')$. Trivializing this bundle along the fibers of $\R_{>0}\times \SM(\R^{0|1},M)\to M$ is an isomorphism 
$$
\Gamma({\rm Ob}(\Ann(M));V')=\Gamma(\R_{>0}\times \SM(\R^{0|1},M);V')\cong C^\infty(\R_{>0})\otimes \Omega^\bullet(M;i^*V').
$$
for $i\colon M\to \R_{>0}\times \SM(\R^{0|1},M)$ determined by $1\in \R_{>0}$ and the canonical inclusion of the reduced manifold into $\SM(\R^{0|1},M)$. Let $V=i^*V'$ denote the resulting super vector bundle over~$M$. 
Then pulling this isomorphism back along the source map in $\Ann(M)$ (which is the projection) we get an isomorphism
$$
\Gamma({\rm Mor}(\Ann(M));s^*V')\cong C^\infty((\R_{>0}\times \H^{2|1})/\Z)\otimes \Omega^\bullet(M;V),
$$
where the right hand side has the obvious $C^\infty((\R_{>0}\times \H^{2|1})/\Z)\otimes \Omega^\bullet(M)$-module structure. 
The $C^\infty((\R_{>0}\times \H^{2|1})/\Z)\otimes \Omega^\bullet(M)$-module $\Gamma({\rm Mor}(\Ann(M));t^*V)$ associated with the target is twisted by the algebra isomorphism 
\beq
C^\infty((\R_{>0}\times \H^{2|1})/\Z)\otimes \Omega^\bullet(M)&\to& C^\infty((\R_{>0}\times \H^{2|1})/\Z)\otimes \Omega^\bullet(M),\label{eq:R21linearmod} \\
 g(r,\tau,\bar\tau, \theta)\mapsto g(r,\tau,\bar \tau,\theta), \ \ g\in C^\infty((\R_{>0}\times \H^{2|1})/\Z)&&\quad f\mapsto f+\theta df \ \ f\in \Omega^\bullet(M).\nonumber
\eeq

With the identifications above in place, a representation $\rho$ is a map of $C^\infty((\R_{>0}\times \H^{2|1})/\Z)\otimes \Omega^\bullet(M)$-modules,
$$
C^\infty((\R_{>0}\times \H^{2|1})/\Z)\otimes \Omega^\bullet(M;V)\to C^\infty((\R_{>0}\times \H^{2|1})/\Z)\otimes \Omega^\bullet(M;V), 
$$
with the standard module structure on the source, and the twisted module structure on the target. From~\eqref{eq:R21linearmod}, the map is linear over $C^\infty((\R_{>0}\times \H^{2|1})/\Z)$ and so defines a vector bundle map over $(\R_{>0}\times \H^{2|1})/\Z$ between the pullbacks of the trivial bundle over $\R_{>0}$ with fiber $\Omega^\bullet(M;V)$ along the source and target maps in $\{(\R_{>0}\times \H^{2|1})/\Z\toto \R_{>0}\}$. This gives a function $\rho(r,\tau,\bar\tau,\theta)$ on $(\R_{>0}\times \H^{2|1})/\Z$ valued in endomorphisms of the vector space $\Omega^\bullet(M;V)$ satisfying the remaining half of~\eqref{eq:R21linearmod} regarding the $\Omega^\bullet(M)$-module structure, and compatibility with composition,
$$
\rho(r,\tau,\bar\tau,\theta)\circ \rho(r,\tau',\bar\tau',\theta')=\rho(r,\tau+\tau',\bar\tau+\bar\tau'+i\theta\theta',\theta+\theta'). 
$$
The statement involving isomorphisms of representations follows from identical arguments to the $1|1$-dimensional case. This proves the lemma. 
\ep

\begin{lem} \label{lem:21semigroup}
For a positive energy representation $\rho\colon \Ann(M)\to \End(V)$, the $\Z$-grading by weight space of the $\Rot(M)$-action on the restriction of $V$ to $\R_{>0}\times M$ extends to a $\Z$-grading on $\Omega^\bullet(\R_{>0}\times \SM(\R^{0|1},M);V)$. Furthermore, the representation respects this grading, meaning that the section of $\Hom(s^*V,t^*V)$ associated with $\rho$ is a map of $\Z$-graded vector spaces. 
\end{lem}

\bp
The $\Z$-grading on the representation comes from the fiberwise action of 
$$
\sRot(M):=\left(\begin{array}{c} (\R_{>0}\times \E)/\Z\times \SM(\R^{0|1},M) \\ \downarrow \downarrow \\ \R_{>0}\times \SM(\R^{0|1},M)\end{array}\right)\subset \Ann(M)
$$
that has $\Rot(M)\subset \sRot(M)$ as its reduced subgroupoid. 

The first part of the claim is that the $\Rot(M)$-action on the restriction to $\R_{>0}\times M$ determines the $\sRot(M)$-action on the bundle over $\R_{>0}\times \SM(\R^{0|1},M)$. Using the descriptions of $\Gamma({\rm Mor}(\Ann(M));s^*V')$ and $\Gamma({\rm Mor}(\Ann(M));t^*V')$ from the previous lemma as $C^\infty((\R_{>0}\times \H^{2|1})/\Z)\otimes \Omega^\bullet(M;V)$ with a pair of $C^\infty((\R_{>0}\times \H^{2|1})/\Z)\otimes \Omega^\bullet(M)$-module structures, the linear map associated with the representation $\rho$ restricted to the subspace $C^\infty((\R_{>0}\times \E)/\Z)\otimes \Omega^\bullet(M;V)$ is a $\Omega^\bullet(M)$-module map. The positive energy condition concerns the further restriction to $C^\infty((\R_{>0}\times \E)/\Z)\otimes \Omega^0(M;V)$, but by virtue of being a $\Omega^\bullet(M)$-module map this completely determines the representation on the subspace $C^\infty((\R_{>0}\times \E)/\Z)\otimes \Omega^\bullet(M;V)$. Concretely, this is simply the extension of a $\Z$-grading on sections $\Omega^0(M;V)$ to a $\Z$-grading on $\Omega^\bullet(M;V)$. 

The second claim is that the $\Z$-grading determined by the $\sRot(M)$-action commutes with the semigroup action. This follows from $\sRot(M)$ being a central subgroupoid of $\Ann(M)$, meaning for $R$ an $S$-point of the space of morphisms of $\Rot(M)$ and $A$ an $S$-point of the space of morphisms of~$\Ann(M)$, $\rho(R)\circ \rho(A)\circ \rho(R)^{-1}=\rho(A)$. This means that the representation restricts to each finite-dimensional weight space of the $\sRot(M)$-action, and is determined by these restrictions. 
\ep

We adopt the notation $q=e^{2\pi i \tau}$, $\bar q=e^{-2\pi i \bar \tau}$, and $\im(\tau)=\frac{\tau-\bar\tau}{2i}$.

\begin{lem} For a finite-dimensional representation of $\Ann(M)$ on which the $\Rot(M)$-action has the fixed weight $k\in \Z$, the representation takes the form
$$
\rho(r,\tau,\bar\tau,\theta)=q^{L_0}\bar q^{\bar L_0}(1+\theta \A)=q^{k/r} e^{-2\im (\tau)\A^2+\theta \A}
$$ 
where $\A$ is a super connection, $L_0-\bar L_0=k/r$ and $2\pi \bar L_0=\A^2$. Isomorphisms between representations with this fixed weight are in bijection with super connection preserving isomorphisms of super vector bundles. 
\end{lem}

\bp
On restriction to the subspace $(\R_{>0}\times \hs)/\Z\times\SM(\R^{0|1},M) \subset (\R_{>0}\times \hs{}^{2|1})/\Z\times\SM(\R^{0|1},M)$ and by Lemma~\ref{lem:21semigroup}, we obtain a representation of $\{(\R_{>0}\times \H)/\Z\toto \R_{>0}\}$ on $\End(\Omega^\bullet(M;V))$, where the self-maps of $\Omega^\bullet(M;V)$ are $\Omega^\bullet(M)$-module maps. This means the representation takes values in sections of $\Omega^\bullet(M;{\rm End}(V))$. By the existence and uniqueness of solutions to differential equations we have
$$
\rho(r,\tau,\bar\tau,0)=q^{L_0}\bar q^{\bar L_0},
$$
for bundle endomorphisms $L_0,\bar L_0\in C^\infty(\R_{>0})\otimes \Omega^\bullet(M;{\rm End}(V))$. The positive energy condition requires that the action by $\E\subset \hs/r\Z$ is by $e^{2\pi i k(\tau-\bar\tau)/2r}$, and so $L_0-\bar L_0=k/r$. 

The full representation extends the above as
$$
\rho(r,\tau,\bar\tau,\theta)=q^{L_0}\bar q^{\bar L_0}(1+\theta\A)
$$
for an odd operator~$\A$. By the description afforded by Lemma~\ref{lem:21modulemaps}, we require
$$
q^{L_0}\bar q^{\bar L_0}(1+\theta\A)(f\cdot v)=(f+\theta df)\cdot q^{L_0}\bar q^{\bar L_0}(1+\theta\A)(v),  
$$
where $v\in \Omega^\bullet(M;V)$ and $f \in C^\infty(\R_{>0}\times\SM(\R^{0|1},M))$. From this we deduce that~$\A$ defines a super connection on $V$. Compatibility with composition demands
$$
q_1^{L_0} \bar q_1^{\bar L_0}(1+ \theta_1\A)q_2^{L_0}\bar q_2^{\bar L_0}(1+ \theta_2\A)=(q_1q_2)^{L_0}(\bar q_1\bar q_2)^{\bar L_0}e^{-2\pi i\theta_1\theta_2 \bar L_0}(1+(\theta_1+\theta_2)\A),
$$ 
which requires $2\pi \bar L_0=\A^2$. Using $L_0-\bar L_0=k/r$, we express this as 
\beq
\rho(r,\tau,\bar\tau,\theta)=q^{L_0}\bar q^{\bar L_0}(1+ \theta\A)&=&q^{L_0-\bar L_0+\bar L_0}\bar q^{\bar L_0}(1+\theta\A)\nonumber\\
&=&q^{k/2}e^{2\pi i (\tau-\bar \tau)\bar L_0}(1+\theta\A)\nonumber\\
&=& q^{k/r} e^{-4\pi \im(\tau)\bar L_0}(1+\theta\A)\nonumber \\
&=&q^{k/r} e^{-2\im (\tau)\A^2+\theta \A}.\label{eq:21rep}
\eeq
Unitarity of this representation is the equality
$$
i^{\deg} \bar q^{k/r} e^{-2\im (\tau)\A^2+i\theta \A}=(q^{k/r} e^{-2\im (\tau)\A^2+\theta \A})^*=\bar q^{k/r}(e^{-2\im (\tau)\A^2+\theta \A})^*
$$
where $i^{\deg}$ acts on a differential $m$-form by $i^m$. This follows from the (previously imposed) unitarity of the $\Rot(M)$ action and a condition on the super connections~$\A_k$. The computation in each degree $k$ is identical to the $1|1$-dimensional case, requiring $\A_k$ be a unitary super connection~\eqref{eq:11unitaryconcl}. 

Finally, we observe a super vector bundle isomorphism~$\varphi\in \Omega^\bullet(M;\Hom(V,W))^\times$ is compatible with the representations $\rho$ and $\rho'$ associated with super connections $\A$ and $\A'$ if and only if $\varphi(\A)=\A'$. 
\ep

\begin{proof}[Proof of Theorem~\ref{thm:geocharof21rep}] 
The positive energy condition means that a representation can be written as a (possibly infinite) direct sum of finite-dimensional representations, where on each of these the rotation subgroupoid acts by a fixed weight, with this weight being bounded below. From the previous lemma, this gives a decomposition
\beq
\rho(\tau,\bar\tau,\theta)=\bigoplus_{k>-N}^\infty q^{k/r} e^{-2\im (\tau)\A_k^2+\theta\A_k}\label{eq:annrepdecomp}
\eeq
for unitary super connections $\A_k$, proving the theorem.
\ep

\subsection{Grothendieck groups}

The following is an immediate corollary to Theorem~\ref{thm:geocharof21rep}; it is important because it implies that concordance of representations is an equivalence relation. 

\begin{cor} The prestack $M\mapsto \Rep(\Ann(M))$ is a stack. 
\end{cor}

Now we compute Grothendieck groups of representations of super annuli.

\begin{proof}[Proof of Theorem~\ref{thm:Tateeasy}]
By Theorem~\ref{thm:geocharof21rep}, a concordance class of a unitary representation of $\Ann(M)$ is the concordance class of a sequence of unitary super connections on a sequence of super vector bundles $V_k$ for $k>-N\in \Z$. Since the space of unitary super connections is affine, the set of concordance classes is the same as isomorphism classes of sequences of super vector bundles. The quotient of the free abelian group on super vector bundles by the subgroup generated by~\eqref{eq:Grothequiv} is exactly the K-theory of~$M$, and so sequences of such give $\K(M)\llbracket q \rrbracket [q^{-1}]\cong \K_{\rm Tate}(M)$.
\ep

\section{Differential elliptic cohomology at the Tate curve}\label{sec:diffellmain}

In this section we study the character theory for positive energy representations of $\Ann(M)$, culminating in the proof of Theorem~\ref{thm:Tate}. A priori, characters of finite-dimensional representations take values in the (nondegenerate) inertia groupoid, $\Lambda^{\rm nd}(\Ann(M))$. In complete parallel to the $1|1$-dimensional case for K-theory, functions on this stack are insufficiently rigid for our intended connection with a differential model for~$\K_\Tate$. This necessitates a refinement of the character theory, and the geometry of the intertia groupoid points us in the right direction. 

The inertia groupoid $\Lambda(\Ann(M))$ consists of constant super annuli in~$M$ with the same source and target super circle. Viewing these annuli as super tori suggests one consider additional automorphisms from super translation and dilations of super tori, which leads to a super double loop stack~$\widetilde{\mathcal{L}}^{2|1}_0(M)$. We define a rescaled partition function with values in holomorphic functions on~$\widetilde{\mathcal{L}}^{2|1}_0(M)$. This structure is what leads to the desired differential cocycle model for $\K_\Tate(M)$. 

The motivation regarding the behavior of partition functions under cutoffs again runs in complete parallel to~\S\ref{eq:11effpart}. The new feature is that we have chosen cutoffs for each weight space of the $\Rot(M)$-action, and the modifications to the partition function also decompose in this manner. 

\subsection{Dilations of super annuli and Bismut--Quillen rescaling}\label{sec:21Bismut--Quillen}

There is a dilation action on $\R^{2|1}$, 
$$
(z,\bar z,\theta)\mapsto (\mu^2 z,\bar \mu^2 \bar z,\bar\mu\theta),\qquad (z,\bar z,\theta)\in \R^{2|1}(S), \ \mu\in \R_{>0}(S)
$$
that descends to an action on $\E^{2|1}$ through group homomorphisms. Here we use the real structure on $\R_{>0}$. We promote this to a functor on constant super annuli, which is the \emph{renormalization group} (RG) action. Below, $\R_{>0}$ is the discrete super Lie category associated with the manifold $\R_{>0}$. 

\begin{defn} \label{defn:21RG}
Define a functor $\RG\colon \R_{>0}\times \Ann(M)\to \Ann(M)$ whose value on $S$-points of objects and morphisms is
$$
(\mu,r,x,\psi)\mapsto (\mu^2 r,x,\bar\mu^{-1}\psi)\qquad (\mu,\tau,\bar\tau,\theta,x,\psi)\mapsto (\mu^2 \tau,\bar\mu^2\bar \tau,\bar\mu\theta,x,\bar\mu^{-1}\psi)
$$
where $\mu\in \R_{>0}(S)$, $(r,\tau,\bar\tau,\theta)\in (\R_{>0}\times \H^{2|1})/\Z(S)$, and $(x,\psi)\in \SM(\R^{0|1},M)(S).$
\end{defn}

We observe that the diagram commutes,
\beq
\begin{array}{c}
\begin{tikzpicture}
  \node (A) {$\R_{>0}\times \R_{>0}\times \Ann(M)$};
  \node (B) [node distance= 6cm, right of=A] {$\R_{>0}\times \Ann(M)$};
  \node (C) [node distance = 1.5cm, below of=A] {$\R_{>0}\times \Ann(M)$};
\node (D) [node distance = 1.5cm, below of=B] {$\Ann(M)$};
  \draw[->] (A) to node [above] {$\id_{\R_{>0}}\times \RG$} (B);
  \draw[->] (A) to node [left] {$m\times \id_{\Ann(M)}$} (C);
  \draw[->] (B) to node [right] {$\RG$} (D);
  \draw[->] (C) to node [below] {$\RG$} (D);
\end{tikzpicture}\end{array}\nonumber
\eeq
where $m$ is the multiplication on~$\R_{>0}$, so that $\RG$ defines a strict $\R_{>0}$-action on $\Ann(M)$. Let $\RG_\mu$ be the restriction of $\RG$ to the subcategory $\{\mu\}\times \Ann(M)$, so $\RG_\mu\colon \Ann(M)\to \Ann(M)$ and $\RG_\mu\circ \RG_\lambda=\RG_{\mu\lambda}$.

Precomposing a representation with~$\RG_\mu$ leads to an $\R_{>0}$-action on the category of representations. We characterize this action on positive energy representations in terms of an action on sequences of super connections using Theorem~\ref{thm:geocharof21rep}. 

\begin{lem}\label{lem:21Bismut--Quillen}
For $\mu\in \R_{>0}$, the action of $\RG_\mu$ on a positive energy representation of $\Ann(M)$ associated with a sequence of unitary super connections $\A_k$ is the Bismut--Quillen rescaling action on each $\A_k$,
$$
\A_k\stackrel{\RG_\mu}{\mapsto} \mu \A_k(0)+\A_k(1)+\mu^{-1}\A_k(2)+\mu^{-2}\A_k(3)+\dots 
$$
where $\A_k(i)\colon \Omega^\bullet(M;V_k)\to \Omega^{\bullet+i}(M;V_k)$ is the degree $i$ piece of the super connection~$\A_k$. 
\end{lem}
\bp 
It suffices to check the claim on a representation for which $\Rot(M)$ acts by a fixed weight $k$. Precomposing by $\RG_\mu$, we find
\beq
q^{k/r} e^{-2\im (\tau)\A_k^2+\theta\A_k}&\stackrel{\RG_\mu}{\mapsto}& \exp\big(2\pi i \mu^2 \tau k/(\mu^2 r)\big) \nonumber \\
&&\cdot\exp\big(-2\mu^2 \im(\tau) \sum_j (\mu^{-i}\A_k(j))^2+\mu \theta \sum_j \mu^{-j}\A_k(j)\big)\nonumber \\
&=&q^{k/r}\exp\big(-2 \im(\tau) \sum_j (\mu^{-j+1}\A_k(j))^2+\theta \sum_j \mu^{-j+1}\A_k(j)\big)\nonumber
\eeq
which is the claimed Bismut--Quillen rescaling on the super connection.
\ep

\subsection{Inertia groupoid and super double loop stacks}
\begin{lem} The inerita groupoid of $\Ann(M)$ is 
$$
\Lambda(\Ann(M))\cong \left\{ \begin{array}{c} (\R_{>0}\times \E \times \H)/\Z \times \SM(\R^{0|1},M)\\ \downarrow\downarrow \\ (\R_{>0}\times \H)/\Z\times \SM(\R^{0|1},M)\end{array}\right\},
$$
where the source and target maps are both projections, and the $\Z$-quotient is by the action generated by
$$
(r,x,\tau,\bar\tau)\mapsto (r,x+r,\tau+r,\bar\tau+\bar r)\qquad r\in \R_{>0}, \ x\in \E(S), \ (\tau,\bar\tau)\in \H(S). 
$$
The nondegenerate inertia groupoid corresponds to the subspace of objects, 
$$
(\R_{>0}\times \mathfrak{H})/\Z\times \SM(\R^{0|1},M)\subset (\R_{>0}\times \H)/\Z\times \SM(\R^{0|1},M),
$$
i.e., the subspace $\mathfrak{H}\subset \H$ with $\im(\tau)>0$. 
\end{lem}

\bp The inertia groupoid of $\Ann(M)$ has as objects those constant super annuli with the same start and endpoint, so we require the largest subspace of the morphisms of $\Ann(M)$ on which the target map is the projection. This is 
$$
\Lambda(\Ann(M))_0=(\R_{>0}\times \H)/\Z\times \SM(\R^{0|1},M)
$$
To determine the morphisms of $\Lambda(\Ann(M))$, we observe that the invertible endomorphisms are exactly those corresponding to the image $\E\subset \H$ of the real axis (i.e., $\im(\tau)=0$). These act by conjugating a super annulus with a rotation. But this action is trivial so~$\Lambda(\Ann(M))$ is as claimed. This description of the invertible morphisms also yields the claimed nondegenerate inertia groupoid, corresponding to the subspace $\im(\tau)>0$. 
\ep

We take a moment to spell out the inertia groupoid explicitly in terms of the geometry of super annuli with maps to~$M$. An $S$-point of the objects is a map of constant super annuli
$$
\phi\colon A^{2|1}_r=(S\times \R^{2|1})/r\Z\to S\times \R^{0|1}\to M
$$
with the same source and target super circle, which is equivalent to invariance of the map~$\phi$ under the $\Z$-action generated by the family of translations $(\tau,\bar\tau)\in \E^{2|1}(S)$ that act on $S\times \R^{2|1}$. This means that the map above descends to the quotient,
\beq
(S\times \R^{2|1})/r\Z\oplus(\tau,\bar\tau)\Z \to S\times \R^{0|1}\to M,\label{eq:markedtori}
\eeq
When $\im(\tau)>0$ (corresponding to the nondegenerate inertia groupoid) the source of this map is a family of \emph{super tori}, meaning a quotient of $S\times \R^{2|1}$ by an $S$-family of lattices (we develop this more systematically in Definition~\ref{defn:superLLM}). Isomorphisms in the inertia groupoid consist of commuting triangles, 
\beq
\begin{tikzpicture}[baseline=(basepoint)];
\node (A) at (0,0) {$(S\times \R^{2|1})/r\Z\oplus(\tau,\bar\tau)\Z$};
\node (B) at (6,0) {$(S'\times \R^{2|1})/r\Z\oplus(\tau,\bar\tau)\Z$};
\node (C) at (3,-1.5) {$M$};
\draw[->] (A) to node [above=1pt] {$\cong$} (B);
\draw[->] (A) to node [left=.3cm]{$\phi$} (C);
\draw[->] (B) to node [right=.2cm]{$\phi'$} (C);
\path (0,-.75) coordinate (basepoint);
\end{tikzpicture}\label{21anntriangle2}
\eeq
where the top horizontal arrow rotates the annulus, with rotation determined by an $S$-point of $\E$. Super tori have additional automorphisms coming from super Euclidean isometries (see~\S\ref{sec:supertrans}), and the action of the renormalization group by dilations. The super Euclidean isometries are determined by $S$-points of $\E^{2|1}\rtimes \Z/2$, whereas the renormalization group is again the action by $\R_{>0}$. These combine to give an action by $\E^{2|1}\rtimes \R^\times$, acting by super translation of super tori and global dilations. We will ask that our (rescaled) partition functions are invariant under these additional symmetries, leading to the following definition. 

\begin{defn} 
Define the \emph{stack of constant super tori in $M$ with marked meridian,} denoted $\widetilde{\mathcal{L}}^{2|1}_0(M)$, as the super Lie groupoid,
$$
\widetilde{\mathcal{L}}^{2|1}_0(M):=\left(\begin{array}{c} (\E^{2|1}\rtimes \R^\times \times \R_{>0}\times \mathfrak{H})/\Z\times \SM(\R^{0|1},M)\\ \downarrow\downarrow \\
(\R_{>0} \times \mathfrak{H})/\Z\times \SM(\R^{0|1},M)\end{array}\right),
$$
where the quotient $(\E^{2|1}\rtimes \R^\times \times \R_{>0}\times \mathfrak{H})/\Z$ comes from the $\Z$-action generated by
\beq
(z,\bar z,\theta,\mu, r,\tau,\bar\tau)&\mapsto& (z+r,\bar z +\bar r,\theta,\mu,r,\tau+r,\bar\tau+\bar r)\nonumber\\
&& (z,\bar z,\theta)\in \E^{2|1}(S), \ \mu\in \R^\times(S), \ (\tau,\bar\tau)\in \mathfrak{H}(S), \ r\in \R_{>0}(S). \nonumber
\eeq
The source map for the groupoid is the projection, and the target map comes from an~$\E^{2|1}\rtimes \R^\times$-action. On $\SM(\R^{0|1},M)$ this action is through the homomorphism $\E^{2|1}\rtimes \R^\times\to \E^{0|1}\rtimes \R^\times$ and then the precomposition action on $\SM(\R^{0|1},M)$. The $\E^{2|1}\rtimes \R^\times$-action on $\R_{>0}\times \mathfrak{H}$ is through the homomorphism $\E^{2|1}\rtimes \R^\times\to \R^\times$ followed by the dilation action,
$$
\R^\times\times \R_{>0}\times\mathfrak{H}\to \R_{>0}\times \mathfrak{H}, \quad (\mu,\tau,\bar\tau,r)\mapsto (\mu^2\tau,\bar\mu^2\bar\tau,\bar\mu^2r). 
$$
\end{defn}

We observe that an $S$-point of objects of $\widetilde{\mathcal{L}}^{2|1}_0(M)$ gives a family of super tori with a map to $M$ as in~\eqref{eq:markedtori}. An $S$-point of morphisms gives a commuting triangle
\beq
\begin{tikzpicture}[baseline=(basepoint)];
\node (A) at (0,0) {$(S\times \R^{2|1})/r\Z\oplus(\tau,\bar\tau)\Z$};
\node (B) at (6,0) {$(S'\times \R^{2|1})/r'\Z\oplus(\tau',\bar\tau')\Z$};
\node (C) at (3,-1.5) {$M$};
\draw[->] (A) to node [above=1pt] {$\cong$} (B);
\draw[->] (A) to node [left=.3cm]{$\phi$} (C);
\draw[->] (B) to node [right=.2cm]{$\phi'$} (C);
\path (0,-.75) coordinate (basepoint);
\end{tikzpicture}\label{21anntriangle3}
\eeq
where the horizontal arrow is determined by an $S$-point of $\E^{2|1}\rtimes \R^\times$ which acts on the family of super tori by super translations and global dilations.

\begin{lem} \label{lem:L2algebra}
There is a natural isomorphism of algebras,
$$
C^\infty(\widetilde{\mathcal{L}}^{2|1}_0(M))\cong \Omega^\ev_\cl(M)\otimes C^\infty(\mathfrak{H}/\Z).
$$
\end{lem}
\bp Invariance under super translation $\E^{2|1}$ requires that the functions on objects be functions on $\R_{>0}\times \mathfrak{H}$ with values in closed differential forms on~$M$. Indeed, $\E^{2|1}$ acts on $\SM(\R^{0|1},M)$ through the projection homomorphism $\E^{2|1}\to \E^{0|1}$ followed by the precomposition action of $\E^{0|1}$ on $\R^{0|1}$. This action is generated by the de~Rham operator. Invariance under $\pm 1\in \R^\times$ requires that this differential form be even. It remains to analyze invariance under the action of~$\R_{>0}<\R^\times$. 

A slice for this $\R_{>0}$-action is $\{1\}\times \mathfrak{H}/\Z\times \SM(\R^{0|1},M)\subset (\R_{>0}\times \mathfrak{H})/\Z\times \SM(\R^{0|1},M)$, and hence an invariant function is determined by its restriction to this slice. This completes the proof.\ep

\begin{defn} \label{defn:holalgebra}
The \emph{Laurent polynomial subalgebra} of $C^\infty(\widetilde{\mathcal{L}}^{2|1}_0(M))$, is generated by the image of
\beq
f\otimes q^k\in \Omega^\bullet_\cl(M)\otimes C^\infty(\mathfrak{H}/\Z)\quad k\in \Z, \ q=e^{2\pi i \tau}
\eeq
under the isomorphism in Lemma~\ref{lem:L2algebra}, giving an injective map 
$$
\Omega^\ev_\cl(M)\otimes \C[q,q^{-1}]\hookrightarrow C^\infty(\widetilde{\mathcal{L}}^{2|1}_0(M)).
$$
Explicitly, the image of such a generator is $q^{k/r}\im(\tau)^{j/2}\otimes f\in C^\infty((\R_{>0} \times \mathfrak{H})/\Z)\otimes \Omega^j_\cl(M)\subset C^\infty(\R_{>0} \times \mathfrak{H}\times \SM(\R^{0|1},M))$. 

\end{defn}

\subsection{Rescaled partition functions of finite-dimensional representations}

\begin{lem} \label{lem:char21formula}
The equation for a character of a finite-dimensional representation of~$\Ann(M)$ is
$$
{\rm sTr}\left(\bigoplus_k^{\rm finite} q^{k/r} e^{-2\im (\tau)\A_k^2}\right)=\sum_k^{\rm finite} q^{k/r}{\rm sTr}(e^{-2\im (\tau)\A_k^2})\in C^\infty(\Lambda(\Ann(M))).
$$
\end{lem}
\bp
This is immediate from~\eqref{eq:21rep} and~\eqref{eq:annrepdecomp}. 
\ep

As in the $1|1$-dimensional case, this character is automatically invariant under translations of super tori, corresponding to the fact ${\rm sTr}(\exp(-\A_k^2))$ is a closed form for each~$k$. However, in general the character will not be dilation invariant. So to obtain a character map valued in functions on $\widetilde{\mathcal{L}}^{2|1}_0(M)$, we again use a form of the Bismut--Quillen rescaling. 

Consider the composition 
\beq
(\R_{>0}\times \mathfrak{H})/\Z\times \SM(\R^{0|1},M)&\to& \R_{>0}\times (\R_{>0}\times \mathfrak{H})/\Z\times \SM(\R^{0|1},M)\nonumber\\
&\stackrel{\RG}{\to}& (\R_{>0}\times \mathfrak{H})/\Z\times \SM(\R^{0|1},M)\label{eq:21rescale}\\
&=&{\rm Ob}(\Lambda(\Ann(M)))\nonumber
\eeq
where the first arrow is determined by $(r,\tau,\bar\tau)\mapsto (1/\im(\tau),r,\tau,\bar\tau)$, and $\RG$ denotes the restriction of the smooth functor $\RG$ to the subset of morphisms
$$
(\R_{>0}\times \mathfrak{H})/\Z\times \SM(\R^{0|1},M)={\rm Ob}(\Lambda(\Ann(M))) \subset {\rm Mor}(\Ann(M)).
$$

\begin{defn}\label{defn:21rescaleZ}
Consider the pullback of the section of the endomorphism bundle determined by a representation $\rho$ along the composition~\eqref{eq:21rescale}. For a finite-dimensional representation, define the \emph{rescaled partition function} $Z(\rho)$ as the super trace of this pullback 
$$
Z(\rho)\in C^\infty((\R_{>0}\times \mathfrak{H})/\Z\times \SM(\R^{0|1},M)).
$$
A bit more explicitly, at $(r,\tau,\bar\tau)\in (\R_{>0}\times \H)/\Z$ we have the formula
$$
Z(\rho)(r,\tau,\bar\tau)={\rm sTr}\Big((\rho\circ \RG_{1/\im(\tau)})(r,\tau,\bar\tau,0)\Big)
$$
\end{defn}

\begin{lem} \label{lem:rescaledchar21}
The rescaled partition function $Z(\rho)$ of a finite-dimensional representation of $\Ann(M)$ descends to the stack $\widetilde{\mathcal{L}}^{2|1}_0(M)$, and defines a function in the Laurent polynomial subalgebra. This gives a map
$$
Z\colon \Rep_{\rm fd}(\Ann(M))\to \Omega^\ev_\cl(M)\otimes \C[q,q^{-1}]\hookrightarrow C^\infty(\widetilde{\mathcal{L}}^{2|1}_0(M))
$$ 
from finite-dimensional representations of $\Ann(M)$ to this polynomial subalgebra. \end{lem}
\bp
By Lemmas~\ref{lem:21Bismut--Quillen} and~\ref{lem:char21formula}, a formula for the rescaled partition function is
\beq
Z(\rho)=\sum_k^{\rm finite} q^{k/r} {\rm sTr}\left(\exp\big(-2 \left(\sum_j (\im(\tau)^{j/2}\A_k(j)\right)^2\right)\nonumber 
\eeq
where~$\A_k(j)$ is the degree $j$ part of the super connection~$\A_k$ associated with the $k$th weight space of the $\Rot(M)$-action. This descends to a function on the constant super loop stack $\widetilde{\mathcal{L}}^{2|1}_0(M)$, and by Definition~\ref{defn:holalgebra} it lands in the Laurent polynomial subalgebra. 
\ep

For a general (possibly infinite-dimensional) positive energy representation, the super trace of the associated endomorphism over~$\Lambda^\nd(\Ann(M))$ need not converge. However, since the eigenspaces of the $\Rot(M)$-action are finite-dimensional representations of $\Ann(M)$, the super trace can always be understood in terms of a \emph{formal} sum of functions on $\widetilde{\mathcal{L}}^{2|1}_0(M)$.

\begin{defn} \label{defn:formalchar}
For a positive energy representation $\rho$ of $\Ann(M)$, define the \emph{formal rescaled partition function} as the formal sum
$$
Z(\rho):=\sum_{k>-N} Z(\rho_k)\mapsto \sum_k q^{k}{\rm sTr}(e^{-\A_k^2})\in \Omega^\bullet_\cl(M)\otimes \C\llbracket q \rrbracket [q^{-1}]
$$
where $\rho_k$ is the restriction of $\rho$ to the $k$th weight space of $\Rot(M)$, and the map applies Lemmas~\ref{lem:L2algebra} and~\ref{lem:rescaledchar21}.
\end{defn}

This gives the formal rescaled partition function
\beq
Z\colon \Rep(\Ann(M))\to \widehat{\mathcal{O}}(\widetilde{\mathcal{L}}^{2|1}_0(M)) \label{eq:rescalformal}
\eeq
where we use the notation $\widehat{\mathcal{O}}(\widetilde{\mathcal{L}}^{2|1}_0(M)):=\Omega^\ev_\cl(M)\otimes \C\llbracket q \rrbracket [q^{-1}]$ to emphasize the relationship between differential forms valued in Laurent series and the geometry of the constant super double loops. 

\subsection{Differential Grothendieck groups} 

\begin{proof}[Proof of Theorem~\ref{thm:Tate}]
First we spell out the data of differential concordance classes (see Definition~\ref{defn:diffconc}) of differential cocycles with respect to the formal rescaled partition function~\ref{eq:rescalformal}. Differential cocycles are
$$
{\widehat\Rep}{}(\Ann(M))=\{\rho\in \Rep(\Ann(M)),\ \alpha\in \Omega^\bullet_\cl(M\times \R)\otimes \C\llbracket q \rrbracket [q^{-1}] \mid i_0^*\alpha=Z(\rho)\}. 
$$
Unwinding the definitions, we have
$$
\widehat{Z}(\rho)-Z(\rho)=i^*_1\alpha-i^*_0\alpha\in \Omega^\ev_\cl(M)\otimes \C\llbracket q \rrbracket[q^{-1}]
$$

Now, by Theorem~\ref{thm:geocharof21rep}, $\rho$ is equivalent to a sequence of unitary super connections $\A_k$. We also decompose $\alpha$ as $\alpha=\sum_kq^k\alpha_k$. By Definition~\ref{defn:formalchar},  
$$
Z(\rho)=i^*_0\alpha=\sum_k q^ki_0^*\alpha_k= \sum_k q^k{\rm sTr}(e^{-\A_k^2})\in \Omega^\ev_{\cl}(M)\otimes \C\llbracket q \rrbracket [q^{-1}].
$$
So we make the identification
$$
{\widehat\Rep}{}(\Ann(M))\cong \{\{\A_k\}_{k>-N},\  \{ \alpha_k\} \mid \sum_kq^k{\rm sTr}(e^{-\A_k^2})=\sum_k q^k i_0^*\alpha_k\}
$$
We similarly identify the data of a differential concordance $(\widetilde\rho,\widetilde\alpha)$ (see Definition~\ref{defn:diffconc}) with a sequence $(\widetilde\A_k,\widetilde\alpha_k)$ for unitary super connections $\widetilde\A_k$ on a bundle over $M\times \R$, and $\widetilde\alpha_k\in \Omega^\bullet_\cl(M\times \R^2)$. We can compute the quotient by differential concordances for each~$k$ separately. By the argument from the $1|1$-dimensional case, this identifies differential concordance classes of the differential cocycles ${\widehat\Rep}{}(\Ann(M))$ with sequences of differential K-theory classes, and so
$$
\K({\widehat\Rep}{}(\Ann(M)))\cong \widehat{\K}(M)\llbracket q \rrbracket [q^{-1}]\cong \widehat{\K}_\Tate(M)
$$
proving the theorem.
\ep

\section{Free fermions, modular partition functions and $\KMF$}\label{sec:freeferKMF}

In this section, we consider positive energy representations of $\Ann(M)$ in a category of modules over the \emph{free fermion algebras},~$\Fer_n$. Roughly the free fermion algebra associated to a circle~$S^1_r=\R/r\Z$ and a vector space~$V$ is the Clifford algebra of~$V$-valued functions on $S^1_r$, and~$\Fer_n$ comes from~$V=\C^n$. There are various version of this, e.g., depending on which flavor of functions one considers (polynomial, smooth, $L^2$, etc.). A convenient choice for us is the restriction to super annuli of Stolz and Teichner's definition~\cite{ST11}. When~$M=\pt$, Stolz and Teichner show that $\Fer_n$-linear representations of~$\Ann(\pt)$ that are restrictions of $2|1$-Euclidean field theories have partition functions with values in a line bundle $\omega^{\otimes n/2}$ over super tori whose sections are weight $-n/2$ modular forms. This motivates a refinement of $\Fer_n$-linear representations of~$\Ann(M)$ to those whose rescaled partition functions are sections of such a line bundle. 

Let $\widetilde{\omega}^{\otimes n/2}$ denote the pullback of $\omega^{\otimes n/2}$ to $\widetilde{\mathcal{L}}^{2|1}_0(M)$. To start, we use the Clifford super trace to define a rescaled partition function
\beq
&&Z\colon \Rep^{\Fer_n}(\Ann(M))\to \widehat{\mathcal{O}}(\widetilde{\mathcal{L}}^{2|1}_0(M);\widetilde{\omega}^{\otimes n/2}) \cong \left\{ \begin{array}{ll} \Omega^\ev_\cl(M)\otimes \C\llbracket q \rrbracket [q^{-1}] & n \ {\rm even}\\ 
\Omega^\odd_\cl(M)\otimes \C\llbracket q \rrbracket [q^{-1}] & n \ {\rm odd}\end{array}\right. \label{eq:diffchar}
\eeq
that lands in closed differential forms valued in Laurent series. As in the previous section, we use the notation $\hat{\mathcal{O}}$ to denote the formal sums of sections over the stack~$\widetilde{\mathcal{L}}^{2|1}_0(M)$ coming from positive energy representations. The first way in which we refine these rescaled partition functions is to ask that the formal sums actually converge. 

\begin{defn} A $\Fer_n$-linear positive energy representation is \emph{trace class} if the formal character of its rescaled partition function defines a section~$\Gamma(\widetilde{\mathcal{L}}^{2|1}_0(M);\widetilde{\omega}^{\otimes n/2})$ via~\eqref{eq:diffchar}. Let $\Rep_{\rm TC}^{\Fer_n}(\Ann(M))\subset \Rep^{\Fer_n}(\Ann(M))$ denote the full subcategory of trace class $\Fer_n$-linear positive energy representations, and similarly $\widehat{\Rep}{}_{\rm TC}^{\Fer_n}(\Ann(M))\subset \widehat{\Rep}{}^{\Fer_n}(\Ann(M))$ denote the full subcategory of trace class differential cocycles. \end{defn}

With this trace class condition in place, we ask that rescaled partition functions possess extra symmetry defined in terms of descent to a section of a line bundle over a stack $\mathcal{L}^{2|1}_0(M)$. This stack consists of constant super tori over~$M$ \emph{without} a choice of meridian super circle. Hence, $\mathcal{L}^{2|1}_0(M)$ receives a map from $\widetilde{\mathcal{L}}^{2|1}_0(M)$. A square root of the Hodge line bundle over the moduli stack of elliptic curves determines line bundles $\omega^{\otimes n/2}$ over $\mathcal{L}^{2|1}_0(M)$. The pullback of $\omega^{\otimes n/2}$ to $\widetilde{\mathcal{L}}^{2|1}_0(M)$ is a line bundle $\widetilde\omega^{\otimes n/2}$, and we get a map on sections
\beq
\Gamma(\mathcal{L}^{2|1}_0(M);\omega^{\otimes n/2})\to \Gamma(\widetilde{\mathcal{L}}^{2|1}_0(M);\widetilde\omega^{\otimes n/2}),\label{eq:qexpand}
\eeq
that turns out to be injective. On a holomorphic subspace of sections,~\eqref{eq:qexpand} is induced by the $q$-expansion map for modular forms of weight $-n/2$. 

\begin{defn}\label{defn:degreek} A differential cocycle associated with a $\Fer_n$-linear, trace class, positive energy representation has \emph{degree $n$} if its rescaled partition function takes values in $\Gamma(\mathcal{L}^{2|1}_0(M);\omega^{\otimes n/2})$, i.e., is in the image of \eqref{eq:qexpand}. Let $\Rep^n_\MF(M)$ denote the category of degree $n$ representations, and $\widehat{\Rep}{}^n_\MF(M)$ denote the category of degree $n$ differential cocycles. \end{defn}

This is the main definition that goes into Theorem~\ref{thm:KMF}. We finish the section by translating Stolz and Teichner's periodicity theorem for $2|1\EFT^n(M)$ from~\cite[\S6]{ST11} into a Bott element $\beta\in \widehat\Rep{}^{-24}_\MF(\pt)$ that implements the $24$-periodicity of $\widehat{\K}_\MF$. 

\subsection{Super double loops stacks and differential cocycles for $\TMF(M)\otimes \C$}

\begin{defn} The \emph{$2|1$-dimensional rigid conformal isometry group} is $\E^{2|1}\rtimes \C^\times$, where~$\E^{2|1}$ is $\R^{2|1}$ as a super manifold with multiplication
$$
(z,\bar z,\theta)\cdot (z',\bar z',\theta')=(z+z',\bar z+\bar z'+i\theta\theta',\theta+\theta'), \quad (z,\bar z,\theta),(z',\bar z',\theta')\in \R^{2|1}(S),
$$
and the semidirect product $\E^{2|1}\rtimes \C^\times$ comes from the action $(z,\bar z,\theta)\mapsto (\lambda^2 z,\bar\lambda^2 \bar z,\bar\lambda \theta)$ for $(\lambda,\bar\lambda)\in \C^\times(S)$. We take the obvious left action of $\E^{2|1}\rtimes \C^\times$ on $\R^{2|1}$. \end{defn}

\begin{rmk} The rigid conformal isometry group above is the super Euclidean group $\E^{2|1}\rtimes \Spin(2)$ together with the renormalization group $\R_{>0}$, using the Lie group isomorphisms~$\Spin(2)\times \R_{>0}\cong U(1)\times \R_{>0}\cong \C^\times$; see~\S\ref{sec:supertrans}. 
\end{rmk}

A \emph{family of $2$-dimensional framed lattices} is an $S$-family of homomorphisms $\Lambda \colon S\times \Z^2\to S\times \R^2$ such that the ratio of the images of $S\times\{1,0\}$ and $S\times \{0,1\}$ under $\Lambda\colon S\times \Z^2\to S\times \R^2\cong S\times \C$ are in $S\times \mathfrak{H}\subset S\times \C$. Let $L$ denote the presheaf whose $S$-points are framed lattices; note that $L\cong \C^\times\times\mathfrak{H}$ is representable. Through the inclusion of groups $\E^2\subset \E^{2|1}$, an $S$-family of framed lattices defines a family of \emph{super tori} via the quotient $S\times \R^{2|1}/\Lambda=:S\times_\Lambda\R^{2|1}$. 

\begin{defn}\label{defn:superLLM}
 The \emph{super double loop stack of $M$}, denoted $\mathcal{L}{}^{2|1}(M)$, has as objects over~$S$ pairs $(\Lambda,\phi)$ where $\Lambda\in L(S)$ determines a family of super tori~$S\times_\Lambda \R^{2|1}$ and $\phi\colon S\times_\Lambda\R^{2|1}\to M$ is a map. Morphisms between these objects over~$S$ consist of commuting triangles
\beq
\begin{tikzpicture}[baseline=(basepoint)];
\node (A) at (0,0) {$S\times_\Lambda \R^{2|1}$};
\node (B) at (3,0) {$S'\times_{\Lambda'} \R^{2|1}$};
\node (C) at (1.5,-1.5) {$M$};
\draw[->] (A) to node [above=1pt] {$\cong$} (B);
\draw[->] (A) to node [left=1pt]{$\phi$} (C);
\draw[->] (B) to node [right=1pt]{$\phi'$} (C);
\path (0,-.75) coordinate (basepoint);
\end{tikzpicture}\label{21triangle}
\eeq
where the horizontal arrow is a map induced by the action of the rigid conformal isometry group. The stack of \emph{constant super tori}, denoted $\mathcal{L}{}^{2|1}_0(M)$, is the full substack for which~$\phi$ is invariant under the translational action of tori, i.e., $(\Lambda,\phi)$ is an $S$-point of $\mathcal{L}{}^{2|1}_0(M)$ if for all families of isometries associated with sections of the bundle of groups $S\times_\Lambda \E^2\to S$, the triangle \eqref{21triangle} commutes with $\Lambda=\Lambda'$ and $\phi=\phi'$. For a map $M\to M'$, postcomposition $S\times_\Lambda \R^{2|1}\to M\to M'$ defines morphisms of stacks $\mathcal{L}^{2|1}(M)\to \mathcal{L}^{2|1}(M')$ and $\mathcal{L}^{2|1}_0(M)\to \mathcal{L}^{2|1}_0(M')$ . 
\end{defn}

\begin{defn} Define an odd line bundle $\omega^{1/2}\to \mathcal{L}^{2|1}_0(M)$ via a functor $\mathcal{L}^{2|1}_0(M)\to \pt\sq \C^\times$ that assigns the trivial line bundle over~$S$ to any family of objects, and to a family of morphisms takes the line bundle automorphism coming from the map $S\to \C^\times$ in the definition of a rigid conformal isometry. Then $\omega^{1/2}$ is the pullback of the canonical odd line bundle over~$\pt\sq \C^\times$. Let $\omega^{\otimes n/2}:=(\omega^{1/2})^{\otimes n}$. \end{defn}

\begin{rmk} The line bundle $\omega^{1/2}$ is a version of the square root of the Hodge line bundle over elliptic curves: it has as sections functions on the moduli stack of tori that transform in the expected way under rotations and rescalings of the associated lattices. \end{rmk}

The stack $\mathcal{L}^{2|1}_0(M)$ has an atlas 
$$
L\times \SM(\R^{0|1},M)\twoheadrightarrow \mathcal{L}^{2|1}_0(M)
$$
that sends $\Lambda\in L(S)$ and $\phi_0\in \SM(\R^{0|1},M)(S)$ to 
$$
(S\times \R^{2|1})/\Lambda\stackrel{{\rm pr}}{\to} S\times \R^{0|1}\stackrel{\phi_0}{\to} M.
$$
Let $\vol\in C^\infty(L)$ be the function that assigns to a lattice the volume of the torus~$\R^2/\Lambda$. 

\begin{prop}[\cite{DBE_WG}]\label{prop:Maass}
A function $f\in C^\infty(L)\otimes \Omega^\bullet(M)\cong C^\infty(L\times \SM(\R^{0|1},M))$ descends to a section $\Gamma(\mathcal{L}^{2|1}_0(M),\omega^{\otimes n/2})$ if it can be written as a linear combination of functions of the form
$$
f=F \cdot \vol^{j/2} \otimes \alpha, 
$$
for $\alpha\in \Omega^{j}_{\cl}(M)$ and $F\in C^\infty(L)^{\SL_2(\Z)}$ a function satisfying $F(\mu\Lambda)=\mu^{n-j}F(\Lambda)$, i.e., $F$ is a weak Maass form of weight $(j-n)/2$. \end{prop}

Multiplication of sections results in a graded algebra, $\Gamma(\mathcal{L}^{2|1}_0(M);\omega^{\otimes \bullet/2})$. Rescaled partition functions land in a preferred subalgebra.

\begin{defn}
Define the \emph{holomorphic subalgebra} $\mathcal{O}(\mathcal{L}^{2|1}_0(M);\omega^{\otimes \bullet/2})\subset \Gamma(\mathcal{L}^{2|1}_0(M);\omega^{\otimes \bullet})$ as the image of
$$
\bigoplus_j \Omega^j_\cl(M)\otimes \MF^{n-j}  \hookrightarrow \Gamma(\mathcal{L}^{2|1}_0(M);\omega^{\otimes n/2})
$$
under the characterization of smooth sections in the previous Proposition, i.e., linear combinations of $F\cdot \vol^j\otimes \alpha$ where $F$ is a modular form of weight $(j-n)/2$. 
\end{defn}

This definition along with Proposition~\ref{prop:Maass} immediately yields the following. 

\begin{thm}[\cite{DBE_WG}]
There is a natural isomorphism of sheaves of graded algebras over~$\C$
$$
\mathcal{O}(\mathcal{L}_0^{2|1}(-);\omega^{\otimes \bullet/2})\stackrel{\sim}{\to}\bigoplus_{i+j=\bullet} \Omega^i_{\rm cl}(-)\otimes \MF^j
$$
whose target is the sheaf of closed differential forms valued in the graded ring $\MF^\bullet$ of weak modular forms. This realizes the source sheaf as a differential cocycle model for $\TMF\otimes \C$ in the sense of Hopkins--Singer~\cite{HopSing}. 
\end{thm}

\subsection{Forgetting a marked circle on a super torus as $q$-expansion}

An $S$-point of $\widetilde{\mathcal{L}}^{2|1}_0(M)$ defines a family of super tori with a map to $M$ by~\eqref{eq:markedtori}, and a morphism between $S$-points of $\widetilde{\mathcal{L}}^{2|1}_0(M)$ determines a fiberwise rigid conformal isometry between these families by~\eqref{21anntriangle3}. This rigid conformal isometry is determined by the inclusion of super Lie groups $\E^{2|1}\rtimes \R^\times<\E^{2|1}\rtimes \C^\times$. Together this defines a morphism of stacks
\beq
\widetilde{\mathcal{L}}^{2|1}_0(M)\to \mathcal{L}^{2|1}_0(M). \label{eq:qexpandstacks}
\eeq
Let $\widetilde{\omega}^{1/2}$ denote the pullback of $\omega^{1/2}$ along this map. 

\begin{lem} The restriction of the induced map on sections $\Gamma(\mathcal{L}^{2|1}_0(M);\omega^{\otimes n/2})\to \Gamma(\widetilde{\mathcal{L}}^{2|1}_0(M),\widetilde{\omega}^{\otimes n/2})$ to the holomorphic subalgebra
$$
\bigoplus_{i+j=n} \Omega^i_\cl(M)\otimes \MF^{j}\cong \mathcal{O}(\mathcal{L}^{2|1}_0(M);\omega^{\otimes n/2})\to \mathcal{O}(\widetilde{\mathcal{L}}^{2|1}_0(M);\widetilde{\omega}^{\otimes n/2})\cong \Omega^{\ev/\odd}_\cl(M)\otimes \C\llbracket q \rrbracket [q^{-1}]
$$
is determined by the $q$-expansion of modular forms, where the parity of the differential forms in the target agrees with~$n$. 
\end{lem} 
\bp
We compute the effect of this map on functions on atlases
$$
\R_{>0}\times \mathfrak{H}\times \SM(\R^{0|1},M)\to L\times \SM(\R^{0|1},M)
$$
determined by the map of stacks. Explicitly, this regards $r$ and $(\tau,\bar\tau)$ as defining a based lattice. Using a slice for the $\R_{>0}\subset \R^\times$-action (as in the proof of Lemma~\ref{lem:L2algebra}) the map on functions is uniquely determined if further restrict to $\{1\}\times \mathfrak{H}\times \SM(\R^{0|1},M)$ on the source. 

Next, the map on functions induced by the inclusion $\mathfrak{H}\times \SM(\R^{0|1},M)\hookrightarrow L\times \SM(\R^{0|1},M)$ is simply restriction. So $F(\tau)\cdot \vol^{j/2}\otimes \alpha$ restricts to $F(q)\cdot \im(\tau)^{j/2}\otimes \alpha$, which automatically descends to a function on the stack~$\widetilde{\mathcal{L}}^{2|1}_0(M)$. We identify this with the $q$-expansion of~$F$: the volume factor and closed differential form $\alpha$ are carried along for the ride. Hence, as claimed, the section in $\mathcal{O}(\mathcal{L}^{2|1}_0(M);\omega^{\otimes n/2})$ pulls back to $\Omega^\bullet(M)\otimes\C\llbracket q \rrbracket [q^{-1}]$, with the map implemented by the $q$-expansion of modular forms. Finally, since $\omega^{1/2}$ is an odd line bundle, the claim about the degree of the differential form in the target follows from the fact that the map on atlases preserves the parity of functions. 
\ep

\subsection{Free fermions, $\Fer_n$-linear representations and the fermion super trace}\label{sec:ferdef}

Below we recall the definition of the free fermions from~\cite[\S6]{ST11}. 
The \emph{restricted tensor product} $\otimes_{m\in \N} A_m$ of algebras consists of the closure of finite sums of tensor products $\otimes_m a_m$ for $a_m\in A_m$ where $a_m=1$ for all but finitely many~$m$. 

\begin{defn}[{\cite[Equation~6.1]{ST11}}] The algebra of \emph{$n$-free fermions} on $S^1_r=\R/r\Z$ is the restricted tensor product
$$
\Fer_n(r):=\left(\Cl_1\otimes \bigotimes_{m\in \N} \Cl(H(\C_m))\right)^{\otimes n}\cong\Cl_n\otimes \bigotimes_{m\in \N} \Cl(H(\C_m^n))
$$
where $\Cl(H(V))$ is the Clifford algebra of $V\oplus V^*\to \R_{>0}$ equipped with the canonical (hyperbolic) pairing of a vector space and its dual
$$
H((v,w),(v',w'))=v(w')+v'(w)\qquad v,v'\in V, \ w,w'\in V^*.
$$
Define an action of $\E/r\Z$ on $\Fer_n(r)$ through actions on each $\C_m$ by $e^{2\pi i mx/r}$ for $x\in \E/r\Z$.
\end{defn}

\begin{defn} \label{defn:2dfer}
Let ${\sf Fer}_n\to \Ann(M)$ be the algebra bundle that on objects $\R_{>0}={\rm Ob}(\Ann(M))$ has fiber $\Fer_n(r)$ over $r\in \R_{>0}$. Define the action of a morphism $(r,\tau,\bar\tau,\theta)\in (\R_{>0} \times \H^{2|1})/\Z(S)$ to be induced by the action on $\C_m$ by $e^{2\pi i m\tau}$ and the trivial action on~$\Cl_1$. This determines an isomorphism of algebra bundles $\phi\colon s^*\Fer_n\to t^*\Fer_n$ over ${\rm Mor}(\Ann(M))\cong (\R_{>0}\times\H^{2|1})/\Z$. 

Let ${\sf Fer}_n$ also denote the pullback of this bundle to $\Ann(M)$. 
\end{defn}

\begin{defn} A \emph{$\Fer_n$-linear representation} of $\Ann(M)$ is a unitary representation in self-adjoint ${\sf Fer}_n$-modules, $\Hom(\null_{{\Fer}_n}V,\null_{{\Fer}_n}V)$ using the isomorphisms between algebra bundles $s^*{\Fer}_n\to t^*{\Fer}_n$ in Definition~\ref{defn:2dfer}. A $\Fer_n$-linear representation is \emph{positive energy} if it has positive energy when forgetting the $\Fer_n$-actions. Let $\Rep^{\Fer_n}(\Ann(M))=:\Rep^n(\Ann(M))$ denote the category of positive energy $\Fer_n$-linear representations. 
\end{defn}

\begin{rmk} Elements of $\Fer_n$ have a weight corresponding to their behavior under the action by circle rotation. Hence, on the lowest energy space of a positive energy representation the ``half" of $\Fer_n$ corresponding to negative weight operators acts by zero. As such, we can view a $\Fer_n$-linear representation as having creation and annihilation operators corresponding to positive and negative Fourier modes in $C^\infty(\R/r\Z,\C^n)$. 
\end{rmk}

There is also an evident notion of a $\Cl_n$-linear representation of $\Ann(M)$.

\begin{defn} 
Let ${\sf Cl}_n\to \Ann(M)$ be the algebra bundle over the super Lie category that is the trivial algebra bundle $\underline{\Cl}_n$ over objects together with the identity isomorphism $s^*\Cl_n\cong t^*\Cl_n$ over~${\rm Mor}(\Ann(M))$. Let ${\sf Cl}_n$ also denote the pullback to~$\Ann(M)$. 
\end{defn}

\begin{defn} A \emph{Clifford-linear representation} of $\Ann(M)$ is a unitary representation in self-adjoint ${\sf Cl}_n$-modules, $\Hom(\null_{{\Cl}_n}V,\null_{{\Cl}_n}V)$ using the (identity) isomorphisms between algebra bundles $s^*{\Cl}_n\to t^*{\Cl}_n$ specified above. Let $\Rep^{\Cl_n}(\Ann(M))$ denote the category of $\Cl_n$-linear representations. 
\end{defn}

There is an evident inclusion of algebras $\Cl_n\hookrightarrow \Fer_n$. This extends to a map of algebra bundles over the objects of $\Ann(M)$ compatible with the isomorphisms of algebra bundles defined over the morphisms of $\Ann(M)$. This gives a functor 
\beq
\Rep^{{\sf Fer}_n}(\Ann(M))\to \Rep^{{\sf Cl}_n}(\Ann(M)), \label{eq:underCl}
\eeq
that simply restricts the (self-adjoint) $\Fer_n$-action to a (self-adjoint) $\Cl_n$-action. 

\begin{defn} The \emph{fermion super trace} of a ${\sf Fer}_n$-linear representation of~$\Ann(M)$ is the Clifford super trace of the underlying ${\sf Cl}_n$-linear representation under~\eqref{eq:underCl}.
\end{defn}

We extend the functor $\RG$ to $\Fer_n$-linear representations by pulling back the algebra bundle and algebra isomorphisms defining ${\sf Fer}_n\to \Ann(M)$ to $\R_{>0}\times \Ann(M)$ along the functor~$\RG$. The following is the evident generalization of Definition~\ref{defn:21rescaleZ}. 

\begin{defn} 
Consider the pullback of the section of the $\Fer_n$-linear endomorphism bundle determined by a $\Fer_n$-linear representation $\rho$ along the composition~\eqref{eq:21rescale}. Define the \emph{formal rescaled partition function} $Z(\rho)$ as the Clifford super trace of each weight space of this pullback 
$$
Z(\rho)=\sum_k {\rm sTr}_{\Cl_n}(\rho_k)\in C^\infty((\R_{>0}\times \mathfrak{H})/\Z\times \SM(\R^{0|1},M)).
$$
\end{defn}

\begin{lem} The formal rescaled partition
$$
Z\colon \Rep^n(\Ann(M))\to \Gamma(\widetilde{\mathcal{L}}^{2|1}_0(M);\widetilde\omega^{\otimes n/2})
$$ 
function takes values in the subspace $\Omega^{\ev/\odd}(M)\otimes \C\llbracket q \rrbracket [q^{-1}]$ as claimed in~\eqref{eq:diffchar}. \end{lem} 
\bp
Given a $\Fer_n$-linear representation, we can always forget the algebra action to recover an ordinary representation. This determines a sequence of unitary super connections. The $\Cl_n$-action determined by the $\Fer_n$-action preserves the weight spaces. Hence the sequence of super connections is in fact $\Cl_n$-linear. By Lemmas~\ref{lem:21Bismut--Quillen} and~\ref{lem:rescaledchar21}, a formula for the rescaled partition function is
\beq
Z(\rho)&=&\sum_k Z(\rho_k)=\sum_k q^{k/r} {\rm sTr}_{\Cl_n}\left(\exp\big(-2 \left(\sum_{i=0}^n (\im(\tau)^{i/2}\A_k(i)\right)^2\right)\nonumber\\
&=&\sum_k q^{k/r} {\rm sTr}\left(\Gamma\exp\big(-2 \left(\sum_{i=0}^n (\im(\tau)^{i/2}\A_k(i)\right)^2\right)\label{eq:rescal21part}
\eeq
where~$\A_k(i)$ is the degree $i$ part of the super connection~$\A_k$ associated with the $k$th weight space of $\Rot(M)$. This lands in the Laurent polynomial subalgebra by the same argument as in the proof of Lemma~\ref{lem:rescaledchar21}. It is an even or an odd function depending on the parity of~$n$ since this is the parity of the chirality operator defining the Clifford super trace. 
\ep

\subsection{Morita equivalences $\Fer_n\simeq \Cl_n$}

We observe that there is a Morita equivalence of algebras, $\Cl_n\simeq \Fer_n$, simply coming from the tensor product of Morita equivalences $\Cl(H(\C_m))\simeq \C$ for all $m$. Stolz and Teichner show that this extends to a projective bundle of Morita equivalences over~$\Ann(M)$, as we now review. 

\begin{lem}[Stolz--Teichner~{\cite[\S6]{ST11}}]\label{lem:ST}
There is a $\Cl_n$-$\Fer_n$ bimodule bundle $B\to {\rm Ob}(\Ann(\pt))$ implementing a fiberwise Morita equivalence $\Cl_n\simeq \Fer_n$ between algebra bundles. It carries a projective action by~$\Ann(\pt)$ whose character is the $-n$th power of the Dedekind eta function. 
\end{lem}
\begin{proof}[Proof sketch] We overview Stolz and Teichner's construction. In~\cite[Equation 6.2]{ST11} they define the bimodule bundle
\beq
B=\left(\bigotimes_{m\in \N} B_m\right)^{\otimes n},\label{eq:ferbimod}
\eeq
where the $B_m$ are irreducible $\Cl(H(\C_m)$-$\C$ bimodules over $\R_{>0}$. One would also hope for a bimodule map over ${\rm Mor}(\Ann(\pt))$ between the pullbacks of $B$ along source and target. However, this is a bit too much to ask for. Instead, Stolz and Teichner construct such a map over the covering of~$\Ann(\pt)$,
\beq
\left(\begin{array}{c} \R_{>0}\times \H^{2|1}\\ \downarrow\downarrow \\ \R_{>0}\end{array}\right)\to \left(\begin{array}{c} (\R_{>0}\times \H^{2|1})/\Z\\ \downarrow\downarrow \\ \R_{>0}\end{array}\right)=\Ann(\pt).\label{Eq:annulicover}
\eeq
Such a map $s^*B\to t^*B$ over $\R_{>0}\times \H^{2|1}$ is in particular a map of vector bundles compatible with composition. It is essentially determined by the $q$-expansion of its character, and for our purposes the character is the only information we need. Stolz and Teichner compute this~\cite[Equation~6.3]{ST11} to be $\eta(r,q)^{-n}$ where
$$
\eta(r,q)=q^{1/24r} \prod_j (1-q^{j/r})
$$
is the Dedekind $\eta$-function. Explicitly, this determines the representation of the source of~\eqref{Eq:annulicover} where each integer coefficient of~$q^{k/r}$ in power series expansion of~$\prod_k (1-q^{j/r})^{-n}$ defines a super dimension of a vector space in~$B$ that carries an action by $q^{k/r-n/24r}$. Unless $n$ is a power of 24, this only defines a \emph{projective} action by $\Ann(\pt)$. Equivalently, this defines an action by a 24-sheeted cover of $\Ann(\pt)$ associated with the $\Z$-covering~\eqref{Eq:annulicover}. 
\ep

The above has a simple corollary. 

\begin{cor}\label{cor:FertoCl}
Tensoring a representation with the Morita bimodule from the previous lemma gives an equivalence of categories, 
\beq
&&\Rep^{\Fer_n}(\Ann(M))\cong q^{n/24r}\Rep^{\Cl_n}(\Ann(M)),\qquad \null_{\Fer_n}V\mapsto \null_{\Cl_n}B_{\Fer_n}\otimes \null_{\Fer_n}V\label{eq:equivofcats}
\eeq
where we identify a $\Cl_n$-linear representation with a projective representation of the desired sort by shifting the annulus action by~$q^{n/24r}$. 
\end{cor}

\begin{rmk} These projective actions like $q^{L_0-c/24}$ are familiar to conformal field theorists, where $c$ is the central charge. \end{rmk}

\subsection{A model for differential $\KMF$}

\begin{proof}[Proof of Theorem~\ref{thm:KMF}]
A differential cocycle in $\widehat{\K}_\MF^{2n}(M)$ is a differential cocycle in $\widehat{\K}_\Tate(M)$ whose curvature takes values in closed differential forms with values in modular forms of the appropriate degree (see Definition~\ref{defn:KMFdiff}). It remains to describe the objects in Definition~\ref{defn:degreek} explicitly in terms of geometric data, and we do this step for all degrees (not necessarily even). So let $(\rho,\alpha)\in \widehat{\Rep}{}^n_\MF(\Ann(M))$ be a degree~$n$ differential cocycle. 

We start by unraveling the data of the representation~$\rho$. By Corollary~\ref{cor:FertoCl}, a $\Fer_n$-linear representation determines a $\Cl_n$-linear one. If we forget the $\Cl_n$-linear structure, we obtain a sequence of super vector bundles with super connections,~$\{V_k,\A_k\}$ for $k>-N$. Adding the $\Cl_n$-linear structure back in, each $V_k$ is a $\Cl_n$-module bundle over~$M$, and $\A_k$ is $\Cl_n$-linear. We calculate the rescaled partition function of the $\Fer_n$-linear representation associated with this $\Cl_n$-linear one using~\eqref{eq:rescal21part}, Lemma~\ref{lem:ST} and Corollary~\ref{cor:FertoCl}. We find
$$
Z(\rho)=\eta(q)^{-n}(q^{n/24})\sum_k q^k{\rm sTr}_{\Cl_n}(\exp(-\A^2))
$$
where $q^{n/24}$ is the modification to the $\Cl_n$-linear representation that makes it projective, and $\eta(q)^{n}$ comes from the super trace of the Morita bimodule~$B$ implementing the equivalence between $\Cl_n$- and $\Fer_n$-linear representations. The identity 
$$
\eta(q)q^{-1/24}=\prod_j (1-q^j)=\Phi(q)
$$
shows that this trace does indeed take values power series in~$q$: there are no fractional powers. 

The remaining data of a differential cocycle is a concordance $\alpha\in \Omega^{\ev/\odd}_\cl(M\times \R)\otimes \C\llbracket q \rrbracket [q^{-1}]$ with source~$Z(\rho)$, the rescaled partition function of~$\rho$. By virtue of defining a cocycle in $\widehat{\Rep}{}_\MF^n(\Ann(M))$, the target of the concordance~$\alpha$ defines a section of $\omega^{\otimes n/2}$
$$
{\widehat\Rep}{}_\MF^n(\Ann(M))\stackrel{\widehat{Z}}{\to} \mathcal{O}(\mathcal{L}^{2|1}_0(M);\omega^{\otimes n/2})\cong \bigoplus_{i+j=n} \Omega^i_\cl(M)\otimes \MF^j \twoheadrightarrow \TMF^{n}(M)\otimes \C,
$$
and hence a differential form with values in modular forms. 

To summarize, we have a sequence of Clifford module bundles with Clifford linear super connections~$\{V_k,\A_k\}$ for~$k\ge -N$. We get a differential form $Z(\rho)\in \Omega^\ev_\cl(M)\llbracket q \rrbracket [q^{-1}]$ by multiplying the power series of their Chern characters by~$\Phi(q)^{-n}$. Then we have a concordance $\alpha$ from $Z(\rho)$ to $\widehat{Z}(\rho)$, where this target is a closed differential form with values in modular forms. When $n$ is even, we can forget this additional property of the target of the concordance, and using the Morita equivalence $\Cl(2n)\simeq \C$ we get a cocycle in $\widehat{\K}_\Tate(M)$ by Theorem~\ref{thm:Tate}. This cocycle satisfies the additional property that the $q$-expansion of $\widehat{Z}(\rho,\alpha)$ is an integral form valued in modular forms. But this is exactly the data of a cocycle for $\widehat{\K}_\MF^{2n}(M)$.
\ep

The above proof actually gives a map $\widehat{\Rep}{}_\MF^n(\Ann(M))\to \widehat{\K}_\MF^n(M)$ for all~$n$. As usual, since not all odd K-theory classes can be represented by finite-dimensional Clifford module bundles (for $M\ne \pt$) we have a weakening. 

\begin{prop} The differential cocycles $\widehat{\Rep}{}^n_\MF(\Ann(M))$ map to $\widehat{\K}_\MF^n(M)$. This map is surjective when~$n$ is even or when $M=\pt$. \end{prop}

\subsection{24-periodicity}
Stolz and Teichner's bimodule in Lemma~\ref{lem:ST} has an honest action by $\Ann(M)$ when $n$ is divisible by~24. We use this to define a Bott element, which is a rephrasing of the ideas behind their periodicity theorem in~\cite[\S6]{ST11}. 

Consider the $\Fer(-24)$-linear representation of~$\Ann(\pt)$ given by the tensor product of the bimodule~$B^{-24}_{\Fer}$ implementing the Morita equivalence $\Fer(-24)\simeq \Cl(-24)$ and a $\Cl(-24)$-module $B^{24}_{\Cl}$ implementing the Morita equivalence $\Cl(-24)\simeq \C$. The action of annuli on the Morita bimodule is trivial on the Clifford module, and the (honest) action on~$B^{-24}_{\Fer}$ as described above. We compute its partition function as 
$$
Z(\beta)=\eta(r,q)^{24}=\left(q^{1/24r}\prod_j (1-q^{j/r})\right)^{24}=\Delta(r,q)\in \mathcal{O}(\mathcal{L}^{2|1}(M);\omega^{\otimes -24/2}),
$$
the modular discriminant. This is indeed a modular form of weight~$24/2=12$, and so this gives a cocycle as claimed. The Morita equivalences it implements along with the invertibility of $\Delta$ shows that the map
$$
\widehat{\Rep}^{\Fer_{n}}_\MF(M)\stackrel{\beta\otimes}{\longrightarrow} \widehat{\Rep}^{\Fer_{n-24}}_\MF(M)
$$
is invertible, so $\beta$ is a Bott class.

\section{The string orientation of~$\KMF$ and the supersymmetric sigma model}\label{sec:examples}

In this section we discuss the string orientation of~$\KMF$, explaining an analytic model for this orientation and its relationship to the supersymmetric sigma model. Witten's construction in~\cite{Witten_Dirac} can be viewed as a string orientation of~$\K_\Tate$, and the refinement of this orientation to~$\KMF$ makes the modularity of the associated (Witten) genus automatic. It remains to be seen whether there is a further refinement leading to an analytic orientation for~$\TMF$. A goal below is to set the stage for possible geometric refinements of the analytic orientation of~$\KMF$ motivated by field theories. Indeed, there are various  decorations one can imagine adding to the field theoretic data considered below. 

The string orientation of $\KMF$ is the composition of $\sigma\colon {\rm MString}\to {\rm TMF}$ and~${\rm TMF}\to \KMF$, using the string orientation of $\TMF$ that has been constructed homotopy theoretically~\cite{AHS,AHR}. One can also build the orientation of $\KMF$ analytically. The idea can be understood in terms of the diagram in spectra,
\begin{equation}
\begin{array}{c}
\begin{tikzpicture}[node distance=3.5cm,auto]
  \node (A) {$\KMF$};
  \node (B) [node distance= 4.5cm, right of=A] {$ {\rm K}_{\rm Tate}$};
  \node (C) [node distance = 1.5cm, below of=A] {${\rm H}_\MF$};
  \node (D) [node distance = 1.5cm, below of=B] {${\rm H}_\C\llbracket q \rrbracket [q^{-1}].$};
  \node (E) [node distance =2.5cm, left of =A] {$\null$};
  \node (F) [node distance=.75cm, above of=E] {${\rm TMF}$};
  \node (GG) [node distance=2.5cm, left of =F] {$\null$};
  \node (G) [node distance=.75cm, above of =GG] {${\rm MString}$};
  \draw[->] (A) to  (B);
  \draw[->] (A) to (C);
  \draw[->] (C) to (D);
    \draw[->] (B) to (D);
  \draw[->,bend right=20] (G) to node [below] {$\sigma_H$} (C);
  \draw[->,bend left=10] (G) to node [above] {$\sigma_\K$} (B);
  \draw[->] (F) to (A);
  \draw[->] (G) to node [above] {$\sigma$} (F);
  \draw[->,bend right=7] (F) to (C);
  \draw[->,bend left=4] (F) to (B);
\end{tikzpicture}\end{array}\nonumber
\end{equation}
There are well-known analytic models for the string orientation $\sigma_\K$ of $\K_\Tate$ and $\sigma_H$ of ${\rm H}_\MF$, so one obtains an orientation ${\rm MString}\to \K_\MF$ if these orientations can be chosen in a model where there is also a compatible homotopy in~${\rm H}_\C(M)\llbracket q \rrbracket [q^{-1}]$. We take the model for $\sigma_\K$ in terms of families of Dirac operators. For $\sigma_{\rm H}$, we use integration of differential forms modified by the Witten class. The compatible homotopy can then be made explicit through a choice of rational string structure. There is a lot of overlap in this description and the construction of secondary invariants of the Witten genus in~\cite{BunkeNaumann}, which is no accident: the map on coefficients ${\rm MString}^{-4k+1}(\pt)\to \KMF^{-4k+1}(\pt)\cong \C\llbracket q \rrbracket [q^{-1}]/\Z\llbracket q \rrbracket [q^{-1}]+\MF^{4k}$ is a complex $\K_\Tate$-variant of the Bunke--Naumann secondary invariants of the Witten genus. 

We can further refine this to a differential orientation. Given a geometric family of rational string manifolds $\pi\colon X\to M$ with fiber dimension~$2d$, we construct a class $\widehat{\sigma}(X)\in \widehat{\K}^{-2d}_\MF(M)$. When $X$ is even dimensional and $M=\pt$, this construction gives the Witten genus
$$
{\rm Wit}(X)=\widehat\sigma(X)\in \widehat{\K}_\MF^{-2d}(\pt)\cong \MF^{-2d}_\Z
$$
as a weight $d$ integral modular form. When $M=\pt$ and $X$ is odd dimensional, there is a version of this construction that gives the Bunke--Naumann invariant. We conclude the section by explaining the relationship between $\widehat{\sigma}(X)$ and a cutoff version of the sigma model. 

\subsection{The analytic string orientation for $\KMF$} 

In this subsection, we describe the analytic model for the string orientation of~$\KMF$ to lay the groundwork for the differential orientation constructed in the next subsection. Because of the definition of~$\KMF$ is as a homotopy pullback, we will need to keep track of representatives for the classes involved in a point set model for the relevant spectra when describing the string orientation in this way. Our model for~$\K_\Tate$ will involve families of Dirac operators, and we refer to~\S\ref{sec:FreedLott} for an overview of the Chern character of the Bismut super connection.

Let $\pi\colon X\to M$ be a geometric family of spin manifolds (see~\S\ref{sec:FreedLott}). If $\pi$ has a chosen Riemannian structure, a \emph{geometric $p_1$-trivialization} on~$\pi$ is a 3-form $H\in \Omega^3(X)$ such that $dH=p_1(T(X/M))$, using the metric to refine the first Pontryagin class to a differential form. A \emph{geometric rational string structure on~$\pi$} is a spin structure, Riemannian structure, and geometric $p_1$-trivialization. We also call $X\to M$ a \emph{geometric family of rational string manifolds}. 

We use the the model for K-theory consisting of families of generalized Dirac operators. This can be viewed as the model for K-theory underlying the differential model in~\cite{BunkeSchick}. Bismut's formula~\eqref{eq:BismutChernchar} defines a cocycle-level model for the Chern character of such a family. We work with de~Rham models for ${\rm H}_\MF$ and ${\rm H}_\C\llbracket q \rrbracket [q^{-1}]$. There are three steps in constructing the string orientation of~$\KMF$: (1) a cocycle level description of $\sigma_\K$, (2) a cocycle level description of $\sigma_{\rm H}$, and (3) a choice of homotopy between the images of these respective cocycles in ${\rm H}_\C\llbracket q \rrbracket [q^{-1}]$.

First we review the construction of $[\sigma_\K(X)]\in \K_\Tate^{-d}(M)$ in such a way that we obtain a cocycle representing this class. For a vector bundle $V$, let $S_{q^k}V$ denote the total symmetric power 
\beq
S_{q^k}V:=\underline{\C}\oplus q^k V\oplus q^k S^2V\oplus \cdots \oplus q^{kl} S^lV\oplus \cdots\label{eq:symq}
\eeq
The vector bundles $S_{q^k}(T_\C(X/M))$ on~$M$ can be used to twist the fiberwise spinor bundle, defining a sequence of Dirac operators over~$X$ that determine a class
$$
\left[(\slashed{D} \otimes \bigotimes_{k=1}^\infty S_{q^k}T_\C(X/M))\right]=\left[\sum_{k\ge 0}q^k \slashed{D}_k \right] \in \K^{-d}(M)\llbracket q\rrbracket.
$$
where $\slashed{D}\otimes R_k=\slashed{D}_k$ is the Dirac operator twisted by~$R_k$, for $R_k$ the coefficient of~$q^k$ in the formal sum of vector bundles~$S_{q^k}T_\C(X/M)$. To complete our definition of $\sigma_\K(X)$ (and ensure the desired modularity properties), we include the normalizing factor of $\eta^{d}(q)q^{-d/24}$ of the Dedekind eta function
\beq
&&[\sigma_\K(X)]=\left[\eta(q)^{d}q^{-d/24}\cdot \sum_{k\ge 0}q^k\slashed{D}_k \right]=\left[\Phi(q)^{d}\sum_{k\ge 0}q^k\slashed{D}_k \right]  \in \K^{-d}(M)\llbracket q \rrbracket [q^{-1}],\label{eq:Witfam}
\eeq
where $\Phi(q)=\prod_{j} (1-q^j)=q^{-1/24}\eta(q)\in \Z\llbracket q \rrbracket [q^{-1}]\cong \K_\Tate(\pt)$.

We have $[\sigma_{\rm H}(X)]\in  {\rm H}_\MF^{-d}(M)$ in the de~Rham model as the integral
\beq
&&[\sigma_{\rm H}(X)]=\left[\int_{X/M} \exp\left(\sum_{k\ge 2} \frac{E_{2k}(\tau) {\rm ph}_k(T(X/M))}{2k}\right)\right]\in  {\rm H}_\MF^{-d}(M)\label{eq:WitMFclass}
\eeq
where $E_{2k}\in \MF^{-4k}$ is the $2k^{\rm th}$ Eisenstein series with weight~$2k$ and ${\rm ph}_k(T(X/M))={\rm Ch}_k(T_\C(X/M))\in {\rm H}_\C^{4k}(M)$ is the $4k^{\rm th}$ component of the Pontryagin character. 

By~\eqref{eq:BismutChernchar} and Zagier's description of the multiplicative sequence defining the Witten genus~\cite{Zagier}, the Chern character of~\eqref{eq:Witfam} is represented by the closed differential form 
$$
\Phi(q)^{d}\int_{X/M} \hat{A}(X/M)\cdot {\rm Ch}\left(\bigotimes_{k=1}^\infty S_{q^k}T_\C (X/M)\right) = \int_{X/M} \exp\left(\sum_{k\ge 1} \frac{E_{2k}(q) {\rm ph}_k(T_\C (X/M))}{2k}\right)
$$
in $\bigoplus_j \Omega^{4j-d}(M)\llbracket q \rrbracket [q^{-1}]$. To obtain a class in $\KMF^{-d}(M)$, we need a homotopy in our de~Rham model for ${\rm H}_\C^{-d}(M)\llbracket q \rrbracket [q^{-1}]$ between the cocycle underlying the class above and the $q$-expansion of~\eqref{eq:WitMFclass}, by which we mean a differential form with values in power series in~$q$ that measures the difference between these cocycles. 

We compute the difference explicitly
$$
\exp\left(\sum_{k\ge 1} \frac{E_{2k}(q) {\rm ph}_k(T_\C (X/M))}{2k}\right)-\exp\left(\sum_{k\ge 2} \frac{E_{2k}(q) {\rm ph}_k(T_\C (X/M))}{2k}\right)=p_1(X/M)\wedge \Theta(X/M)(q)
$$
where 
$$
\Theta(X/M)(q)=\exp\left(\sum_{k\ge 2} \frac{E_{2k}(q) {\rm ph}_k(T_\C (X/M))}{2k}\right)\left(\frac{\exp(E_2p_1(X/M)/2)-1}{p_1}\right)
$$
is a closed form, and the second factor in the product on the right hand side uses that $\exp(E_2p_1/2)-1$ is (formally) divisible by~$p_1$. For a choice of rational string structure $p_1(X/M)=dH$ we have
\beq
d\int_{X/M}H\wedge \Theta(X/M)(q)=\int_{X/M} p_1(X/M)\wedge \Theta(X/M)(q)\label{eq:stringorhtpy}
\eeq
Hence, $\int_{X/M}(H\wedge \Theta(X/M)(q))\in \bigoplus_j \Omega^{4j-d+1}(M)\llbracket q \rrbracket [q^{-1}]$ yields a homotopy between the required classes in~${\rm H}^{-d}_\C(M)\llbracket q \rrbracket [q^{-1}]$. This produces a class in $[\sigma(X/M)]\in \KMF^{-d}(M)$ from the family of string manifolds $X\to M$ with geometric spin structures and rational string structures on the fibers. 

For odd-dimensional manifolds, if the chosen rational string structure comes from an \emph{integral} string structure, the associated invariant factors through~${\rm TMF}^{-2k+1}(\pt)$, so is necessarily torsion. Bunke and Naumann have constructed such invariants, showing they are nontrivial. We sketch how this fits into the story above. 

\begin{ex}[Invariants of $4k-1$-dimensional string manifolds]\label{Ex:BunkeNaumann}
Let $X$ be a geometric spin manifold of dimension $4k-1$. Let $H\in \Omega^3(X)$ be a choice of rational string structure (e.g., coming from an integral string structure). From the above, the class in $\KMF^{-4k+1}(\pt)$ is determined by the triple
$$
0\in \bigoplus_{i+j=-4k+1} \Omega^i(\pt;\MF^j),\qquad  \sum_k q^k {\rm Ind}(\slashed{D}_k)
$$
$$ 
\int_X H\wedge \Theta(X)\in \bigoplus_j \Omega^{4j}(\pt)\llbracket q \rrbracket [q^{-1}]=\C\llbracket q \rrbracket [q^{-1}].
$$
Because $\K_\Tate^\odd(\pt)=0$, we know there is a homotopy between ${\rm Ind}(\slashed{D}_k)$ and the zero vector space for all~$k$. This can be implemented by a deformation of $\slashed{D}_k$ to an invertible operator (called a \emph{taming} in~\cite{BunkeNaumann}). This necessarily modifies the homotopy by a sequence of eta invariants, $\{\eta_k\}$, and as a result the only data of the class is this homotopy 
\beq
&&\int_X H\wedge \Theta(X)+\sum_k q^k\eta_k\in \C\llbracket q \rrbracket [q^{-1}]/\Z\llbracket q \rrbracket [q^{-1}]+\MF^{4k}=\KMF^{-4k+1}(\pt).\label{eq:BunkeNaumanneta}
\eeq
This is a complex K-theory version of Bunke and Naumann's secondary analytic invariant~\cite[Definition~3.1]{BunkeNaumann}. 
\end{ex}

\subsection{A Freed--Lott differential orientation for $\KMF$ in even degree}

When $\pi\colon X\to M$ has even fiber dimension $2d$, we can refine the analytic orientation of~$\KMF$ from the previous section to a differential orientation that constructs a cocycle~$\widehat{\sigma}(X)\in \widehat{\K}_\MF^{-2d}(M)$ in our model. This is a straightforward modification of Freed and Lott's construction of the analytic pushforward in differential K-theory~\cite[\S7]{LottFreed} that we reviewed in~\S\ref{sec:FreedLott}. The basic idea is to find a finite-dimensional subbundle of each spinor bundle $\$\otimes R_k$ that contains the kernel of $\slashed{D}_k$, together with a differential form that mediates between Bismut's Chern character of $\slashed{D}_k$ and the Chern character of the finite-dimensional subbundle with its restricted super connection. 
This builds a differential cocycle~$\widehat{\sigma}_\K(X)\in \widehat{\K}(M)\llbracket q \rrbracket [q^{-1}]\cong \widehat{\K}_\Tate(M)$, and the compatibility homotopy is essentially the same as before. 

\begin{proof}[Proof of Theorem~\ref{thm:diffpush}]
We apply Lemma~\ref{lem:FreedLott} to each $\slashed{D}_k$ acting on $\$\otimes \R_k$ in~\eqref{eq:Witfam}
\beq
(\slashed{D}_k,\$\otimes R_k)\rightsquigarrow  (\A_k,V_k,\alpha_k),\label{eq:cutoff}
\eeq
yielding a sequence of finite-dimensional vector bundles $V_k\to M$ with super connections $\A_k$ and $\alpha_k\in \Omega^{\odd}(X)/d\Omega^{\ev}(X)$ such that
$$
d\alpha_k={\rm Ch}(\slashed{D}_k)-{\rm Ch}(\A_k)
$$
We also promote $\Phi(q)^{-d}\in \Z\llbracket q \rrbracket [q^{-1}]$ to a sequence of vector spaces $\bigoplus q^kF_k$ with ${\rm sdim}(F_k)$ equal to the (integer) coefficient of~$q^k$ (e.g., the bimodule from Lemma~\ref{lem:ST}). We promote this to a trivial bundle with trivial connection over~$M$, denoted $\widehat{\Phi}(q)^{-d}=\bigoplus q^k \underline{F}_k$. 
Now we define
\beq
&&\widehat\sigma_\K(X)=\Big(\widehat{\Phi}(q)^{d}\otimes \bigoplus_{k} q^k V_k,\widehat{\Phi}(q)^{d}\otimes \bigoplus_{k} \A_k ,\Phi(q)^{d}\sum q^k\alpha_k\Big)\in \widehat{\K}^{-d}(M)\llbracket q \rrbracket [q^{-1}],\nonumber
\eeq
which gives the differential refinement of $\sigma_\K(X)$. 

We take the differential cocycle model for ${\rm H}_\MF^{-d}(M)$ given by $\bigoplus_{i+j=-d} \Omega^i_\cl(M;\MF^j)$, and so~\eqref{eq:WitMFclass} has the obvious refinement, 
\beq
&&\widehat{\sigma}_{\rm H}(X)=\int_{X/M} \exp\left(\sum_{k\ge 2} \frac{E_{2k}(\tau) {\rm ph}_k(T(X/M))}{2k}\right)\in \bigoplus_{i+j=-d} \Omega^i_\cl(M;\MF^j)\label{eq:WitMFclassdiff}
\eeq

The curvature of $\widehat{\sigma}_\K(X)$ is 
\beq
\Phi(q)^{d}\sum_k ({\rm Ch}(\A_k)+d\alpha_k)&=&\Phi(q)^{d}\sum_k q^k{\rm Ch}(\slashed{D}_k)\nonumber\\
&=&\Phi(q)^{d}\int_{X/M} \sum_k q^k\hat{A}(X/M)\cdot {\rm Ch}\left(R_k\right) \nonumber \\
&=& \int_{X/M} \exp\left(\sum_{k\ge 1} \frac{E_{2k}(q) {\rm ph}_k(T_\C (X/M))}{2k}\right)\nonumber
\eeq
where the first equality is by the construction from Lemma~\ref{lem:FreedLott}, the second line is the local index theorem, and the third line uses Zagier's description of the Witten genus~\cite{Zagier}. So by the calculation in the previous subsection, the rational string structure~$H$ gives a concordance
$$
d\left(\lambda \int_{X/M}H\wedge \Theta(X/M)(q)\right)\in \Omega^\odd_\cl(M\times \R)\llbracket q\rrbracket[q^{-1}]
$$
where~$\lambda$ is a coordinate on~$\R$. The source of this concordance is the curvature of $\widehat{\sigma}_\K(X)$ and the target is the $q$-expansion of $\widehat{\sigma}_{\rm H}(X)$. We can repackage this as 
$$
\int_{X/M}H\wedge \Theta(X/M)(q)=\sum q^kh_k\in \Omega^\odd(M)\llbracket q\rrbracket[q^{-1}]
$$
and then 
$$
\widehat{\sigma}(X):=\Big(\widehat{\Phi}(q)^{d}\otimes \bigoplus_{k} q^k V_k,\widehat{\Phi}(q)^{d}\otimes \bigoplus_{k} \A_k,\Phi(q)^{d}\sum q^k(\alpha_k+h_k)\Big)\in \widehat{\K}_\MF^{-2d}(M)
$$
gives a differential cocycle in $\KMF$ associated with the geometric family of rational string manifolds $X\to M$. To be explicit, $\widehat{\Phi}(q)^{d}\otimes \bigoplus_{k} q^k V_k$ is a sequence of vector bundles on $M$, $\widehat{\Phi}(q)^{d}\otimes \bigoplus_{k} \A_k$ is a sequence of super connections, and $\Phi(q)^{d}\sum q^k(\alpha_k+h_k)$ modifies the Chern characters of these super connections in a manner that the result is a differential form valued in modular forms. 

We observe that when~$M=\pt$, the cocycle $\widehat{\sigma}(X)\in \widehat{\K}_\MF^{-2d}(\pt)$ is determined by a sequence of vector spaces: $\MF^\bullet$ is concentrated in even degrees, so in this case the differential form data is all zero. This identifies $\widehat{\sigma}(X)\in \MF^{-2d}_\Z$ with an integral modular form. Its image in modular forms over $\C$ is~\eqref{eq:WitMFclassdiff}, which is the Witten genus of~$X$. Hence, $\widehat{\sigma}(X)$ is the Witten genus as an integral modular form. 
\ep

\begin{rmk} The above construction works for $\pi\colon X\to M$ with odd fiber dimension, provided the existence of a finite-dimensional subbundle of the spinor bundle containing the kernel of the (Clifford linear) Dirac operator, as in Lemma~\ref{lem:FreedLott}. In particular, when $M=\pt$ there are no obstructions. \end{rmk}

\subsection{The differential pushforward as a cutoff supersymmetric sigma model}

Now we discuss how the above differential string orientation of~$\widehat{\K}_\MF$ can be understood (following Witten~\cite{Witten_Dirac}) in terms of a cutoff version of the supersymmetric sigma model. This is really just a translation from the mathematical objects of the previous subsection to their corresponding avatars in physics. 


The first step is to identify~$\widehat{\sigma}_\K(X)$ as an element in $\widehat{\Rep}{}^{-2d}(M)$, so in particular, we need a positive energy representation of super annuli. Consider the sequence of Dirac operators acting on the spinor bundles
\beq
&&\slashed{D}_{L(X/M)}:=\sum_{q\ge 0}q^k\slashed{D}_k \curvearrowright \$_{L(X/M)}:= q^{-2d/24}\left(\$\otimes \bigoplus q^kR_k\right).\label{eq:Witfam2}
\eeq
Witten constructed this by a version of Hamiltonian quantization applied to the (classical) supersymmetric sigma model with target~$X$. The factor of $q^{-2d/24}$ is important physically (related to a central charge), and is included in Witten's definition.  

When $M=\pt$, Witten explains how the operator~\eqref{eq:Witfam2} together with the circle action determined by the powers of~$q$ can be viewed as kind of low-energy approximation to the (super) time-evolution operator on the space of states of the supersymmetric sigma model with target~$X$. In families, this extends to give a representation of super annuli
$$
\rho(\tau,\bar\tau,\theta)=q^{-2d/24}\bigoplus_k q^k e^{-2{\rm im}(\tau)A_k^2+\theta A_k}
$$
where $A_k$ is the Bismut super connection whose degree zero piece is $\slashed{D}_k$. This representation isn't (finite-type) positive energy because the weight spaces are infinite-dimensional. 

However, Equation~\ref{eq:cutoff} and Lemma~\ref{lem:FreedLott} give a way to extract a positive energy representation by way of a cutoff version of the sigma model, meaning
\beq
&&\slashed{D}_{L(X/M)}^{\rm cutoff}=\bigoplus_{q\ge 0}\A_k' \curvearrowright \$_{L(X/M)}^{\rm cutoff}= q^{-2d/24}\bigoplus q^kV_k'.\label{eq:Witfamcut}
\eeq
where each $V_k'$ is a finite-dimensional vector bundle on~$M$ with a superconnection~$\A_k'$. By tensoring with the a Morita bimodule implementing the Morita equivalence $\Cl(-2d)\simeq \C$, we can promote the above to a sequence of Clifford module bundles $V_k$ with Clifford linear super connections $\A_k$. Equivalently, we could have started with $\slashed{D}_k$ being the Clifford-linear Dirac operators. 

Ignoring the $q^{-2d/24}$,~\eqref{eq:Witfamcut} defines a positive energy representation of~$\Ann(M)$, by Theorem~\ref{thm:geocharof21rep}. Including the of $q^{-2d/24}$, we get a \emph{projective} representation. However, it is exactly the sort of projective representation that when tensored with the Morita bimodule from Corollary~\ref{cor:FertoCl} yields a $\Fer_{-2d}$-linear representation. Hence, to a bundle of string manifolds~$X\to M$ we have constructed a positive energy, $\Fer(-2d)$-linear representation of super annuli, 
$$
\rho_X(\tau,\bar\tau,\theta)={\null_{\Fer_{-d}}}B \otimes \left(q^{-2d/24}\bigoplus_k q^k e^{-2{\rm im}(\tau)\A_k^2+\theta \A_k}\right)
$$
that we view as an $M$-family of time-evolution operators for the family of sigma models. 

Now we need to account for the change in the partition function associated with the chosen cutoff~\eqref{eq:Witfamcut}. These are the eta forms from Lemma~\ref{lem:FreedLott}, which combine to give $\sum_k q^k\alpha_k\in \Omega^{\odd}/d\Omega^{\ev}(M)\llbracket q\rrbracket[q^{-1}]\cong\widehat{\mathcal{O}}(\widetilde{\mathcal{L}}^{2|1}_0(M))$. This is a formal sum of functions that measures the difference between partition functions in the sense that
$$
Z\left(\slashed{D}_{L(X/M)}\right)-Z\left(\slashed{D}_{L(X/M)}^{\rm cutoff}\right)=d\left(\sum_k \alpha_k\right),
$$ 
where we understand $Z\left(\slashed{D}_{L(X/M)}\right)$ through the Chern character of Bismut's super connection. This sum of is exactly the data we require to promote the positive energy representation of~$\Ann(M)$ from~\eqref{eq:Witfamcut} to a differential cocycle, $\widehat{\sigma}_\K(X)\in \widehat{\Rep}{}^{-2d}(\Ann(X))$. By the calculation in the previous subsection, the rescaled partition function of this differential cocycle is 
\beq
\widehat{Z}(\sigma_\K(X))&:=&Z\left(\slashed{D}_{L(X/M)}^{\rm cutoff}\right)+d\left(\sum_k \alpha_k\right)\nonumber\\
&=&\int_{X/M} \exp\left(\sum_{k\ge 1} \frac{E_{2k}(q) {\rm ph}_k(T_\C (X/M))}{2k}\right)\in \mathcal{O}(\widetilde{\mathcal{L}}^{2|1}_0(M))\cong \Omega^\ev_\cl(M)\llbracket q \rrbracket [q^{-1}].\nonumber
\eeq

Next, we describe a cocycle $\widehat{\sigma}_{\rm H}(X)\in \mathcal{O}(\mathcal{L}^{2|1}_0(M);\omega^{\otimes -2d/2})$. In brief, this is just the section of the line bundle associated with the closed differential form on $M$ given by~\eqref{eq:WitMFclass}. However, this cocycle can be constructed as the 1-loop (quantum) partition function of the fiberwise supersymmetric sigma model; we carry this out in~\cite{DBE_WG} when $M=\pt$ and \cite{DBE_MQ} for general families. The procedure is the same as the construction of the $\hat{A}$-genus in the physics proof of the index theorem~\cite{Alvarez}, replacing $1|1$-dimensional classical mechanics with the $2|1$-dimensional classical sigma model. The stack $\mathcal{L}^{2|1}(M)$ is the space of fields, on which there is a function, the \emph{classical action}. The action vanishes on~$\mathcal{L}^{2|1}_0(M)$, identifying these constant super tori with a stack of \emph{classical vacua}. On the normal bundle to the inclusion~$\mathcal{L}_0^{2|1}(M)\subset\mathcal{L}^{2|1}(M)$ the Hessian of this classical action defines a family of invertible operators called \emph{kinetic operators}. When $p_1(X/M)=dH$ for a chosen 3-form~$H$, the $\zeta$-regularized super determinant of this family defines a function on $\mathcal{L}^{2|1}_0(X)$ that we can integrate along the fibers of $\mathcal{L}^{2|1}_0(X)\to \mathcal{L}^{2|1}_0(M)$, producing a function on $\mathcal{L}^{2|1}_0(M)$ that is~$\widehat{\sigma}_{\rm H}(X)$. This is exactly the computation of the partition function in the Lagrangian formalism. 

Finally, we describe the compatibility between the rescaled partition function $\widehat{Z}(\widehat{\sigma}_\K(X))$ from the Hamiltonian picture and the 1-loop partition function $\widehat{\sigma}_{\rm H}(X)$ from the Lagrangian picture. The difference between these is 
$$
\widehat{Z}(\widehat{\sigma}_\K(X))-\widehat{\sigma}_{\rm H}(X)=p_1(X/M)\cdot\Theta(X/M)\in \mathcal{O}(\widetilde{\mathcal{L}}^{2|1}_0(M))\cong \Omega^\ev_\cl(M))\llbracket q\rrbracket[q^{-1}]
$$ 
following~\eqref{eq:stringorhtpy}. Generally this difference is nonzero. However, when the family $X\to M$ has a rational string structure, a choice of $H$ with $dH=p_1(X/M)$ specifies a concordance from $p_1(X/M)\cdot\Theta(X/M)$ to zero, and so gives a concordance between the two versions of the partition function. Physically, this concordance is a trivialization of an anomaly: $\widehat{Z}(\widehat{\sigma}_\K(X))$ does not descend to the moduli stack of tori, but the concordance associated with a rational string structure identifies it with a function that \emph{does} descend.

\subsection{The odd case when $M=\pt$}
When $M=\pt$ and $X$ is odd dimensional, we can also interpret the Bunke--Naumann invariant in terms of the sigma model. In this case, the kernel of the Dirac operators give a cutoff theory, 
\beq
&&
\widehat{\sigma}_\K(X)=\left(\null_{\Fer_{-2d+1}}B\otimes \sum_{k\ge 0}q^k{\rm ker}(\slashed{D}_k),0,\Phi(q)^{2d-1}\sum q^k\alpha_k\right)  \in \widehat{\K}^{-2d+1}(\pt)\llbracket q \rrbracket [q^{-1}],\nonumber
\eeq
with the bimodule~$B$ from Lemma~\ref{lem:ST}. 
For our description, it is convenient to take~$\slashed{D}$ as the $\Cl(-2d+1)$-linear Dirac operator, so~${\rm ker}(\slashed{D}_k)$ is a (finite-dimensional) $\Cl(-2d+1)$-module. Because we have restricted to the kernel and $M=\pt$, the super connections $\A_k$ are the zero map. The Chern character is identically zero, as is the Chern character of the original Dirac operators. However, the eta forms~$\alpha_k$ mediating between these zeros need not vanish; they combine to give a power series, $\sum q^k\alpha_k\in \C\llbracket q \rrbracket [q^{-1}]$. A choice of rational string structure further modifies this power series, combining to give 
\beq
&&
\widehat{\sigma}(X)=\left(\null_{\Fer_{-2d+1}}B\otimes\sum_{k\ge 0}q^k{\rm ker}(\slashed{D}_k),0,\sum q^k(h_k+\alpha_k)\right)  \in \widehat{\K}_\MF^{-2d+1}(\pt).\label{eq:WitfamBN}
\eeq
We can identify the differential concordance class of this cutoff sigma model with an element of $\C\llbracket q \rrbracket [q^{-1}]/\Z\llbracket q \rrbracket [q^{-1}]+\MF^{-2d}$ by choosing an invertible odd linear map on each ${\rm ker}(\slashed{D}_k)$ that commutes with the Clifford action. Such maps always exist for $\Cl(-d)$-modules with~$d$ odd. This gives a (differential) concordance from $\widehat{\sigma}_\K(X)$ to a differential cocycle whose underlying representation of super annuli represents zero in the Grothendieck group. In the differential Grothendieck group, we need to remember the Chern--Simons for of the differential concordance. Together with the forms $\alpha_k$ above, this combines to give the forms~$\eta_k$ in~\eqref{eq:BunkeNaumanneta}: they are the Cheeger--Simons forms that measure the difference between the Chern character of the original Dirac operators~$\slashed{D}_k$ and the zero map on the zero vector space. 
This shows that the differential concordance class of $\widehat{\sigma}_\K(X)$ is determined by
$$
\null_{\Fer_{-2d+1}}B\otimes\sum q^k(h_k+\eta_k)\in \Omega^\ev(\pt)\llbracket q\rrbracket[q^{-1}]\cong \C\llbracket q\rrbracket[q^{-1}]\twoheadrightarrow \widehat{\K}^\odd_\MF(\pt)
$$
which is the Bunke--Naumann invariant (or rather, a complex K-theory version thereof). When the chosen string structure comes from an \emph{integral} string structure this invariant factors through~$\TMF^{-2d+1}(\pt)$, so is necessarily torsion. The above discussion gives physical meaning to these torsion invariants. They arise from a choice of trivializations of the anomaly of the supersymmetric sigma model with odd dimensional targets, together with a path in the space of field theories from the quantum sigma model to the zero theory.

\appendix

\section{Background miscellany}\label{appenA}

\subsection{Modular forms} \label{appen:MF}
\emph{Weight $n$ (weak) modular forms}, denoted $\MF_{2n}$, are the set of holomorphic functions on~$\mathfrak{H}\subset \C$ satisfying
$$
f\left(\frac{a\tau+b}{c\tau+d}\right)=(c\tau+d)^{n} f(\tau),\quad \left[\begin{array}{cc} a& b\\ c&d \end{array}\right]\in \SL_2(\Z),\quad \tau\in \mathfrak{H}.
$$
We define the half-weight modular forms to be the zero group~$\MF_{2n+1}=\{0\}$. Multiplication of functions assembles these abelian groups into a graded commutative ring. We define $\MF^\bullet=\MF_{-\bullet}$ to be the ring with the reversed grading. The action by the subgroup
$$
\Z\hookrightarrow \SL_2(\Z),\quad l\mapsto \left[\begin{array}{cc} 1 & l \\ 0 & 1\end{array}\right]
$$
allows us to identify a modular form with a function on the cylinder $\mathfrak{H}/\Z$ with properties. This permits a Fourier expansion called the \emph{$q$-expansion}, $\MF\to \C\llbracket q\rrbracket [q^{-1}]$, where $q=e^{2\pi i \tau}$. 

An alternate description of modular forms comes from viewing weight~$n$ weak modular forms as functions on framed lattices $\Lambda\subset \C$ with the property that $f(\mu\cdot\Lambda)=\mu^{-n}f(\Lambda)$ for $\mu\in \C^\times$. Using $\mu$ to set one of the lattice generators to be $1\in \C$ recovers the definition above.

\subsection{super manifolds}\label{sec:smfld} 
A \emph{$k|l$-dimensional supermanifold} is a locally ringed space whose structure sheaf is locally isomorphic to $C^\infty(U)\otimes_\C\Lambda^\bullet(\C^l)$ as a super algebra over~$\C$ for $U\subset \R^k$ an open submanifold. We follow the usual convention, writing the global sections of the structure sheaf of a supermanifold $N$ as $C^\infty(N)$, and referring to these global section as the (smooth) functions on the super manifold. Partitions of unity guarantee that maps between supermanifolds are determined by maps between the global sections of their structure sheaves. We denote the category of supermanifolds and maps of supermanifolds by~${\sf SMfld}$. In~\cite{DM} these supermanifolds are called $cs$-manifolds; apart from this terminological difference our conventions agree with theirs. 

\begin{rmk} Our use of super manifolds with algebras of functions over~$\C$ stems from the study of Wick-rotated field theories. Symmetries and various other basic ingredients don't make sense in the context of real super manifolds. See Example~4.9.3 of~\cite{DM} for a discussion.
\end{rmk}

For any supermanifold~$N$, there is a \emph{reduced manifold} we denote by $N_{\rm red}$ and a canonical map $N_{\rm red}\hookrightarrow N$ induced by the map of superalgebras $C^\infty(N)\to C^\infty(N)/I\cong C^\infty(N_{\rm red})$ where $I$ denotes the ideal of nilpotent elements in the structure sheaf of $N$. M.~Batchelor~\cite{batchelor} showed any supermanifold~$N$ is isomorphic to $(N_{\rm red},\Gamma(\Lambda^\bullet E^*))$ for $E\to N_{\rm red}$ a complex vector bundle over a smooth manifold~$N_{\rm red}$. We denote such a supermanifold by~$\Pi E$.

We sometimes use notation like $z,\bar{z}$ or $f,\bar{f}$ for elements of $C^\infty(N)$ that are complex conjugates in their image under the quotient $C^\infty(N)\to C^\infty(N_{\rm red})$. The main example comes from viewing the smooth manifold~$\C$ as a supermanifold. Then $S$-points can be described as
$$
\C(S)=\{a,b \in C^\infty(S)^{\ev}\mid a_{\rm red}=\overline{b}_{\rm red}\}
$$
where $a_{\rm red}$ and $b_{\rm red}$ denote the restriction of $a$ and $b$ to the reduced manifold~$S_{\rm red}$. Indeed, conjugation on functions only makes sense on this reduced manifold. In a slight abuse of notation, we write this pair of functions~$(a,b)$ as~$(a,\bar a)$. In particular, when working with $S$-points of $\R^{2|1}$, we implicitly use the identification $\R^2\cong \C$ and the above convention to write and $S$-point as $(z,\bar z,\theta)\in \R^{2|1}(S)$. One must exercise care when using this variant of the functor of points to define (e.g.) map $\C\to \C$: the reality condition on $(z,\bar z)$ needs to be preserved by the map on $S$-points. 

A \emph{vector bundle} over a super manifold $N$ is a locally free sheaf of modules over $C^\infty(N)$. We use the common notation~$\Gamma(E)$ for the module over~$C^\infty(N)$ defining a vector bundle~$E$. We caution that the vector space~$\Gamma(E)$ is typically quite different from the literal sections~$s\colon N\to E$ when there happens to be a candidate total space for the vector bundle~$E\to N$.

We frequently use results from~\cite{HKST} that identify structures on differential forms with the super geometry of the odd tangent bundle, $C^\infty(\Pi TM)\cong \Omega^\bullet(M)$. Most of this comes through the isomorphism
$$
\Pi TX(S)\cong \SM(\R^{0|1},X)(S):= {\sf SMfld}(S\times\R^{0|1},X)
$$
where $\SM(N,M)$ is the presheaf of sets on super manifolds defined by $S\mapsto {\sf SMfld}(S\times N,M)$. The above bijections show $\SM(\R^{0|1},X)$ is a representable presheaf, $\Pi TX\cong \SM(\R^{0|1},X)$. Furthermore, in this description the de~Rham operator---as an odd vector field on $\Pi TX$---is the derivative at the identity of the~$\R^{0|1}$-action on~$\SM(\R^{0|1},X)$ gotten from precomposition with the action of~$\R^{0|1}$ on itself by translations. Explicitly, this action is
\beq
\Omega^\bullet(X)\cong C^\infty(\Pi TX)&\to& C^\infty(\R^{0|1}\times \Pi TX)\cong \Omega^\bullet(X)[\theta],\nonumber\\
f&\mapsto&f-\theta df,\quad f\in \Omega^\bullet(X),\nonumber
\eeq
where $\theta$ is a coordinate function on~$\R^{0|1}$, and $d$ is the de~Rham operator. We also have an action of dilations $\R^\times\times \R^{0|1}\to \R^{0|1}$, which gives an action on the odd tangent bundle that on form is $f\mapsto \mu^{-j}f$ for $\mu\in \R^\times$ and $f\in \Omega^k(M)$. 

 \subsection{Stacks}
A smooth super stack is a category fibered in groupoids over supermanifolds satisfying descent, taking covers to be surjective submersions of supermanifolds. We will often drop the ``super" modifier. To any super manifold~$S$, a stack defines a groupoid and a map~$S\to S'$ gives a functor between these groupoids. 

A \emph{geometric stack} $\mathcal{X}$ admits at atlas $p\colon U\twoheadrightarrow \mathcal{X}$ for $U$ an ordinary super manifold. Being an atlas means that for all super manifolds $N$ and maps $q\colon N\to \mathcal{X}$, the weak 2-pullback in stacks,
\begin{equation}
\begin{array}{c}
\begin{tikzpicture}[node distance=3.5cm,auto]
  \node (A) {$U\times_{\mathcal{X}} N$};
  \node (B) [node distance= 3.5cm, right of=A] {$U$};
  \node (C) [node distance = 1.5cm, below of=A] {$N$};
  \node (D) [node distance = 1.5cm, below of=B] {$\mathcal{X}$};
  \draw[->] (A) to (B);
  \draw[->] (A) to (C);
  \draw[->] (C) to node [swap] {$q$} (D);
  \draw[->] (B) to node {$p$} (D);
\end{tikzpicture}\end{array}\nonumber
\end{equation}
is a super manifold and $U\times_\mathcal{X} N\to N$ is a surjective submersion. All the stacks considered in this paper are geometric. An atlas determines a groupoid presentation of a stack $\{U \times_\mathcal{X} U\toto U\}$. We often identify a geometric stack with a super Lie groupoids that presents it. 

Define a vector bundle $\mathcal{V}$ over a stack $\mathcal{X}$ in terms of $S$-points: to an object of $\mathcal{X}$ over~$S$, $\mathcal{V}$ assigns a module over $C^\infty(S)$ and to a map $S\to S'$, $\mathcal{V}$ assigns a map of $C^\infty(S)$-modules. In the case of the trivial bundle, this defines the algebra of functions on the stack. 

For a geometric stack with atlas~$U$, any vector bundle over the stack pulls back to~$U$. The sections of the vector bundle over the stack are precisely those sections over $U$ that are invariant under isomorphisms in the groupoid presentation determined by~$U$. This gives a concrete method of computation, which we will use throughout. 

\subsection{Super isometry groups in dimensions $1|1$ and $2|1$}\label{sec:supertrans}

In~\cite[\S6.3]{HST}, Hohnhold, Stolz and Teichner define model geometries for super manifolds. The definition requires one specify a super manifold ${\sf M}$ (called the \emph{model space}) and a super Lie group ${\sf G}$ (called the \emph{isometry group}) with an action of ${\sf G}$ on ${\sf M}$. Then super manifolds with an $({\sf M},{\sf G})$-geometry come from gluing open submanifolds of ${\sf M}$ along isometries in ${\sf G}$. We only use a small piece of their theory, but it provides a link between our framework and theirs so we spell out the relevant model spaces below. 

Let $\E^{1|1}$ denote the \emph{($1|1$-dimensional) super translation group} whose underlying super manifold is $\R^{1|1}$ and super group structure is
$$
\R^{1|1}(S)\times \R^{1|1}(S)\to \R^{1|1}(S),\quad (t,\theta)\cdot (s,\eta)=(t+s+\theta\eta,\theta+\eta), 
$$
for $(t,\theta),(s,\eta)\in \R^{1|1}(S).$ There $S$-points are akin to coordinates on $\R^{1|1}$. More carefully, $t\in \R(S)^{\ev}\cong C^\infty(S)^{\ev}$ and $\theta\in \R(S)^{\odd}\cong C^\infty(S)^{\odd}$, and so together $(t,\theta)$ determines a map $S\to \R^{1|1}$. The super manifold $\R^{1|1}$ has a real structure coming from complex conjugation of complex valued functions on~$\R$ and the isomorphism $C^\infty(\R^{1|1})\cong \C^\infty(\R)[\theta]$ (compare Example~67 in~\cite{HST}). 

Consider the action of $\Spin(1)\cong \Z/2=\{\pm 1\}$ on $\E^{1|1}$ by $(t,\theta)\mapsto (t,-\theta)$. The semidirect product $\E^{1|1}\rtimes \Z/2$ is the \emph{super Euclidean isometry group}, and its action on $\R^{1|1}$ defines the \emph{super Euclidean model space}. Similarly, consider the action of $\R^\times$ on $\E^{1|1}$ by $(t,\theta)\mapsto (\mu^2t,\mu\theta)$ for $\mu\in \R^\times(S)$. This defines the \emph{rigid conformal isometry group}, and its action on $\R^{1|1}$ defines the \emph{rigid conformal model space.}

Next, consider the \emph{($2|1$-dimensional) super translation group} $\E^{2|1}$ whose underlying super manifold is $\R^{2|1}$ with the super group structure
$$
\R^{2|1}(S)\times \R^{2|1}(S)\to \R^{2|1}(S),\quad (z,\bar z,\theta)\cdot (w,\bar w,\eta)=(z+w,\bar z+\bar w+i\theta\eta,\theta+\eta)
$$
for $(z,\bar z,\theta),(w,\bar w,\eta)\in \R^{2|1}(S)$. These $S$-points require a little clarification (alluded to in the previous subsection); we have $z,\bar z\in C^\infty(S)^\ev$ and $\theta\in C^\infty(S)^\odd$ with the requirement that $z$ and $\bar z$ are complex conjugates of one another \emph{only after} modding out by nilpotents, i.e., on restriction to the reduced manifold of~$S$. We emphasize that $z$ and $\bar z$ are not conjugate functions ons $S$; to make sense out of such a statement would require a real structure on~$S$. The real structure on $\R^{1|1}$ permits a map $\R^{1|1}(S)\to \R^{2|1}(S)$, $(t,\theta)\mapsto (t,\bar t,\theta)$ that is natural in $S$ and lifts the inclusion of the real axis~$\R\subset \C$ to these super manifolds. 

There is an action by ${\rm Spin}(2)\cong U(1)\subset \C^\times$ on $\E^{2|1}$,
$$
(\mu,\bar\mu)\cdot (z,\bar z,\theta)=(\mu^2 z,\bar \mu^2 \bar z,\bar \mu\theta),\quad (\mu,\bar\mu)\in \Spin(2)(S),\ (z,\bar z,\theta)\in \R^{2|1}(S).
$$
The semi-direct product $\E^{2|1}\rtimes {\rm Spin}(2)$ is the $2|1$-dimensional \emph{super Euclidean isometry group}. Its action on $\R^{2|1}$ defines the \emph{super Euclidean model space}. This extends to the obvious $\C^\times$-action on $\E^{2|1}$ by
$$
(\mu,\bar\mu)\cdot (z,\bar z,\theta)=(\mu z,\bar\mu \bar z,\bar\mu\theta),\quad (\mu,\bar\mu)\in \C^\times(S),\ (z,\bar z,\theta)\in \R^{2|1}(S).
$$
In the above, the notation $(\mu,\bar\mu)\in \C^\times(S)$ carries the same caveats as $(z,\bar z,\theta)\in \R^{2|1}(S)$. The semi-direct product $\E^{2|1}\rtimes \C^\times$ defines the \emph{rigid conformal isometry group} and its action on $\R^{2|1}$ defines the \emph{rigid conformal model space.}

\begin{rmk} The isomorphisms $\R^\times\cong \Z/2\times \R_{>0}$ and $\C^\times\cong \Spin(2)\times \R_{>0}$ allow us to view the rigid conformal groups above as the super Euclidean group together with isometries coming from global rescalings. In the context of field theories, this global rescaling action is exactly the renormalization group, and so these rigid conformal isometry groups are the combination of the super Euclidean group and this renormalization group action. \end{rmk}

\subsection{Super connections and Chern--Simons forms}\label{sec:CSback}

Let $E\to X$ be a super vector bundle. A \emph{(Quillen) super connection}~\cite{Quillensuper} on~$E$ is an odd linear map 
$$
\A\colon \Omega^\bullet(X;E)\to \Omega^\bullet(X;E)
$$
satisfying the Leibniz rule
$$
\A(f\cdot s)=df \cdot s+(-1)^{|f|}f \cdot \A(s),\quad s\in \Omega^\bullet(X;E), \ f\in \Omega^\bullet(X). 
$$

For super vector bundles with super connection $(E_0,\A^{E_0})$, $(E_1,\A^{E_1})$ on~$X$ and an isomorphism $\phi\colon E_0\to E_1$, a path $\A^\mu$ in the space of super connections with $\A^0=\A^{E_0}$ and $\A^1=A^{E_1}$ defines a \emph{Chern--Simons form} 
$$
\CS(\A^{E_0},\A^{E_1})=\int_{X\times I/X} {\rm Ch}(\A^\lambda)=\int_{X\times I/X} {\rm Tr}(e^{-(\A^\lambda)^2})
$$
where the integral is over the fibers of the projection $X\times I\to X$. A different choice of path changes the integral by an exact form, so the the data of the isomorphism~$\phi$ gives a well-defined class $\CS(\A^{E_0},\A^{E_1})\in \Omega^{\odd}(X)/d\Omega^{\ev}(X)$. This class measures the difference of the differential form valued Chern character:
$$
{\rm Ch}(\A^{E_1})-{\rm Ch}(\A^{E_0})=d\CS(\A^{E_0},\A^{E_1}).
$$
\subsection{Clifford super traces}\label{sec:ClsTr}
There are a few different normalizations for the super trace of a Clifford module in the literature. Our conventions agree with~\cite[\S3]{ST04}. Another helpful reference is~\cite[\S2]{MathaiQuillen}. The \emph{chirality operator} $\Gamma\in \Cl_n$ can be written in terms of an oriented orthonormal basis $\{e_1,\dots,e_n\}$ of $\R^n$ as
$$
\Gamma=i^{n/2}2^{-n/2}e_1\cdots e_n\in \Cl_n.
$$
The element $\Gamma\in \Cl_n$ is independent of this choice of basis. Our normalization for $\Gamma$ is chosen so that ${\rm sTr}(\Gamma\colon \$\to \$)=1$ for an irreducible Clifford module $\$$. Then for $A\colon \null_{\Cl_n}V\to \null_{\Cl_n}V$ a map of Clifford modules, define
$$
{\rm sTr}_{\Cl_n}(A):={\rm sTr}(\Gamma A\colon V\to V). 
$$
This super trace vanishes on super commutators, taking values in $\Cl_n/[\Cl_n,\Cl_n]$ which is a 1-dimensional even vector space when $n$ is even and a 1-dimensional odd vector space when $n$ is odd. 

\subsection{Concordances}\label{sec:conccat}\label{appen:conc}

The notion of concordance generalizes smooth homotopies to stacks.

\begin{defn} Let $\F$ be a stack. A pair of objects $\rho,\rho'\in \F(M)$ are \emph{concordant} if there exists an object $\widetilde{\rho}\in \F(M\times \R))$ and isomorphisms
\beq
i_0^*\widetilde{\rho}\cong \rho\qquad i_1^*\widetilde{\rho}\cong \rho'\label{eq:concisos}
\eeq
for $i_0,i_1\colon M\hookrightarrow M\times \R$ the inclusions at $0,1\in \R$ respectively. The data of $\widetilde{\rho}$ and the isomorphisms~\eqref{eq:concisos} is called a \emph{concordance with source $\rho$ and target $\rho'$}. 
\end{defn}

\begin{ex} For a representable stack associated to a smooth manifold $N$, sections over $S$ are smooth maps $S\to N$. A concordance between a pair of smooth maps is a smooth homotopy. \end{ex}

Using the stack property to glue sections shows that concordance is an equivalence relation, denoted $\sim_c$. The argument is identical to the one that shows smooth homotopy is an equivalence relation. When $M\mapsto \F(M)$ is an essentially small groupoid for each~$M$ and satisfies the stack condition, the assignment $M\mapsto \F(M)/{\sim_c}$ is presheaf that assigns the same map of sets to smoothly homotopic maps of manifolds. In particular, $\F(M)/{\sim_c}$ only depends on the smooth homotopy type of~$M$. This leads one to think of the natural map
$$
\F(M)\to \F(M)/{\sim_c}
$$
as a type of cocycle map, sending (possibly) geometric information about~$M$ to purely topological information. 

The following has an obvious generalization to stacks, but for our applications we only need it for sheaves of sets. 

\begin{defn} \label{defn:relconc}
Let $\F$ be a sheaf of sets. A pair of concordances $\alpha,\alpha'\in \F(M\times \R)$ are \emph{concordant rel boundary} if there is a section $\widetilde\alpha\in \F(M\times \R^2)$ whose restrictions satisfy $i_{y=0}^*\widetilde\alpha=\alpha$, $i_{y=1}^*\widetilde\alpha=\alpha'$, and $i_{x=0}^*\widetilde\alpha=p^*\alpha_0$ and $i_{x=1}^*\widetilde\alpha=p^*\alpha_1$, where the restrictions correspond to the inclusions $\R\hookrightarrow \R^2$ at $x=0$, $x=1$, $y=0$, and $y=1$ for the standard $(x,y)$-coordinates on $\R^2$, and $p\colon M\times \R\to M$ is the projection. \end{defn}

Note in particular that if $\alpha$ is concordant rel boundary to $\alpha'$, then these concordances have the same source and target.

\subsection{Concordances of differential forms}

\begin{lem}
Consider the sheaf $\Omega^k_\cl$ of closed differential forms. A pair of sections are concordant if and only if the closed differential forms are cohomologous. In particular, concordance classes of sections are de~Rham cohomology classes. 
\end{lem}
\bp
To see this, suppose we are given $\alpha_0,\alpha_1\in \Omega^k_{\cl}(M)$ with $\alpha_1-\alpha_0=d\beta$. Then there is a concordance $\alpha=\alpha_0+d(t\beta)\in \Omega^\bullet_\cl(M\times \R)$ from $\alpha_0$ to $\alpha_1$. Conversely, given a concordance $\alpha$ from $\alpha_0$ to $\alpha_1$, by Stokes theorem the integral of $\alpha$ over the fibers of $M\times I\to M$ satisfies
$$
d\int_{M\times I/M}\alpha=i_1^*\alpha-i_0^*\alpha=\alpha_1-\alpha_0,
$$
and defining this fiberwise integral to be $\beta\in \Omega^{k-1}(M)$ we have $\alpha_1-\alpha_0=d\beta$. 
\ep

\begin{lem} \label{lem:relconc}
The set of concordances $\alpha\in \Omega^k_\cl(M\times \R)$ with a fixed source $\alpha_0=i_0^*\alpha\in \Omega^k_\cl(M)$ up to concordance rel boundary is in bijection with the set $\Omega^{k-1}(M)/d\Omega^{k-2}(M)$ with the natural map given by fiberwise integration over $[0,1]\subset \R$, 
$$
I\colon \Omega^k_\cl(M\times \R)/{\sim} \longrightarrow \Omega^{k-1}(M)/d\Omega^{k-2}(M) \qquad \alpha\stackrel{I}{\mapsto} \int_{M\times I/M} \alpha
$$
\end{lem}

\bp 
First we show the map $I$ is well-defined. For $\widetilde\alpha\in \Omega^k_\cl(M\times \R^2)$ a concordance rel boundary between $\alpha$ and $\alpha'$, the fiberwise integral over $I^2\subset \R^2$ gives by Stokes theorem
$$
d\int_{M\times I^2/M} \widetilde{\alpha} =\int_{M\times I/M} \alpha'-\int_{M\times I/M}\alpha
$$
where we have used that constancy of the concordance on the other two boundary components of the rectangle $I^2$ imply that the associated integrals vanish. The integrals differ by an exact form, verifying the map is well-defined. 

Next we define a candidate inverse map~$J$. To make the formulas readable, we adapt the notation where, e.g., $\alpha_0$ denotes both a differential form on $M$ and the differential form on $M\times \R$ gotten by pulling back along the projection $M\times \R\to M$. For an equivalence class in the target, choose a representative $\beta\in \Omega^{k-1}(M)$ and take the concordance 
$$
J([\beta])=[\alpha_0+d(t\beta)]\in \Omega^k_\cl(M\times \R)/{\sim}.
$$
To see this is well-defined, if $\beta-\beta'=d\gamma\in \Omega^{k-1}(M)$, define the concordance rel boundary
$$
\widetilde\alpha=\alpha_0+d(t\beta)+d(s\gamma dt )=\alpha_0+d(t\beta+sd\gamma)+ds\gamma dt \in \Omega^k_\cl(M\times \R^2).
$$
When $s=0$ or $s=1$, we have that $ds$ pulls back to zero and so the pullback to $s=0$ is the concordance associated with $\beta$, and $s=1$ is the concordance associated with $\beta'$. On restriction to $t=0$, $dt$ pulls back to zero and we get the constant concordance on $\alpha_0$; similarly restricting to $t=1$ is the constant concordance on $\alpha_1$. This verifies that $J([\beta])=J([\beta'])$ and so $J$ is well-defined.

We have composition $I\circ J=\id$ by inspection, as it holds for the maps as defined before passing to equivalence classes. 

To see that $J\circ I$ is also the identity, first we claim that $\alpha-\alpha_0\in \Omega^\bullet_\cl(M\times \R)$ is exact: the restriction of the de~Rham cohomology of $M\times \R$ to $M\times \{0\}$ is an isomorphism, and the restriction of $\alpha-\alpha_0$ is the zero form. So suppose that $\alpha-\alpha_0=d\widetilde\beta$ for $\widetilde\beta\in \Omega^{k-1}(M\times \R)$. Consider
$$
\widetilde\alpha=\alpha-d(s\widetilde\beta)+d(st\beta)\in \Omega^\bullet_\cl(M\times \R^2).
$$
The restriction to $s=0$ is the original concordance, and the restriction to $s=1$ is the concordance $\alpha_0+d(t\beta)$, which is in the image of~$J$. The formula $\alpha-\alpha_0=d\widetilde\beta$ tells us that the restriction of $\widetilde\beta$ to $t=0$ is $0$ and the restriction to $t=1$ is $\alpha_1-\alpha_0=d\beta$. Hence, the restriction of $\widetilde\alpha$ to $t=0$ is $\alpha_0$ and the restriction to $t=1$ is $\alpha_1$. This verifies the lemma.
\ep

\subsection{Differential (Tate) K-theory} \label{sec:backKthy}
In this paper we use explicit models for differential cohomology theories, all of which will be elaborations on Klonoff's model for differential K-theory from~\cite{Klonoff}, Section~4.1. 

For a smooth manifold~$M$, consider a groupoid $\mathcal{V}(M)$ whose objects triples $(V,\A,\alpha)$ for $V$ a super vector bundle on~$M$, $\A$ a super connection, and $\alpha\in \Omega^{\odd}(M)/d\Omega^{\ev}(X)$. Define a morphism from $(V,\A,\alpha)$ to $(V',\A',\alpha')$ to be an isomorphism $\phi\colon V\to V'$ of super vector bundles such that
$$
\alpha=\alpha'+\CS(\A,\phi^*\A').
$$
Let $\mathcal{F}(M)$ denote the free abelian group on isomorphism classes in~$\mathcal{V}(M)$. Let $\mathcal{Z}(M)$ denote Consider the subgroup generated by the elements
$$
(V\oplus V',\A\oplus\A',\alpha+\alpha')-(V,\A,\alpha)+ (V',\A',\alpha')\qquad {\rm and} \qquad (V\oplus \Pi V,\A\oplus \Pi \A,0)
$$
where $(\Pi V,\Pi \A)$ denotes the parity reversal of the super vector bundle $(V,\A)$. \

\begin{defn} 
Define \emph{the differential K-theory of $M$} to be the abelian group
$$
\widehat{\K}(M):=\mathcal{F}(M)/\mathcal{Z}(M). 
$$
This has a ring structure with multiplication 
$$
[V,\A,\alpha]\cdot[V',\A',\alpha']=[V\otimes V',\A\otimes \A',\alpha \wedge {\rm Ch}(\A')+{\rm Ch}(\A)\wedge \alpha'+\alpha\wedge d\alpha'].
$$
The \emph{curvature map} is 
\beq
R\colon\widehat{\K}(M)\to \Omega^\ev_{\cl}(M),\qquad R(V,\A,\alpha)= {\rm Ch}(\A)+d\alpha.\label{eq:Kthycurv}
\eeq
\end{defn} 

\begin{rmk}
Note that we can repackage the data of a differential cocycle as $(V,\A,\widetilde\alpha)$ where $\widetilde\alpha$ is a concordance, $\widetilde{\alpha}\in \Omega_{\cl}^\ev(M\times \R)$ with source ${\rm Ch}(\A)$ and target $R(V,\A,\alpha)$. 
\end{rmk}

\begin{defn}\label{defn:KTatediff}
Define \emph{the differential elliptic cohomology at the Tate curve of $M$} as
$$
\widehat{\K}_{\rm Tate}(M)= \widehat{\K}(M)\llbracket q \rrbracket [q^{-1}],
$$
where differential cocycles on the right are formal sums of cocycles $\sum q^k(V_k,\A_k,\alpha_k)$ in Klonoff's model for $\widehat{\K}(M)$. There is a ring structure on $\widehat{\K}_\Tate(M)$ inherited from $\widehat\K(X)$ and $\Z\llbracket q \rrbracket [q^{-1}]$. The \emph{curvature map} is
$$
R\colon \widehat{\K}_{\rm Tate}(X)\to \Omega^{\ev}_\cl(M)\llbracket q \rrbracket [q^{-1}],\qquad R\left(\sum q^k(V_k,\A_k,\alpha_k)\right)= \sum q^k({\rm Ch}(\A_k)+d\alpha_k).
$$
\end{defn}

To explain our explicit definition for $\widehat{\K}_\MF$ we require a digression on the abstract framework of differential cohomology theories. A differential refinement $\widehat{E}$ of a cohomology theory~$E$ is a homotopy pullback 
\begin{equation}
\begin{array}{c}
\begin{tikzpicture}[node distance=3.5cm,auto]
  \node (A) {$\widehat{E}$};
  \node (B) [node distance= 4cm, right of=A] {$E$};
  \node (C) [node distance = 1.5cm, below of=A] {$\Omega_A$};
  \node (D) [node distance = 1.5cm, below of=B] {${\rm H}_{E_\C}.$};
  \draw[->] (A) to node {$I$} (B);
  \draw[->] (A) to node [swap] {$R$} (C);
  \draw[->] (C) to (D);
  \draw[->] (B) to node {${\rm Ch}$} (D);
\end{tikzpicture}\end{array}\nonumber
\end{equation}
taken, e.g., in sheaves of spectra on the site of smooth manifolds. In the above, $E_\C=E(\pt)\otimes \C$ is the complexification of the coefficient ring of~$E$ and $\Omega_A$ is the de~Rham complex with values in $A$. The map $R$ is called the \emph{curvature map} and the map $I$ takes the underlying cohomology class of a differential cocycle. The map $\Omega_{A}\to {\rm H}_{E_\C}$ comes from a chosen isomorphism~$A\to E_\C$ and the de~Rham map. Finally, ${\rm Ch}$ is the Chern--Dold character. Hopkins--Singer~\cite{HopSing} and Bunke--Gepner~\cite{BunkeGepner} give universal constructions of differential refinements as sketched above. A map into the homotopy pullback is a map into the spectrum~$E$, a differential form with coefficients in~$A$, and the data of a homotopy between the images of these classes in a cochain model for~${\rm H}_{E_\C}$. 

In the example of K-theory, the relevant diagram is 
\begin{equation}
\begin{array}{c}
\begin{tikzpicture}[node distance=3.5cm,auto]
  \node (A) {$\widehat{\K}$};
  \node (B) [node distance= 4cm, right of=A] {$\K$};
  \node (C) [node distance = 1.5cm, below of=A] {$\Omega_{\C[u,u^{-1}]}$};
  \node (D) [node distance = 1.5cm, below of=B] {${\rm H}_{\C[u,u^{-1}]}.$};
  \draw[->] (A) to node {$I$} (B);
  \draw[->] (A) to node [swap] {$R$} (C);
  \draw[->] (C) to (D);
  \draw[->] (B) to node {${\rm Ch}$} (D);
\end{tikzpicture}\end{array}\nonumber
\end{equation}
where~$u$ is a degree $-2$ element. In Klonoff's model, a vector bundle with connection $(V,\A)$ gives a map~$M\to \K^n$, and the curvature of the cocycle $R(V,\A,\alpha)$ gives a closed, even differential form. The Chern character of the bundle and the image of this closed form in cohomology differ by $d\alpha$, and so $\alpha$ itself furnishes a homotopy between them.  

When considering a differential refinement of $\widehat{\K}_\MF$ and the definition of $\KMF$ itself as a homotopy pullback, we get a pair of adjoining homotopy pullback squares
\begin{equation}
\begin{array}{c}
\begin{tikzpicture}[node distance=3.5cm,auto]
  \node (A) {$\widehat{\K}$};
  \node (B) [node distance= 4cm, right of=A] {$\KMF$};
  \node (C) [node distance = 1.5cm, below of=A] {$\Omega_\MF$};
  \node (D) [node distance = 1.5cm, below of=B] {${\rm H}_\MF$};
  \node (E) [node distance = 4cm, right of=B] {$\K_\Tate$};
  \node (F) [node distance =4cm, right of =D] {${\rm H}_{\C\llbracket q \rrbracket [q^{-1}]}$};
  \draw[->] (A) to node {$I$} (B);
  \draw[->] (A) to node [swap] {$R$} (C);
  \draw[->] (C) to (D);
  \draw[->] (B) to (D);
  \draw[->] (B) to (E);
  \draw[->] (D) to (F);
  \draw[->] (E) to (F);
\end{tikzpicture}\end{array}\nonumber
\end{equation}
and so the outer square is also a homotopy pullback. The composition of the lower arrows factors through $\Omega_{\C\llbracket q \rrbracket [q^{-1}]}$ using the (injective) map on differential forms induced by $q$-expansion of modular forms, 
\beq
\bigoplus_{i+j=n} \Omega^{2i}(M; \MF^{2j})\to \Omega^{\ev}(M)\llbracket q \rrbracket [q^{-1}].\label{eq:formqexpansion}
\eeq
By the universal property of the differential spectrum $\widehat{\K}_\Tate$, this gives a map~$\widehat{\K}_\MF(M)\to \widehat{\K}_\Tate(M)$. Lifting a map into $\widehat{\K}_\Tate(M)$ to one into $\widehat{\K}_\MF$ amounts to requiring the curvature factor through~$\Omega_\MF\to \Omega_{\C\llbracket q \rrbracket [q^{-1}]}$ up to homotopy. Fixing the model for differential elliptic cohomology at the Tate curve from Definition~\ref{defn:KTatediff}, this leads to the following.

\begin{defn}\label{defn:KMFdiff}
Define $\widehat{\K}_\MF^{2n}(M)$ to have as differential cocycles pairs $(x,h)$ for $x\in \widehat{\K}_\Tate(M)$ and $h$ a concordance (i.e., smooth homotopy) from the curvature $R(x)\in \Omega^\ev(M)\llbracket q \rrbracket [q^{-1}]$ to an element in the image of~\eqref{eq:formqexpansion}. The curvature map,
$$
R\colon \widehat{\K}_\MF^{2n}(M)\to \bigoplus_{i+j=n} \Omega_{\rm cl}^{2i}(M; \MF^{2j})
$$ 
is the uniquely defined differential form with values in modular forms from the target of the concordance $h$. 
\end{defn}

By adjusting the the sequence of odd differential forms~$\alpha_k$ defining~$x\in \widehat{\K}_\Tate(M)$, we can always replace this data $(x,h)$ by a cocycle $x\in \widehat{\K}_\Tate(M)$ whose curvature factors through~\eqref{eq:formqexpansion} on the nose. The curvature map in this case is 
$$
R\colon \widehat{\K}_\MF^{2n}(M)\to \bigoplus_{i+j=n} \Omega_{\rm cl}^{2i}(M; \MF^{2j}),\qquad R\left(\sum q^k(V_k,\A_k,\alpha_k)\right)=\sum q^k({\rm Ch}(V_k,\A_k)+d\alpha_k),
$$
where the differential form on the right is in the image of~\eqref{eq:formqexpansion}, and so can be identified uniquely with a differential form valued in modular forms.

\bibliographystyle{amsalpha}
\bibliography{references}

\end{document}